\documentclass[a4paper,11pt]{amsart}
\usepackage[utf8]{inputenc}
\usepackage[T1]{fontenc}

\usepackage{color}
\usepackage{bm}
\usepackage{libertine}
\usepackage{microtype}
\usepackage{mathtools}
\usepackage{amsmath}
\usepackage{amsfonts}
\usepackage{amssymb}
\usepackage{mathrsfs}
\usepackage{fancyhdr}

\usepackage{booktabs}
\usepackage[hmargin=2.8cm,vmargin=2cm]{geometry}
\usepackage{mathrsfs}
\usepackage{gensymb}
\usepackage[hidelinks]{hyperref}
\hypersetup{pdfpagemode=UseNone}

\usepackage[abbrev, backrefs]{amsrefs}
\usepackage[capitalize]{cleveref}
\usepackage[version=4]{mhchem}

\numberwithin{equation}{section}
\newtheorem{proposition}{Proposition}[section]
\newtheorem{theorem}[proposition]{Theorem}
\newtheorem{lemma}[proposition]{Lemma}
\newtheorem{corollary}[proposition]{Corollary}
\theoremstyle{definition}
\newtheorem{definition}[proposition]{Definition}
\theoremstyle{remark}
\newtheorem{remark}[proposition]{Remark}

\usepackage{paralist}
\usepackage{constants}
\usepackage{esint}

\title[Ginzburg--Landau relaxation for harmonic maps into a general manifold]{Ginzburg--Landau relaxation for harmonic maps on planar domains into a general compact vacuum manifold}

\author{Antonin Monteil}
\address{
Antonin Monteil\\
 University of Bristol\\
 School of Mathematics\\
 Fry building\\
 Woodland Road\\
 BS8 1UG Bristol, United Kingdom
}
\email{antonin.monteil@bristol.ac.uk}

\author{R\'emy Rodiac}
\address{
Rémy Rodiac\\
Université Paris-Saclay, CNRS\\
Laboratoire de Mathématiques d'Orsay\\
91405, Orsay, France\\
}
\email{remy.rodiac@universite-paris-saclay.fr}

\author{Jean Van Schaftingen}
\address{
Jean Van Schaftingen\\
 Université catholique de Louvain\\
 Institut de Recherche en Mathématique et Physique\\
 Chemin du cyclotron 2, L7.01.02\\
1348 Louvain-la-Neuve, Belgium}
\email{jean.vanschaftingen@uclouvain.be}

\thanks{A. Monteil was supported as postdoctoral researcher (charg\'e de recherches) by the Fonds de la Recherche Scientifique--FNRS; R. Rodiac and J. Van Schaftingen were supported by the Mandat d'Impulsion Scientifique F.4523.17, ``Topological singularities of Sobolev maps'' of the Fonds de la Recherche Scientifique--FNRS;
J. Van Schaftingen was supported by the Projet de Recherche T.0229.21 du ``Singular Harmonic Maps and Asymptotics of Ginzburg-Landau Relaxations'' of the Fonds de la Recherche Scientifique–FNRS.}

\usepackage{mathtools}
\newcommand{\st}{\;:\;}

\newcommand{\abs}[1]{\lvert #1 \rvert}
\newcommand{\norm}[1]{\lVert #1 \rVert}

\newcommand{\bigabs}[1]{\bigl\lvert #1 \bigr\rvert}

\newcommand{\weakto}{\rightharpoonup}

\newcommand{\Rset}{\mathbb{R}}

\newcommand{\Nset}{\mathbb{N}}

\newcommand{\Sset}{\mathbb{S}}
\newcommand{\Hset}{\mathbb{H}}

\newcommand{\Deriv}{D}
\newcommand{\compose}{\,\circ\,}
\newcommand{\dif}{\,\mathrm{d}}
\newcommand{\VMO}{\mathrm{VMO}}
\newcommand{\manifold}[1]{\mathcal{#1}}

\newcommand{\defeq}{\coloneqq}
\newcommand{\eqdef}{\eqqcolon}

\DeclareMathOperator{\diam}{diam}

\newcommand{\Conf}[1]{\operatorname{Conf}_{#1}}
\newcommand{\equivnorm}[1]{\lambda(#1)}

\newcommand{\Esing}{\mathcal{E}^{\mathrm{sg}}}

\newcommand{\synhar}[2]{d_{\mathrm{synh}} (#1, #2)}

\newcommand{\e}{\varepsilon}

\DeclareMathOperator{\dist}{dist}
\DeclareMathOperator{\id}{id}
\DeclareMathOperator{\Lin}{Lin}

\DeclareMathOperator{\tr}{tr}

\DeclareMathOperator{\syst}{sys}
\date{\today}

\begin{document}

\begin{abstract}
We study the asymptotic behaviour, as a small parameter $\varepsilon$ tends to zero, of minimisers of a Ginzburg--Landau type energy with a nonlinear penalisation potential vanishing on a compact submanifold $\mathcal{N}$ and with a given $\mathcal{N}$--valued Dirichlet boundary datum.
We show that minimisers converge up to a subsequence to a singular $\mathcal{N}$--valued harmonic map, which is smooth outside a finite number of points around which the energy concentrates
and whose singularities' location minimises a renormalised energy,
generalising known results by Bethuel, Brezis and H\'elein for the circle $\mathbb{S}^1$.
We also obtain $\Gamma$--convergence results  and uniform Marcinkiewicz weak $L^2$ or Lorentz $L^2$ estimates on the derivatives.
We prove that solutions to the corresponding Euler--Lagrange equation converge uniformly to the constraint and converge to harmonic maps away from singularities.
\end{abstract}
\maketitle
\tableofcontents

\section{Introduction}

Given a smooth compact connected manifold \(\manifold{N}\) which can be assumed, thanks to Nash's embedding theorem \cite{Nash_1956}, to be isometrically embedded into the Euclidean space \(\Rset^\nu\) for some \(\nu \in \mathbb{N}_\ast\), given a bounded domain \(\Omega \subset \Rset^2\) with Lipschitz boundary and given \(g\in W^{1/2,2}(\partial \Omega, \manifold{N})\), a \emph{minimising harmonic map} \(u : \Omega \to \manifold{N}\) with boundary condition \(g\) is a map which minimises the \emph{Dirichlet energy}
\begin{equation}
  \label{eq_Dirichlet_energy}
\int_\Omega \frac{\abs{\Deriv u}^2}{2}
\end{equation}
on the nonlinear subspace 
\begin{equation}
  W^{1,2}_g (\Omega,\manifold{N})\defeq \{u \in W^{1,2}(\Omega,\Rset^\nu) \st u \in \manifold{N} \text{ almost everywhere 
      in } \Omega \text{ and }  \tr_{\partial \Omega}u=g  \}
\end{equation}
of the \emph{Sobolev space} \(W^{1, 2} (\Omega, \Rset^\nu)\) of functions having a square-summable weak derivative, where \(\tr_{\partial \Omega}\) denotes the trace operator on \(W^{1, 2} (\Omega, \Rset^\nu)\).
It is known since Morrey's work that, when the domain \(\Omega\) is two-dimensional, any minimising harmonic map is smooth \cite{Morrey_1948}.

Because of topological obstructions, the space \(W^{1, 2}_g (\Omega, \manifold{N})\) can happen to be empty; if \(g \in \mathcal{C} (\partial \Omega, \manifold{N})\), this will be the case if and only if the map \(g\) cannot be extended to a continuous map from \(\Omega\) to \(\manifold{N}\) (see \cite{Schoen_Uhlenbeck_1982}).
This occurs for example when the domain \(\Omega\) is simply-connected while the manifold \(\manifold{N}\) is not simply-connected and the map \(g\) is not homotopic to a constant map.

The \emph{Ginzburg--Landau relaxation strategy} consists in replacing the constraint that \(u \in \manifold{N}\) almost everywhere in \(\Omega\) by an additional penalisation term to the Dirichlet energy \eqref{eq_Dirichlet_energy}.
Fixing a nonnegative function \(F \in \mathcal{C} (\Rset^\nu, [0, +\infty))\) such that
\(F^{-1} (\{0\}) = \manifold{N}\), one defines for every \(\varepsilon \in (0,+\infty)\), the \emph{Ginzburg--Landau energy} as 
\begin{equation}
\label{eq_GLenergy}
\mathcal{E}^\varepsilon_F (u)\defeq \int_\Omega \frac{\abs{\Deriv u}^2}{2}+\frac{F(u)}{\varepsilon^2}.
\end{equation}
In the present work, we will require the function \(F\) to satisfy the following non-degeneracy condition:
\begin{equation}
  \label{hyp20}
        \begin{gathered}
    \text{there exist \(\delta_F,m_F,M_F \in (0,+\infty)\)  such that for every  \(z\in\Rset^\nu\) with  \(\dist(z,\manifold{N})<\delta_F\)}, \\
    \frac{m_F}{2}\dist(z,\manifold{N})^2 \leq F(z)\leq \frac{M_F}{2}\dist(z,\manifold{N})^2.
      \end{gathered}
\end{equation}

The existence of minimisers \(u_\varepsilon\) of \(\mathcal{E}^\varepsilon_F\) under the Dirichlet boundary condition \(\operatorname{tr}_{\partial\Omega}u_\varepsilon =g\) follows from a classical result in the direct method of calculus of variations (see for example \citelist{\cite{Dacorogna_2008}*{Corollary 3.24}}).
When \(\varepsilon \to 0\), one expects the function \(u_\varepsilon\) to eventually take its values into the manifold \(\manifold{N}\) except in some small singular regions; the limiting map can then play a role of generalised solution of the Dirichlet problem for harmonic maps into the manifold \(\manifold{N}\).

Our first main result (\cref{theorem_main1}) describes this asymptotic behaviour of minimisers of the Ginzburg--Landau energy when \(\varepsilon \to 0\): if for each \(\varepsilon > 0\), the function \(u_\varepsilon\) is a minimiser of \(\mathcal{E}^\varepsilon_F\) under the boundary condition \(\tr_{\partial \Omega} u_\varepsilon = g\), then there exists a sequence \( (\varepsilon_n)_{n\in \mathbb{N}}\) converging to \(0\), a finite set of points \(\{a_1,\dotsc,a_k\} \subset \Omega\) and a map \(u_*\in W^{1,2}_{\textrm{loc}} (\Bar{\Omega} \setminus \{a_1,\dotsc,a_k \})\) such that
\(u_{\varepsilon_n}\rightarrow u_*\) strongly in \(W^{1,2}_{\textrm{loc}}(\Bar{\Omega} \setminus \{a_1,\dotsc,a_k \})\), where 
\(u_*\) is an \(\manifold{N}\)--valued harmonic map in \(\Omega \setminus \{a_1,\dotsc,a_k \}\)
and the configuration of points \(\{a_1, \dotsc, a_k\}\) minimises a \emph{renormalised energy}.
This renormalised energy is defined as the sum of a renormalised energy for harmonic maps that we have defined in \cite{Monteil_Rodiac_VanSchaftingen_RE} and that we present in \S \ref{section_renormalised_energies}, and a term defined in \S \ref{section_minimal_energy_on_balls} depending on the singularities and on the penalisation nonlinearity \(F\).

When \(\manifold{N}=\mathbb{S}^1\subset \Rset^2\)and \(F(z)=(1-\abs{z}^2)^2\),
we recover the seminal results of Bethuel, Brezis \& Hélein \cite{Bethuel_Brezis_Helein_1994},
for the original Ginzburg--Landau functional used to model the behaviour of \emph{type II superconductors} for a star-shaped domain \(\Omega\); the results were later extended to simply-connected domains in \cite{Struwe_1994}; here we do not assume that \(\Omega\) is simply-connected and provide thus new results for the original Ginzburg--Landau functional in the multiply-connected case. 
In the case of a general target manifold \(\manifold{N}\), the leading-order asymptotics and the topological charges of singularities in our results (\cref{theorem_main1} \eqref{it_ooChifooNg2zeequohsahsoo} and \eqref{it_eigh6wav2ieS3yahThaipho8} at the \(o (\log 1/\varepsilon)\) level) are due to Canevari \cite{Canevari_2015}. 

Functionals of the form \eqref{eq_GLenergy} appear in various other physical models besides the Ginzburg--Landau model in superconductivity.
The \emph{Landau--de Gennes theory} describes the state of a \emph{nematic liquid crystal} via a field of symmetric traceless \(3\times 3\) matrices which minimises an energy of the form \eqref{eq_GLenergy} with \(\manifold{N} \simeq \Rset \mathbb{P}^2\); the study of such minimisers has been the object of many works \cites{Bauman_Park_Phillips_2012,Golovaty_Montero_2014,Canevari_2015,Ball_Zarnescu_2011}.
Energies of the form \eqref{eq_GLenergy} also appear in physics in \emph{Chern-Simon-Higgs theory} \cite{Bauman_Park_Phillips_2012} with \(\manifold{N}=\mathbb{S}^1\times \{0\}\simeq \mathbb{S}^1 \) and other phase transitions problems like \emph{biaxial molecules in nematic phase} (\(\manifold{N}\simeq SU(2)/Q\), where \(Q\) is the quaternion group), \emph{superfluid \ce{^3He} in dipole-free phase} with \( \manifold{N} \simeq SU(2)\times SU(2)/H \) where \(H\) is a subgroup of \(SU(2) \times SU(2)\) isomorphic to four copies of \(\mathbb{S}^1\) and \emph{superfluid \ce{^3He} in dipole-locked phase} with \( \manifold{N} \simeq \Rset \mathbb{P}^3\) \cite{Mermin_1979}.

Minimisation of Ginzburg--Landau type energies also appears as a strategy in meshing algorithms for numerical analysis and computer graphics:
in order to generate a quadrangular meshing of a surface or a hexahedral meshing of a three-dimensional domain, one constructs first a guiding cross-field or frame-field which is mathematically a map taking its values into \(SO(2)/C_4\) and \(SO(3)/O\), where \(C_4\) is the cyclic group of order \(4\) of   direct symmetries of a square, and \(O\) is the octahedral group of direct symmetries of the cube \cites{Liu_Zhang_Chien_Solomon_Bommes_2018, HTWB,Beaufort_Lambrechts_Henrotte_Geuzaine_Remacle_2017,Viertel_Osting_2017,Chemin_Henrotte_Remacle_VanSchaftingen}.
Mathematically, in the latter case \(\pi_1 (SO (3)/O) = 2O\) is the \emph{nonabelian} binary octahedral group. Since one would like these cross-fields or frame-fields to minimise a Dirichlet energy and since one can face topological obstructions as described earlier in this introduction, the strategy consists in constructing these fields using a Ginzburg--Landau relaxation. The cross-fields and frame-fields will necessarily have singularities and one hopes to place these singularities in an optimal way using this procedure.

The asymptotics that we obtain imply in particular that when the domain \(\Omega\) is a disk and the boundary datum \(g\) is an atomic minimising geodesic in \(\manifold{N}\) (see \S \ref{section_singular_energy}), then the asymptotic profile is of the form \(u_* (x) = g (x/\abs{x})\) (\cref{theorem_explicit}). This generalises the answer of Bethuel, Brezis \& H\'elein to Matano's original problem on the Ginzburg--Landau equation \cite{Elliott_Matano_Tang_1994}.

As another consequence of our results, the \emph{stress-energy tensor} of the limit \(u_*\) has vanishing flux around the singularities --- equivalently, the residue of the \emph{Hopf differential} of \(u_*\) vanishes at each singularity.

The results presented above are not confined to minimisers of the Ginzburg--Landau energy, 
and imply in particular \(\Gamma\)--convergence results at first and second order similar to the classical case, generalising \(\Gamma\)--convergence results at first order
\citelist{\cite{Jerrard_1999}\cite{Jerrard_Soner_2002}\cite{Sandier_1998}} and \(\Gamma\)--convergence results at second order  \cite{Alicandro_Ponsiglione_2014}.
All our results also come with Marcinkiewicz weak \(L^2\) estimates --- or equivalently estimates in the endpoint Lorentz space \(L^{2, \infty}\) --- on the gradient as for the original Ginzburg--Landau functional \cite{Serfaty_Tice_2008}.

We consider next the improvements in the asymptotics that can be obtained when \(u_\varepsilon\) is a weak solution to the equation
\begin{equation}
  \label{eq_GinzburgLandau}
  \Delta u_\varepsilon = \frac{\nabla F (u_\varepsilon)}{\varepsilon^2} \qquad \text{in \(\Omega\)},
\end{equation}
that we refer to \eqref{eq_GinzburgLandau} as the \emph{generalised Ginzburg--Landau equation}.
Minimisers of the Ginzburg--Landau energy \(\mathcal{E}^\varepsilon_F\) 
satisfy the corresponding Euler--Lagrange equation, i.e., \eqref{eq_GinzburgLandau} is satisfied under reasonable assumptions (see \S \ref{section_Euler_Lagrange}).
We prove in \cref{proposition_uniform_convergence_N} that under a boundedness assumption on \(\nabla F (u_\varepsilon)\), the distance to the manifold
\(\dist (u_{\varepsilon_n}, \manifold{N})\) \emph{converges uniformly to \(0\)} up to the boundary and away from singularities for any boundary datum \(g \in W^{1/2, 2} (\partial \Omega, \manifold{N})\) --- which is not continuous in general.
We next prove in \cref{proposition_convergence_solutions} that weakly converging solutions of \eqref{eq_GinzburgLandau} converge to harmonic maps.
Finally, we obtain higher-order convergence up to the boundary under a higher regularity assumption on the boundary datum (\cref{prop_strongconvC0}). 

Another strategy to study phase-transition problems where one deals with manifold-valued order-parameters has been implemented in \cites{Canevari_Orlandi_2020a,Canevari_Orlandi_2020b} by constructing a substitute to the Jacobian determinant used in the classical \(\mathbb{S}^1\)-valued Ginzburg-Landau theory to obtain first-order \(\Gamma\)--convergence results; this substitute is obtained by using flat chains in the setting of manifolds with abelian fundamental groups. Other types of topological obstructions have been analysed via a Ginzburg--Landau relaxation in the case of two-dimensional Riemannian manifolds \cites{Ignat_Jerrard_2017,Ignat_Jerrard_2020}: 
the authors prove the convergence of vector fields minimising some Ginzburg--Landau type energy to a canonical unit-length harmonic tangent field with a finite number of singularities; the singularities arise from a non-vanishing Euler-Poincaré characteristic, their number is determined by the Poincar\'e--Hopf index theorem and their position is governed by a renormalised energy.

\medskip 

We continue the present work with a preliminary section on the projection onto the manifold and on non-degeneracy conditions on \(F\) (\S \ref{section_preliminaries}).
We next recall in \S \ref{section_renormalised_energies} the definitions and properties of singular energy, geometric renormalised energy, renormalisable singular mappings and synharmony from \cite{Monteil_Rodiac_VanSchaftingen_RE}.
In \S \ref{section_minimal_energy_on_balls}, we introduce a quantity measuring the energy of a vortex with a given boundary condition at infinity.
We combine then the different tools to obtain an \emph{upper bound} on the energy of minimisers in \S \ref{section_upper_bound}.

In \S \ref{sect:Lower_bounds}, we obtain by Sandier's vortex-ball method \cite{Sandier_1998} a first lower-bound on the energy and then following Jerrard's strategy \cite{Jerrard_1999} we obtain localised estimates.
We apply then these estimates to energy convergence results, implying convergence of minimisers and \(\Gamma\)--convergence results (\S \ref{section_energy_convergence}). 
We also explain how our results locate singularities on a disk with an atomic minimising geodesic as boundary datum (\S \ref{section_simple_boundary_conditions}).

In the last section \S \ref{section_solutions} we give sufficient conditions for minimisers to be solutions of the Ginzburg--Landau equation. Then we study solutions to this equation and we prove uniform convergence of these solutions to the constraint manifold \(\manifold{N}\), weak convergence to harmonic maps and higher-order convergence away from singularities.

\section{Retraction on the manifold and non-degeneracy of the relaxation potential}
\label{section_preliminaries}

\subsection{Embedding and nearest point retraction}

The Ginzburg--Landau relaxation procedure requires an isometric embedding of the vacuum manifold \(\manifold{N}\) into \(\Rset^\nu\).
The classical Nash embedding theorem \cite{Nash_1956} provides such an embedding.
When \(\manifold{N} = G/H\) where \(G\) is a Lie group and \(H \subset G\) is a closed subgroup, it can be relevant to use an \emph{equivariant isometric embedding}
due to Moore \cite{Moore_1976} (see also \cite{Moore_Schlafly_1980}): there exists an isometric embedding \(\Psi : G/H \to \Rset^\nu\) and a representation  \(R : G \to \Lin (\Rset^\nu)\) such that for every \(g \in G\) and \(y \in G/H\), \(\Psi (gy) = R (g) (\Psi (y))\); in contrast with Nash's embedding theorem, the dimension \(\nu\) of the target space \(\Rset^\nu\) depends on the metric on \(G\) and on the choice of the subgroup \(H\), and the compactness of \(G/H\) is essential (there is no such embedding if \(G/H\) is the hyperbolic space \(\Hset^n \simeq O (1, n)/ O(1) \times O(n)\)).

We define the function \(\dist_{\manifold{N}} : \Rset^\nu \to [0,+\infty)\) by setting for each \(y \in \Rset^\nu\),
\[
  \dist_{\manifold{N}} (y) \defeq \dist (y, \manifold{N})
  \defeq \inf\,\bigl\{\abs{y - z} \st z \in \manifold{N}\bigr\}
\]
and for each \(\delta \in (0, +\infty)\), the set
\[
 \manifold{N}_{\delta}
  \defeq
  \bigl\{ y \in \Rset^\nu \st \dist(y,\manifold{N})<\delta \bigr\}.
\]
The next lemma describes the nearest point retraction of a neighbourhood of \(\manifold{N}\) on \(\manifold{N}\).

\begin{lemma}
  \label{lemma_derivativeNearestPointRetraction}
  There exists \(\delta_{\manifold{N}} >0\) such that the nearest point retraction \(\Pi_{\manifold{N}} : \manifold{N}_{\delta_{\manifold{N}}} \to \manifold{N}\) characterized by
  \begin{equation*}
  \abs{y-\Pi_{\manifold{N}}(y)}=\dist(y,\manifold{N})
  \end{equation*}
is well-defined and smooth. Moreover, if the mappings \(P_{\manifold{N}}^\top : \manifold{N} \to \operatorname{Lin} (\Rset^\nu, \Rset^\nu)\) and \(P_{\manifold{N}}^\perp : \manifold{N} \to \operatorname{Lin} (\Rset^\nu, \Rset^\nu)\) are defined for each \(y \in \manifold{N}\) by setting \(P_{\manifold{N}}^\top (y)\) and \(P_{\manifold{N}}^\perp (y)\) as the orthogonal projections on \(T_y \manifold{N}\) and \((T_y\manifold{N})^\perp\), identified as linear subspaces of \(\Rset^\nu\), then for every \(y\in\manifold{N}_{\delta_\manifold{N}}\) and \(v\in\Rset^\nu\),
\begin{equation}
\label{firstProjectionEstimate}
\abs{\Deriv \dist_{\manifold{N}} (y) [v]}^2 \le \abs{P_{\manifold{N}}^\perp (\Pi_{\manifold{N}} (y)) [v]}^2
\end{equation}
and
\begin{equation}
\label{secondProjectionEstimate}
\Bigl(1-\frac{\dist_{\manifold{N}} (y)}{\delta_\manifold{N}}\Bigr)\abs{\Deriv\Pi_{\manifold{N}}(y)[v]}^2\le \abs{P_{\manifold{N}}^\top (\Pi_{\manifold{N}} (y))[v]}^2
\le C\abs{\Deriv\Pi_{\manifold{N}}(y)[v]}^2,
\end{equation}
for some constant \(C\in(0,+\infty)\) depending on \(\manifold{N}\) and \(\nu\) only.
\end{lemma}

In the particular case of the sphere \(\manifold{N} = \Sset^n \subseteq \Rset^{n + 1}\), one has \(\Pi_{\manifold{N}}(y)= y/\abs{y}\) if \(y \in \Rset^{n+1}\setminus \{0\}\), \(\Deriv \Pi_{\manifold{N}}(y)[v] = (v\abs{y}^2 - y (y \cdot v))/\abs{y}^3\), and thus \(\abs{\Deriv \Pi_{\manifold{N}}(y)[v]}^2 = \abs{v}^2/\abs{y}^2 - (y \cdot v)^2/\abs{y}^4\) for \(v\in \Rset^{n+1}\). Moreover \(\dist_{\Sset^n}(y)=\abs{\abs{y}-1}\) and \(\abs{\Deriv\dist_{\Sset^n}(y)[v]}=\abs{v \cdot y}/\abs{y}\) for \(y \in \Rset^{n+1}\setminus \{0\}\) and \(v\in \Rset^{n+1}\).
Besides, if \(z\in \Sset^n\) and \(v \in \Rset^{n+1}\): \(P_{\Sset^n}^\perp (z)[v] = z (z \cdot v)\) and \(P_{\Sset^n}^\top (z)[v] = v - z (z \cdot v)\), so that in this case \cref{lemma_derivativeNearestPointRetraction} is a consequence of  the formulae
\begin{align*}
\abs{\Deriv \dist_{\Sset^n} (y)[v]}^2
&= \abs{P_{\Sset^n}^\perp (\Pi_{\Sset^n}(y)) [v]}^2
&\text{ and } &
&\abs{y}^2 \abs{\Deriv \Pi_{\manifold{N}}(y)[v]}^2 = \abs{P_{\Sset^n}^\top (\Pi_{\Sset^n} (y)) [v]}^2,
\end{align*}
for \(y \in \Rset^{n+1}\setminus \{0\}\) and \(v \in \Rset^{n+1}\).

The smoothness of the nearest point retraction is classical \cite{Foote_1984}.
 For related computations on the distance function to embedded manifolds, we refer the reader to \citelist{\cite{Ambrosio_Mantegazza_1998}\cite{Eminenti_Mantegazza_2004}}. 
For every \(y \in \manifold{N}_{\delta_\manifold{N}}\) and \(v \in \Rset^\nu\), we have 
by orthogonality \(\abs{P_{\manifold{N}}^\perp (\Pi_{\manifold{N}} (y)) [v]}^2 +  \abs{P_{\manifold{N}}^\top (\Pi_{\manifold{N}} (y))[v]}^2 = \abs{v}^2\),
and thus by \cref{lemma_derivativeNearestPointRetraction} 
\begin{equation}
\label{eq_xoishee4IuHi8eefahyeithe}
    \abs{\Deriv \dist_{\manifold{N}} (y) [v]}^2+\Bigl(1-\frac{\dist_{\manifold{N}} (y)}{\delta_\manifold{N}}\Bigr)\abs{\Deriv\Pi_{\manifold{N}}(y)[v]}^2\le \abs{v}^2.
\end{equation}

In the proof of \cref{lemma_derivativeNearestPointRetraction} and throughout this work we will use the following facts about the nearest point projection:
\begin{equation}\label{eq:caracterisation_ortho}
\text{for all } y \in \manifold{N}_{\delta_\manifold{N}}, \quad  y-\Pi_\manifold{N}(y) \in (T_{\Pi_\manifold{N}(y)}\manifold{N})^\perp,
\end{equation}
\begin{equation}\label{eq:proj_ortho}
\text{ for all } y \in \manifold{N}, \quad \Deriv\Pi_\manifold{N}(y) \text{ is the orthogonal projection onto } T_y\manifold{N} \text{ i.e., } \Deriv\Pi_\manifold{N}(y)=P^\top_\manifold{N}(y),
\end{equation}
\begin{multline}\label{eq:second_fund_form}
\text{ for all } y \in \manifold{N}, \quad -\Deriv^2 \Pi_\manifold{N}(y): T_y\manifold{N} \otimes T_y\manifold{N} \rightarrow (T_y\manifold{N})^\perp  \\
\text{ is the second fundamental form of } \manifold{N}\subset \Rset^\nu \text{ at } y .
\end{multline}
Point \eqref{eq:caracterisation_ortho} follows from the characterization of the map \(\Pi_\manifold{N}\). For \eqref{eq:proj_ortho} we refer to \cite{Moser_2005}*{Lemma 3.1}. We denote by \(B_x: T_x\manifold{N} \otimes T_x\manifold{N} \rightarrow (T_x \manifold{N})^\perp\) the second fundamental form of \(\manifold{N}\) at \(x \in \manifold{N}\) and we refer to \cite{Do_Carmo_1992}*{definition 6.2.2} for the definition. We observe that, for \(y \in \manifold{N}\), \( \Deriv^2\Pi_\manifold{N}(y)_{\vert T_y \manifold{N} \otimes T_y\manifold{N} }=\Deriv P^\top_\manifold{N}(y)_{\vert T_y \manifold{N} \otimes T_y\manifold{N} }\) and we refer to \cite{Moser_2005}*{Lemma 3.2} for \eqref{eq:second_fund_form}.

\begin{proof}[Proof of \cref{lemma_derivativeNearestPointRetraction}]
It is well-known that when \(\delta > 0\) is small enough, the nearest point retraction \(\Pi_{\manifold{N}}\) is well-defined on \(\manifold{N}_{\delta}\). For every \(y \in \manifold{N}_{\delta}\), by using \eqref{eq:caracterisation_ortho} we find
  \[
  P_{\manifold{N}}^\top (\Pi_{\manifold{N}}(y))[\Pi_{\manifold{N}}(y) - y] = 0.
  \]  
Differentiating this identity with respect to \( y\) by using the chain rule and the Leibniz rule, we find for every \(y \in \manifold{N}_{\delta}\) and \(v \in \Rset^\nu\),
  \begin{equation}
  \label{eq_aePetaeghoow7hux2AoP2bae}
  P_{\manifold{N}}^\top (\Pi_{\manifold{N}}(y))\bigl[\Deriv \Pi_{\manifold{N}}(y)[v] - v\bigr]
    +(\Deriv P_{\manifold{N}}^\top (\Pi_{\manifold{N}}(y)) [\Deriv \Pi_{\manifold{N}}(y)[v]]) \bigl[\Pi_{\manifold{N}}(y) - y\bigr]=0.
  \end{equation}
Noting that \(\Deriv \Pi_{\manifold{N}}(y)[v] \in T_{\Pi_{\manifold{N}}(y)}\manifold{N}\), that for every \(z\in\manifold{N}\), \(P_{\manifold{N}}^\top(z)+P_{\manifold{N}}^\perp(z) = \operatorname{id}\) so that \(\Deriv P_{\manifold{N}}^\top(z)[w]=-\Deriv P_{\manifold{N}}^\perp(z)[w]\) whenever \(w\in T_z\manifold{N}\), we infer from \eqref{eq_aePetaeghoow7hux2AoP2bae} that 
  \begin{equation}
    \label{identity_D_Pi}
    P_{\manifold{N}}^\top (\Pi_{\manifold{N}}(y))\bigl[\Deriv \Pi_{\manifold{N}}(y)[v] - v\bigr]
    - (\Deriv P_{\manifold{N}}^\perp (\Pi_{\manifold{N}}(y)) [\Deriv \Pi_{\manifold{N}}(y)[v]]) \bigl[\Pi_{\manifold{N}}(y) - y\bigr]=0.
  \end{equation}
  We observe that for every \(w \in \Rset^\nu\), \(x\in\manifold{N} \mapsto P_{\manifold{N}}^\perp (x)[w]\in T_x^\perp \manifold{N}\) is a smooth map, and therefore we have \cite{Do_Carmo_1992}*{proposition 6.2.3} if \(x\in\manifold{N}\), \(w, z \in T_x \manifold{N}\) and \(u \in (T_x \manifold{N})^\perp\),
  \begin{equation}
  \label{eq_booxi7chiecoozaelu6Aep1T}
  z \cdot (\Deriv P_{\manifold{N}}^\perp (x)[w])[u]
  = -u\cdot B_{x} (z, w)
  .
  \end{equation}
Moreover, since for every \(y\in\manifold{N}_\delta\), \(v\in\Rset^\nu\), \(\Deriv\Pi_{\manifold{N}}(y)[v] \in T_{\Pi_\manifold{N} (y)} \manifold{N}\), we have
\begin{equation}\label{eq_heiyahchiePaibeejej3Phod}
  \Deriv \Pi_{\manifold{N}}(y)[v] \cdot P_{\manifold{N}}^\top (\Pi_{\manifold{N}}(y))[\Deriv \Pi_{\manifold{N}}(y)[v] - v]
  = \abs{\Deriv \Pi_{\manifold{N}}(y)[v]}^2 -  \Deriv \Pi_{\manifold{N}}(y)[v] \cdot P_{\manifold{N}}^\top (\Pi_{\manifold{N}}(y))[v] .
  \end{equation}
Therefore, we have, by taking the inner product of \eqref{identity_D_Pi} with the vector \(\Deriv\Pi_{\manifold{N}}(y)[v]\), in view of \eqref{eq_booxi7chiecoozaelu6Aep1T} and \eqref{eq_heiyahchiePaibeejej3Phod}
  \begin{multline}
    \label{eq_Quuech0cah2Aizior9yietah}
  \abs{\Deriv \Pi_{\manifold{N}}(y)[v]}^2 +
  (\Pi_{\manifold{N}}(y) - y)\cdot B_{\Pi_{\manifold{N}}(y)} [\Deriv \Pi_{\manifold{N}}(y)[v], \Deriv \Pi_{\manifold{N}}(y)[v]]\\ 
  = P_{\manifold{N}}^\top (\Pi_{\manifold{N}}(y))[v] \cdot \Deriv \Pi_{\manifold{N}}(y) [v].
\end{multline}

Hence, if \(\delta_{\manifold{N}}\in(0,\delta)\) satisfies \(\frac{1}{\delta_\manifold{N}}\ge\sup \{\abs{B_y (z, w)} \st y \in \manifold{N}, z, w \in T_y \manifold{N}, \abs{z} \le 1,\, \abs{w} \le 1 \}\), we have for every \(y\in\manifold{N}_{\delta_\manifold{N}}\) and \(v\in\Rset^\nu\),
  \begin{equation}
    \label{eq_ong6ose5ohphuH5pa}
  \Bigl(1 - \frac{1}{\delta_\manifold{N}}\abs{\Pi_{\manifold{N}}(y) - y}\Bigr) \abs{\Deriv \Pi_{\manifold{N}}(y)[v]}
  \le \abs{P_{\manifold{N}}^\top (\Pi_{\manifold{N}}(y))[v]},
\end{equation}
which is the first inequality in \eqref{secondProjectionEstimate}. In particular, \(\ker P_\manifold{N}^\top(\Pi_{\manifold{N}}(y)) \subset\ker \Deriv\Pi_{\manifold{N}}(y) \) and moreover \(\ker \Deriv\Pi_{\manifold{N}}(y) = \ker P_{\manifold{N}}^\top (\Pi_{\manifold{N}}(y))\) since \(\Deriv\Pi_{\manifold{N}}(y)\) and \( P_\manifold{N}^\top(\Pi_{\manifold{N}}(y))\) are onto from \(\Rset^\nu\) to \(T_{\Pi_\manifold{N}(y)}\manifold{N}\). This yields the second inequality in \eqref{secondProjectionEstimate}.

The first estimate \eqref{firstProjectionEstimate}, follows from the fact that for every \(y \in \manifold{N}_{\delta_{\manifold{N}}} \setminus \manifold{N}\) and \(v \in \Rset^\nu\)
  \begin{equation*}
  \Deriv \dist_{\manifold{N}} (y) [v] = \frac{v \cdot (y - \Pi_{\manifold{N}}(y))}{\abs{y - \Pi_{\manifold{N}}(y)}}
  = \frac{P_{\manifold{N}}^\perp (\Pi_{\manifold{N}}(y))[v] \cdot (y - \Pi_{\manifold{N}}(y))}{\abs{y - \Pi_{\manifold{N}}(y)}}.\qedhere
\end{equation*}
\end{proof}

\subsection{Non-degeneracy of the penalising potential}
\label{subs:behavior_energy_density}
If the function \(F\) satisfies the following first order non-degeneracy condition,
\begin{equation}
  \label{hyp21}
        \begin{gathered}
    \text{\(F \in \mathcal{C}^1(\Rset^\nu,[0,+\infty))\) and there exist \(\delta_F\in(0,\delta_{\manifold{N}})\) and \(m_F,M_F \in (0,+\infty)\) such that} \\
m_F\dist(z,\manifold{N})^2\leq DF(z)[z-\Pi_{\manifold{N}}(z)]\leq M_F \dist(z,\manifold{N})^2\quad\text{for every \(z\in\manifold{N}_{\delta_F}\)},
      \end{gathered}
\end{equation}
then it satisfies our zero order non-degeneracy assumption \eqref{hyp20}. This fact will be useful in \cref{higherOrderSection}.
\begin{lemma}
\label{nonDegeneracyZeroOne}
If \(F \in \mathcal{C}^1(\Rset^\nu , [0, +\infty))\) with \(F=0\) on \(\manifold{N}\) and if \eqref{hyp21} holds, then \eqref{hyp20} holds.
\end{lemma}
\begin{proof}
By the assumption \eqref{hyp21}, we have for every \(z\in\manifold{N}_{\delta_F}\) and \(t\in[0,1]\),
\[
m_F\, t\dist(z,\manifold{N})^2\leq \Deriv F((1-t)\Pi_{\manifold{N}}(z)+tz)[z-\Pi_{\manifold{N}}(z)]\leq M_F\, t\dist(z,\manifold{N})^2;
\]
the conclusion follows by integration with respect to \(t\) over \([0,1]\) since \(F=0\) on \(\manifold{N}\).
\end{proof}

A more explicit condition on \(F\) that implies \eqref{hyp21} is given by the second order condition:
\begin{equation}
  \label{hyp22}
  \text{\(F \in \mathcal{C}^2(\Rset^\nu,[0,+\infty))\) and for every \(y\in\manifold{N}\) and \(v \in (T_y\manifold{N})^\perp \setminus \{0 \}\)}, \ \Deriv^2F(y)[v,v] >0.
\end{equation}
\begin{lemma}
\label{lemma_compFanddist}
If \(F \in \mathcal{C}^2 (\Rset^\nu, [0, +\infty))\) with \(F=0\) on \(\manifold{N}\) and if \eqref{hyp22} holds, then \eqref{hyp21} holds.
\end{lemma}

\begin{proof}
By compactness of \(\manifold{N}\), by continuity of \(\Deriv^2 F\) and by \eqref{hyp22}, there exist \(\delta_F\in (0,\delta_{\manifold{N}})\) and \(m_F,M_F\in(0,+\infty)\) such that for every \(z\in\manifold{N}_{\delta_F}\) and \(v\in (T_{\Pi_{\manifold{N}}(z)}\manifold{N})^\perp\),
\[
m_F\abs{v}^2\leq \Deriv^2F(z)[v,v]\leq M_F\abs{v}^2.
\] 
In particular, since \(z-\Pi_{\manifold{N}}(z)\in (T_{\Pi_{\manifold{N}}(z)}\manifold{N})^\perp\), we have for every \(t\in[0,1]\),
\[
m_F\dist(z,\manifold{N})^2\leq\Deriv^2F((1-t)\Pi_{\manifold{N}}(z)+tz)[z-\Pi_{\manifold{N}}(z),z-\Pi_{\manifold{N}}(z)]\leq M_F\dist(z,\manifold{N})^2,
\]
and the conclusion follows by integration over \([0,1]\) since \(\Deriv F\equiv 0\) on \(\mathcal{N}\).
\end{proof}

\begin{remark}\label{remarkDistanceSquare}
Many potentials \(F\) satisfy the condition \eqref{hyp22}, the most canonical being \(F(z)\defeq \dist(z,\manifold{N})^2\) in a neighbourhood of \(\manifold{N}\): we have for every \(z\in\manifold{N}_{\delta_\manifold{N}}\) and \(v\in\Rset^\nu\),
\[
\Deriv F(z)[v]=2(z-\Pi_{\manifold{N}}(z))\cdot v,
\]
and for every \(v_1,v_2\in \Rset^\nu\),
\begin{equation*}
\Deriv^2F(z)[v_1,v_2]=2(v_1-\Deriv\Pi_{\manifold{N}}(z)[v_1])\cdot v_2,
\end{equation*}
so that, in particular, \(\Deriv^2F(z)[v,v]=2\abs{v}^2\) if \(z\in\manifold{N}\) and \(v\in (T_z\manifold{N})^\perp\), since then \(\Deriv\Pi_{\manifold{N}}(z)\) is the orthogonal projection on \(T_z\manifold{N}\).
\end{remark}

\begin{remark}
In the previous example of the squared distance function, we have \(\abs{\nabla F}^2=4F\). In general, if \(F \in \mathcal{C}^3(\Rset^\nu , [0, +\infty))\) vanishes on \(\manifold{N}\) and satisfies \eqref{hyp22}, then the function \(G\), defined by \(G(y) \defeq \abs{\nabla F(y)}^2\), vanishes on \(\manifold{N}\) and satisfies \eqref{hyp22}. Indeed, if \(y\in\manifold{N}\) and \(v\in T_y\manifold{N}^\perp\setminus\{0\}\), then \(\Deriv G(y)[v]=2(\Deriv^2F(y)[v])[\nabla F(y)]\) and so \(\Deriv^2G(y)[v,v]=2\abs{\Deriv^2F(y)[v]}^2>0\) as \(\Deriv F\equiv 0\) on \(\manifold{N}\).
\end{remark}

\section{Renormalised energies and renormalisable harmonic maps}

\label{section_renormalised_energies}
\subsection{Topological resolution of the boundary datum}

Following our previous work \cite{Monteil_Rodiac_VanSchaftingen_RE}, we describe here the resolution of obstructions of the boundary datum that are responsible for asymptotical singularities for Ginzburg--Landau type functionals.

Given an open set \(\Omega \subset \Rset^2\), an integer \(k \in \Nset\) and a family of distinct points \(a_1, \dotsc, a_k \in \Omega\),
we define the quantity
\begin{multline}
\label{def_rho_barre}
  \Bar{\rho} (a_1, \dotsc, a_k)\\
  \defeq 
  \sup 
  \{\rho > 0 \st  \Bar{B}_{\rho} (a_i) \cap \Bar{B}_{\rho} (a_j) = \emptyset \text{ for each \(i, j \in  \{1, \dotsc, k\}\) such that \(i \ne j\) }\\
  \text{ and }
  \Bar{B}_\rho (a_i) \subset \Omega \text{ for each \(i \in \{1, \dotsc, k\}\)}\}
\end{multline}
and the notion of topological resolution \cite{Monteil_Rodiac_VanSchaftingen_RE}*{Definitions 2.1 and 2.2}
\begin{definition}
\label{def_resolution}
Given \(\Omega \subset \Rset^2\) a domain with a Lipschitz boundary, \(k\in\Nset_\ast\), \(k\) maps \(\gamma_1, \dotsc, \gamma_k\in \VMO (\Sset^1, \manifold{N})\) and a map \(g \in \VMO (\partial \Omega, \manifold{N})\), we say that \((\gamma_1, \dotsc, \gamma_k)\) is a \emph{topological resolution} of \(g\) whenever there exist points \(a_1, \dotsc, a_k \in \Omega\), a radius \(\rho \in (0, \Bar{\rho} (a_1, \dotsc, a_k))\),  and a continuous map \(u \in \mathcal{C} (\Bar{\Omega} \setminus \bigcup_{i = 1}^k B_\rho (a_i), \manifold{N})\) such that \(u \vert_{\partial \Omega}\) is homotopic to \(g\) in \(\VMO (\partial \Omega, \manifold{N})\) and for each \(i \in \{1, \dotsc, k\}\), \(u (a_i + \rho \cdot) \vert_{\Sset^1}\) is homotopic to \(\gamma_i\) in \(\VMO (\Sset^1, \manifold{N})\).
\end{definition}

\Cref{def_resolution} is invariant under changes of the positions of points and of the radius, and under homotopies of \(g\) in \(\VMO (\partial \Omega, \manifold{N})\) and of \(\gamma_1, \dotsc, \gamma_k\) in \(\VMO (\Sset^1, \manifold{N})\).
If the maps \(g, \gamma_1, \dotsc, \gamma_k\) are continuous, then 
we can assume without loss of generality in the definition that \(g = u \vert_{\partial \Omega}\)
and \(u (a_i + \rho \cdot) \vert_{\Sset^1} = \gamma_i\) everywhere \citelist{\cite{Brezis_Nirenberg_1995}\cite{Brezis_Nirenber_1996}}. Topological resolutions can be characterised algebraically in the fundamental group \(\pi_1 (\manifold{N})\) by conjugacy classes \cite{Monteil_Rodiac_VanSchaftingen_RE}*{Proposition 2.4}.

\subsection{Singular energy}
\label{section_singular_energy}
The minimal length in the homotopy class of \(\gamma\in\VMO (\Sset^1,\manifold{N})\) is defined as 
\begin{equation}
  \label{eq_minimalLengthVMO}
\equivnorm{\gamma}\defeq \inf \Bigl\{ \int_{\Sset^1} \abs{\Tilde{\gamma}'}\st \Tilde{\gamma} \in \mathcal{C}^1 (\Sset^1, \manifold{N}) \text{ and } \gamma \text{ are homotopic in \(\VMO (\Sset^1, \manifold{N})\)} \Bigr\}.
\end{equation}
We have then
\begin{equation}
  \label{eq_Eojei1cooruef6uiP}
 \inf \Bigl\{ \int_{\Sset^1} \abs{\Tilde{\gamma}'}^2 \st \Tilde{\gamma} \in \mathcal{C}^1 (\Sset^1, \manifold{N}) \text{ and } \gamma \text{ are homotopic in \(\VMO (\Sset^1, \manifold{N})\)} \Bigr\}
 = \frac{\equivnorm{\gamma}^2}{2 \pi},
\end{equation}
and equality is achieved if and only if \(\gamma\) is a minimising geodesic. The quantity \(\equivnorm{\gamma}\) is  invariant under homotopy in \(\VMO (\Sset^1, \manifold{N})\).

The \emph{systole} of the manifold \(\manifold{N}\) is the length of the shortest closed non-trivial geodesic on \(\manifold{N}\):
\begin{equation}
\label{eq_systole}
  \operatorname{sys} (\manifold{N}) =
  \inf \bigl\{\lambda (\gamma) \st \gamma \in \mathcal{C}^1(\Sset^1,\manifold{N}) \text{ is not homotopic to a constant}\bigr\}.
\end{equation}
In particular, for every \(\gamma \in \VMO (\Sset^1, \manifold{N})\), we have \(
 \equivnorm{\gamma}
 \in \{0\} \cup [\operatorname{sys} (\manifold{N}), +\infty)\).
When \(\manifold{N}\) is compact, \(\operatorname{sys} (\manifold{N}) > 0\).

\begin{proposition}
\label{proposition_equiv_norm_discrete}
If \(\manifold{N}\) is compact, then the set \(
 \{\equivnorm{\gamma} \st \gamma \in \VMO (\Sset^1, \manifold{N})\}
\)
is discrete.
\end{proposition}

\begin{proof}
By homotopy invariance of \(\equivnorm{\gamma}\) and thanks to the existence of geodesics in each homotopy class, we can assume that the maps \(\gamma\) are taken to be minimising geodesics.
We consider thus a sequence \((\gamma_n)_{n \in \Nset}\) in \(\mathcal{C}^1 (\Sset^1, \manifold{N})\) of minimising closed geodesics such that the sequence of numbers \((\equivnorm{\gamma_n})_{n \in \Nset}\) converges.
In view of \eqref{eq_Eojei1cooruef6uiP} and the Ascoli--Arzel\'a compactness criterion, there is a subsequence of \((\gamma_n)_{n \in \Nset}\) that converges uniformly and hence up to a further subsequence all the maps in the sequence \((\gamma_n)_{n \in \Nset}\) are homotopic and thus \((\equivnorm{\gamma_n})_{n \in \Nset}\) is constant, which implies that the set \(
 \{\equivnorm{\gamma} \st \gamma \in \VMO (\Sset^1, \manifold{N})\}
\) is discrete. 
\end{proof}

The first key quantity in the asymptotics for Ginzburg--Landau type functionals is the following \cite{Monteil_Rodiac_VanSchaftingen_RE}*{Section 2.2}.

\begin{definition}
\label{def_loose_equiv_norm}
If \(\Omega \subset \Rset^2\) is a Lipschitz bounded domain and \(g \in \VMO ( \partial \Omega, \manifold{N})\), we define its \emph{singular energy} to be
\begin{equation*}
    \Esing (g)
    \defeq \inf \Biggl\{ \sum_{i = 1}^k \frac{\equivnorm{\gamma_i}^2}{4 \pi} \st k \in \Nset_\ast \text{ and \((\gamma_1, \dotsc, \gamma_k)\) is a topological resolution of \(g\)}\Biggr\}.
\end{equation*}
\end{definition}
The singular energy \(\Esing\) is invariant under homotopies in \(\VMO (\partial \Omega, \manifold{N})\). For every \(\gamma \in \VMO (\Sset^1, \manifold{N})\), we have \(\Esing (\gamma)
\le \frac{\equivnorm{\gamma}^2}{4 \pi}\) (where in the definition of \(\Esing (\gamma)\), the circle \(\Sset^1\) is thought as the boundary of \(\Omega=B_1\)) and for every\(g \in \VMO (\partial \Omega, \manifold{N})\),
\begin{equation}
\label{eq_paih1ieb8eiNgo4fumaequ6u}
 \Esing (g)
 \in \{0\} \cup \Bigl[\frac{\syst(\manifold{N})^2}{4 \pi}, +\infty\Bigr).
\end{equation}

We say that \((\gamma_1,\dotsc,\gamma_k)\) is a \emph{minimal topological resolution} of \(g\)
whenever it is a topological resolution of \(g\) such that
\(
\Esing (g) = \sum_{i = 1}^k \frac{\equivnorm{\gamma_i}^2}{4 \pi}
\)
and for every \(i \in \{1, \dotsc, k\}\), \(\equivnorm{\gamma_i} > 0\) \cite{Monteil_Rodiac_VanSchaftingen_RE}*{Definition 2.7}. 
For example, if \(g \in \VMO(\partial \Omega,\mathbb{S}^1)\) and \(\deg (g)=d\in \mathbb{Z}\) then \(\Esing(g)=\pi \abs{d}^2\), and a minimal topological resolution is given by \(\abs{d}\) maps of degree \(1\) if \(d>0\), and \(\abs{d}\) maps of degree \(-1\) if \(d<0\). However, in general, minimal topological resolutions are not necessarily unique.

A closed curve \(\gamma \in \mathcal{C} (\Sset^1, \manifold{N})\) is said to be \emph{atomic} whenever \((\gamma)\) is a minimal topological resolution of \(\gamma\) \cite{Monteil_Rodiac_VanSchaftingen_RE}*{Definition 2.8}.
In particular, if \(\lambda (\gamma) = \operatorname{sys} (\manifold{N})\), then \(\gamma\) is atomic.
Atomicity does not exclude the existence of an alternative minimal topological resolution into several maps, this is the case for the manifold \(\manifold{N}\) arising as quotient of \(SU (2)\times SU(2)\) in models of superfluid \ce{^3He} \cite{Monteil_Rodiac_VanSchaftingen_RE}*{Section 9.3.5}.

\subsection{Synharmony between geodesics}
The notion of synharmony between geodesics quantifies how homotopic mappings can be connected almost through a minimising harmonic map \cite{Monteil_Rodiac_VanSchaftingen_RE}*{Section 3.2}.

\begin{definition}
\label{definition_synharmonic}
The \emph{synharmonicity} between two given maps \(\gamma, \beta\in W^{1/2,2}(\Sset^1 , \manifold{N})\), is defined as
\begin{multline*}
  \synhar{\gamma}{\beta}
  \defeq \inf \, \biggl\{ \int_{\Sset^1 \times [0, L]} \frac{\abs{\Deriv u}^2}{2} - \frac{L}{4 \pi} \equivnorm{\gamma}^2
  \st L \in (0, +\infty),
  u \in W^{1, 2} (\Sset^1 \times [0, L], \manifold{N}),\\[-1em]
  \tr_{\Sset^1 \times \{0\}} u  = \gamma \text{ and }
  \tr_{\Sset^1 \times \{L\}} u  = \beta \biggr\}.
\end{multline*}
\end{definition}

The synharmonicity is an extended pseudo-distance which is continuous with respect to the strong topology in \(W^{1/2,2} (\Sset^1, \manifold{N})\) \cite{Monteil_Rodiac_VanSchaftingen_RE}*{Proposition 3.3}.
Bounded sets in \(W^{1/2, 2} (\Sset^1, \manifold{N})\) which contain only homotopic maps have bounded synharmonicity \cite{Monteil_Rodiac_VanSchaftingen_RE}*{Proposition 3.5}.

Two maps \(\gamma, \beta \in W^{1/2, 2} (\Sset^1, \manifold{N})\) are \emph{synharmonic} whenever \(\synhar{\gamma}{\beta} = 0\) \cite{Monteil_Rodiac_VanSchaftingen_RE}*{Definition 3.6}.
The synharmony between minimising geodesics is an equivalence relation, partitioning each homotopy class of minimising geodesics into synharmony classes.
If \(\gamma, \beta \in W^{1/2, 2} (\Sset^1, \manifold{N})\) and \(\synhar{\gamma}{\beta} = 0\), then either \(\gamma = \beta\) almost everywhere in \(\Sset^1\) or
both \(\beta\) and \(\gamma\) are minimising geodesics \cite{Monteil_Rodiac_VanSchaftingen_RE}*{Proposition 3.7}.
Minimising geodesics that are homotopic through minimising geodesics are synharmonic \cite{Monteil_Rodiac_VanSchaftingen_RE}*{Proposition 3.8}; this covers in particular 
\(\gamma \compose R\) and \(\gamma\) where \(R \in SO (2)\).
Although homotopic minimising closed geodesics on a manifold are not synharmonic in general,
this is the case on examples that motivate the use of Ginzburg--Landau type energies in physics and geometry.

\subsection{Renormalised energies of configurations of points}
\label{section_ren1}

Given a bounded open set \(\Omega \subset \Rset^2\) with Lipschitz boundary \(\partial \Omega\), a map \(g \in W^{1/2,2}(\partial \Omega, \manifold{N})\subset \VMO(\partial \Omega, \manifold{N})\), \(k \in \Nset_\ast\) and \(k\) closed minimising geodesics \(\gamma_1,  \dotsc, \gamma_k \in W^{1/2,2}(\Sset^1, \manifold{N})\) that form a topological resolution of \(g\), we consider the \emph{geometrical renormalised energy} \cite{Monteil_Rodiac_VanSchaftingen_RE}*{Section 2.3} defined on the configuration space of \(\Omega\),
\[
\Conf{k} \Omega\defeq
\{(a_1,\dots,a_k)\in\Omega^k\st a_i\neq a_j\text{ if \(i\neq j\)}\},
\]
by setting for every \((a_1,\dots,a_k)\in\Conf{k} \Omega\),
\begin{equation}
\label{eq_def_renorm_geom}
\begin{split}
 \mathcal{E}_{g, \gamma_1, \dotsc, \gamma_k}^{\mathrm{geom}}  (a_1, \dotsc, a_k) &\defeq \lim_{\rho \to 0} \mathcal{E}_{g, \gamma_1, \dotsc, \gamma_k}^{\mathrm{geom}, \rho} (a_1, \dotsc, a_k) - \sum_{i = 1}^k \frac{\equivnorm{\gamma_i}^2}{4 \pi}  \log \frac{1}{\rho}\\
 & = \inf_{\rho \in (0, \Bar{\rho} (a_1, \dotsc, a_k))}
 \mathcal{E}_{g, \gamma_1, \dotsc, \gamma_k}^{\mathrm{geom}, \rho} (a_1, \dotsc, a_k) - \sum_{i = 1}^k \frac{\equivnorm{\gamma_i}^2}{4 \pi}  \log \frac{1}{\rho},
 \end{split}
\end{equation}
where for a radius \(\rho \in (0, \Bar{\rho} (a_1, \dotsc, a_k))\),
we have set 
\begin{multline}
\label{eq_def_renorm_geom_rho} \mathcal{E}_{g, \gamma_1, \dotsc, \gamma_k}^{\mathrm{geom}, \rho} (a_1, \dotsc, a_k)
  = \inf \Bigl\{  \int_{\Omega \setminus \bigcup_{i = 1}^k \Bar{B}_\rho (a_i)}
  \frac{\abs{\Deriv u}^2}{2} \st u \in W^{1, 2} (\Omega \setminus \textstyle \bigcup_{i = 1}^k \Bar{B}_\rho (a_i), \manifold{N}), \\[-.3em]
  \tr_{\partial \Omega} u = g \text{ on \(\partial \Omega\)}
  \text{ and } \tr_{\Sset^1} u (a_i + \rho \cdot) = \gamma_i
  \Bigr\}.
\end{multline}

The function \(\mathcal{E}^\mathrm{geom}_{g, \gamma_1, \dotsc, \gamma_k} : \operatorname{Conf}_k \Omega \to \Rset\) is locally Lipschitz-continuous \cite{Monteil_Rodiac_VanSchaftingen_RE}*{Proposition 4.1}.
If \((\gamma_1, \dotsc, \gamma_k)\) is a minimal topological resolution of \(g\), then the function \(\mathcal{E}^\mathrm{geom}_{g, \gamma_1, \dotsc, \gamma_k}\) is bounded from below on \(\operatorname{Conf}_k \Omega\) \cite{Monteil_Rodiac_VanSchaftingen_RE}*{Proposition 5.6}; moreover if \(\limsup_{n \to \infty} \mathcal{E}_{g, \gamma_1, \dotsc, \gamma_k}^{\mathrm{geom}}
 (a_1^n, \dotsc, a_k^n) < + \infty\), then the singularities \(a_1^n, \dotsc, a_k^n\) always stay away from the boundary and from each other unless their recombination yields another minimal topological resolution of the boundary datum \(g\) \cite{Monteil_Rodiac_VanSchaftingen_RE}*{Proposition 6.1}. (In many examples of interest this does not happen; it occurs for instance for the torus \(\Sset^1\times \Sset^1\) and in models of superfluid \ce{^3He} in dipole-free phase.)

The quantity \(\mathcal{E}_{g, \gamma_1, \dotsc, \gamma_k}^{\mathrm{geom}} (a_1, \dotsc, a_k)\) depends on the curves \(\gamma_i\) only up to synharmonicity: if for each \(i\), the curves \(\gamma_i\) and \(\Tilde{\gamma_i}\) are synharmonic, then \cite{Monteil_Rodiac_VanSchaftingen_RE}*{Proposition 3.10}
\begin{equation}
\label{synharmonicDependanceGeometric}\mathcal{E}_{g, \gamma_1, \dotsc, \gamma_k}^{\mathrm{geom}} (a_1, \dotsc, a_k)=\mathcal{E}_{g,\Tilde{\gamma_1}, \dotsc,\Tilde{\gamma_k}}^{\mathrm{geom}} (a_1, \dotsc, a_k).
\end{equation}

\label{section_ren2}

\subsection{Renormalised energy of renormalisable maps}
A  last notion from \cite{Monteil_Rodiac_VanSchaftingen_RE} that we will be using is the notion of renormalisable singular mappings and their renormalised energies \cite{Monteil_Rodiac_VanSchaftingen_RE}*{Definition 7.1}.

\begin{definition}
\label{def_renormalisable}
  Let \(\Omega\subset\Rset^2\) be a bounded Lipschitz domain. A mapping \(u:\Omega\to\manifold{N}\) is \emph{renormalisable} whenever there exists a finite set \(\{a_1,\dotsc,a_k\}\subset \Omega\) such that if \(\rho > 0\) is small enough,
  \(u \in W^{1, 2} (\Omega \setminus \bigcup_{i = 1}^k \Bar{B}_\rho(a_i) , \manifold{N})\) and
  its \emph{renormalised energy} is finite:
  \begin{equation*}
    \mathcal{E}^{\mathrm{ren}}(u)
    \defeq \liminf_{\rho \to 0} \int_{\Omega\setminus\bigcup_{i=1}^k \Bar{B}_\rho(a_i)}\frac{\abs{\Deriv u}^2}{2}-\sum_{i=1}^k
    \frac{\equivnorm{\tr_{\partial B_\rho (a_i)} u}^2}{4\pi}\log\frac{1}{\rho}<+\infty.
  \end{equation*}
\end{definition}

The set of renormalisable mappings is denoted by \(W^{1, 2}_{\mathrm{ren}} (\Omega, \manifold{N})\).
For every \(u \in W^{1, 2}_{\mathrm{ren}} (\Omega, \manifold{N})\) one has
  \begin{equation}
\label{renormalised_energy_u}
   \begin{split}
    \mathcal{E}^{\mathrm{ren}}(u)
    &=
      \lim_{\rho \to 0}
        \int_{\Omega\setminus\bigcup_{i=1}^k \Bar{B}_\rho(a_i)}
          \frac{\abs{\Deriv u}^2}{2}
          -\sum_{i=1}^k \frac{\equivnorm{\tr_{\Sset^1}u(a_i+\rho\,\cdot)}^2}{4\pi} \log\frac{1}{\rho} \\
    &= \sup_{\rho\in (0, \Bar{\rho} (a_1, \dotsc, a_k))}
    \int_{\Omega\setminus\bigcup_{i=1}^k \Bar{B}_\rho(a_i)}\frac{\abs{\Deriv u}^2}{2}
    -\sum_{i=1}^k \frac{\equivnorm{\tr_{\Sset^1}u(a_i+\rho\,\cdot)}^2}{4\pi} \log\frac{1}{\rho}.
    \end{split}
  \end{equation}

The structure of renormalisable mappings is described in the following \cite{Monteil_Rodiac_VanSchaftingen_RE}*{Proposition 7.2}:

\begin{proposition}
\label{admissibleMaps}
Let \(\Omega\subset\Rset^2\) be a bounded Lipschitz domain. If \(u\in W^{1,2}_{\mathrm{ren}}(\Omega,\manifold{N})\), then either one has \(u\in W^{1,2}(\Omega,\manifold{N})\) or there exist \(k\in\Nset_\ast\), \((a_1, \dotsc, a_k)\in\Conf{k}\Omega\) and \(\gamma_1, \dotsc,\gamma_k\in \mathcal{C}^1 (\Sset^1, \mathcal{N})\) such that
\begin{enumerate}[(i)]
  \item \((\gamma_1, \dotsc, \gamma_k)\) is a topological resolution of \(\tr_{\partial \Omega} u\),
  \item for each \(i \in \{1, \dotsc, k\}\), \(\gamma_i\) is a non-trivial minimising closed geodesic,
  \item for each \(i \in \{1, \dotsc, k\}\), there exists a sequence \((\rho_\ell)_{\ell \in \Nset}\) converging to \(0\) such that the sequence \((\tr_{\Sset^1} u (a_i + \rho_\ell\,\cdot))_{\ell \in \Nset}\) converges strongly to \(\gamma_i\) in \(W^{1, 2} (\Sset^1, \manifold{N})\),
  \item\label{synhamonicConvergenceSing} for each \(i \in \{1, \dotsc, k\}\),
  \(
  \lim_{\rho \to 0}\synhar{\tr_{\Sset^1}u(a_i+\rho\,\cdot)}{\gamma_i}=0,
  \)
    \item \label{it_bohp7shaephei0eeT}
  \(\displaystyle
    \mathcal{E}^{\mathrm{ren}}(u)\ge \mathcal{E}^{\mathrm{geom}}_{g, \gamma_1, \dotsc, \gamma_k}(a_1,\dotsc,a_k)\).
\end{enumerate}
In this case, we denote the set of singularities by \(\operatorname{sing}(u)=\{(a_1,\gamma_1), \dotsc, (a_k, \gamma_k)\}\); in the case where \(u\in W^{1,2}(\Omega,\manifold{N})\), we set \(\operatorname{sing} (u) = \emptyset\).
\end{proposition}

Stricly speaking, the set \(\operatorname{sing}(u)\) is only defined up to synharmony of the mappings \(\gamma_1, \dotsc, \gamma_k\).
Given \(u\in W^{1,2}_{\mathrm{ren}}(\Omega,\manifold{N})\), \(a\in\Omega\) and a minimising geodesics \(\gamma\), we have that \((a, \gamma)\in\operatorname{sing}(u)\) if and only if \(\Deriv u\) is not square-integrable near \(a\) and if \(\lim_{\rho \to 0}\synhar{\tr_{\Sset^1}u(a_i+\rho\,\cdot)}{\gamma}=0\). 
In particular, the set \(\operatorname{sing}(u)\) is well-defined up to synharmony of minimising geodesics.

\section{Minimal energy on disks with boundary conditions}
\label{section_minimal_energy_on_balls}
We recall that \(F\) denotes the Ginzburg--Landau penalisation which satisfies \(F \in \mathcal{C} (\Rset^\nu, [0,+\infty))\) and  \(F^{-1}(\{0\}) = \manifold{N}\).
For every radius \(R\in (0,+\infty)\) and every curve \(\gamma\in W^{1/2,2}(\Sset^1,\Rset^\nu)\), we set
\begin{equation}
\label{definitionQ}
\mathcal{Q}^R_{F, \gamma}\defeq \inf\left\{
 \int_{B_R}  \frac{\abs{\Deriv u}^2}{2} + F(u)
 \st u\in W^{1,2}(B_R,\Rset^\nu)\text{ s.t.
} \tr_{\Sset^1} u(R\,\cdot)=\gamma
\right\},
\end{equation}
where \(B_R \subset \Rset^2\) is the disk of radius \(R\) centred at the origin \(0 \in \Rset^2\).
By scaling, we have for every \(\varepsilon,R\in (0,+\infty)\)
\begin{equation}
\label{scaling_invariance}
\inf\left\{
 \int_{B_R}  \frac{\abs{\Deriv u}^2}{2} + \frac{F(u)}{\varepsilon^2}
 \st u\in W^{1,2}(B_R,\Rset^\nu)\text{ and } \tr_{\Sset^1} u(R\,\cdot)=\gamma
\right\}=\mathcal{Q}^{R/\varepsilon}_{F,\gamma}.
\end{equation}

\begin{proposition}
\label{proposition_lowerBndBalls}
If \(\gamma\in\mathcal{C}^1(\Sset^1,\manifold{N})\) is a minimising geodesic, then the map
\[
R\in (0,+\infty)\mapsto \mathcal{Q}^R_{F, \gamma} - \frac{\equivnorm{\gamma}^2}{4\pi} \log R
\]
is non-increasing.
\end{proposition}

By \cref{proposition_lowerBndBalls}, for every \emph{minimising closed geodesic} \(\gamma \in \mathcal{C}^1 (\Sset^1, \manifold{N})\), we can define
\begin{equation}
  \label{def_topCst}
  \mathcal{Q}_{F, \gamma}\defeq\lim_{R\to +\infty}\biggl(\mathcal{Q}^R_{F, \gamma}-\frac{\equivnorm{\gamma}^2}{4\pi} \log R\biggr)\in [-\infty,+\infty).
\end{equation}

When \(\manifold{N} = \Sset^1\), \cref{proposition_lowerBndBalls} is due to Bethuel, Brezis and Hélein \cite{Bethuel_Brezis_Helein_1994}*{Lemma III.1}.

\begin{remark}
  We shall see in Section \ref{sect:Lower_bounds} that \(\mathcal{Q}_{F, \gamma}>-\infty\) if \(\gamma\) is an atomic minimising geodesic, i.e.\ \(\Esing (\gamma)=\frac{\equivnorm{\gamma}^2}{4 \pi}\) (see \cref{corollary_lowerBndBalls}).
\end{remark}
\begin{proof}[Proof of \cref{proposition_lowerBndBalls}]
Given  \(0<R<S<+\infty\), we consider a map \(u\in W^{1,2}(B_{R},\Rset^\nu)\) such that \(\tr_{\Sset^1} u(R\,\cdot)=\gamma\) on \(\Sset^1\) and we define the map \(v\in W^{1,2}(B_{S},\Rset^\nu)\) for \(x \in B_R\) by
\[
v(x)=
\begin{cases}
u(x)&\text{if }x\in B_{R},\\
\gamma\left(\frac{x}{\abs{x}}\right)&\text{if }x\in B_{S}\setminus B_{R}.
\end{cases}
\]
Since \(\gamma\) is by assumption a minimising geodesic, we have
\begin{equation*}
\begin{split}
\int_{B_{S}}\frac{\abs{\Deriv v}^2}{2}+F(v)&\le \int_{B_{R}}\biggl(\frac{\abs{\Deriv u}^2}{2}+ F(u)\biggr)+\int_{\Sset^1}\frac{\abs{\gamma'}^2}{2}\int_{R}^{S}\frac{\dif r}{r}\\
&=\int_{B_{R}}\biggl(\frac{\abs{\Deriv u}^2}{2}+ F(u)\biggr)+\frac{\equivnorm{\gamma}^2}{4\pi} \log\frac{S}{R}.
\end{split}
\end{equation*}
By minimising over \(u\) and by definition  \eqref{def_topCst} of \(\mathcal{Q}^R_{F, \gamma}\), we get 
\[
\mathcal{Q}^{S}_{F,\gamma}-\frac{\equivnorm{\gamma}^2}{4\pi} \log S\le \mathcal{Q}^{ R}_{F,\gamma}-\frac{\equivnorm{\gamma}^2}{4\pi} \log R.\qedhere
\]
\end{proof}

\begin{proposition}
\label{proposition_QW_dependence}
If \(\gamma,\Tilde{\gamma}\in W^{1/2,2}(\Sset^1,\manifold{N})\), then for every \(R \in (0,+\infty)\), we have
\[
 \inf_{S \ge R} \biggl(\mathcal{Q}^{S}_{F, \Tilde{\gamma}}
 -
 \frac{\equivnorm{\Tilde{\gamma}}^2}{4\pi} \log S \biggr)
\le
\mathcal{Q}^R_{F, \gamma}-\frac{\equivnorm{\gamma}^2}{4\pi} \log R
+\synhar{\gamma}{\Tilde{\gamma}}.
\]
\end{proposition}

In particular, if \(\gamma\) and \(\Tilde{\gamma}\) are minimising geodesics, then
\(
\mathcal{Q}_{F,\Tilde{\gamma}}\le\mathcal{Q}_{F,\gamma}+\synhar{\gamma}{\Tilde{\gamma}}.
\)
If moreover the maps \(\gamma\) and \(\Tilde{\gamma}\) are synharmonic, then \(\mathcal{Q}_{F,\gamma}=\mathcal{Q}_{F,\Tilde{\gamma}}\).

\begin{proof}[Proof of \cref{proposition_QW_dependence}]
We can assume that the maps \(\gamma\) and \(\Tilde{\gamma}\) are homotopic and, in particular, that \(\equivnorm{\gamma}=\equivnorm{\Tilde{\gamma}}\) since otherwise \(\synhar{\gamma}{\Tilde{\gamma}}=+\infty\). 

We take \(R \in (0,+\infty)\), \(u \in W^{1, 2} (B_R,\Rset^\nu)\) such that \(\tr_{\Sset^1} u(R\,\cdot)=\gamma\), \(L > 0\) and \(H\in W^{1, 2} (\Sset^1 \times [0, L], \manifold{N})\) such that \(H(\cdot, 0) = \gamma\), \(H(\cdot, L) = \Tilde{\gamma}\).
We define \(v \in W^{1, 2} (B_{e^LR},\Rset^\nu)\) by
\[
v (x)=
\begin{cases}
  u (x) & \text{if \( x \in B_R\)},\\
  H\Bigl(\frac{x}{\abs{x}},  \log \frac{\abs{x}}{R}\Bigr)
  & \text{if \(x \in B_{e^L R}\setminus B_R\)}.
\end{cases}
\]
By taking the infimum with respect to \(u\) in the energy of \(v\) we obtain, 
\[
\mathcal{Q}^{e^L R}_{F,\Tilde{\gamma}} - \frac{\equivnorm{\Tilde{\gamma}}^2}{4 \pi} \log (e^L R)
\le \mathcal{Q}^{R}_{F, \gamma}  - \frac{\equivnorm{\gamma}^2}{4 \pi} \log R  + \int_{\Sset^1 \times [0, L]} \frac{\abs{\Deriv H}^2}{2} - \frac{L}{4 \pi}\equivnorm{\gamma}^2,
\]
and thus by definition of synharmonicity (\cref{definition_synharmonic}),
\[
 \inf_{S \ge R} \biggl(\mathcal{Q}^{S}_{F, \Tilde{\gamma}}
 -
 \frac{\equivnorm{\Tilde{\gamma}}^2}{4\pi} \log S \biggr)
\le
\mathcal{Q}^R_{F, \gamma}-\frac{\equivnorm{\gamma}^2}{4\pi} \log R
+\synhar{\gamma}{\Tilde{\gamma}}.\qedhere
\]
\end{proof}

Let \(u\in W^{1,2}_{\textrm{ren}} (\Omega,\manifold{N})\) and let \(\operatorname{sing} (u) \eqdef \{(a_1, \gamma_1), \dotsc, (a_k, \gamma_k)\}\) be given by \cref{admissibleMaps}. We define
\begin{equation}
\label{defTopCstU}
\mathcal{Q}_F(u)\defeq \sum_{i = 1}^k \mathcal{Q}_{F,\gamma_i},
\end{equation}
where the quantity \( \mathcal{Q}_{F, \gamma}\) is defined in \eqref{def_topCst}.
Finally if \(\gamma\) does not takes its values in \(\manifold{N}\) but is still close to it, the difference between \(\mathcal{Q}^R_{F,\gamma}\) and \(\mathcal{Q}^R_{F,\Pi_\manifold{N}(\gamma) }\)  can be estimated as follows.

\begin{proposition}
\label{prop_Q_F_projection}
If \(F \in \mathcal{C} (\Rset^\nu, [0,+\infty))\) satisfies \(F^{-1}(\{0\}) = \manifold{N}\) and \eqref{hyp20}, and if \(\gamma \in W^{1, 2} (\Sset^1, \Rset^\nu)\) satisfies \(\dist_{\manifold{N}} (\gamma(\cdot)) < \delta_{\manifold{N}}/2\) on \(\manifold{N}\),
then for every \(R \ge 2\),

\[
  \abs{\mathcal{Q}^{R}_{F, \gamma}
  - \mathcal{Q}^{R}_{F, \Pi_{\manifold{N}}\compose \gamma}}
  \le
  C
  \int_{\Sset^1}\Bigl( \frac{\abs{\gamma'}^2}{R} + R F (\gamma)\Bigr).
\]
\end{proposition}
\begin{proof}
Given \(u \in W^{1, 2} (B_R, \Rset^\nu)\) such that \(\tr_{\Sset^1} u(R\,\cdot)=\gamma\), we define \(v : B_R \to \Rset^\nu\) by setting for each \(x \in B_R\), 
\[
 v (x) =
 \begin{cases}
   u (\frac{R}{R-1}x) & \text{if \(\abs{x} \le R  - 1\)},\\
   (R - \abs{x}) \gamma (\tfrac{x}{\abs{x}}) + (\abs{x} - (R - 1)) \Pi_{\manifold{N}} (\gamma (\tfrac{x}{\abs{x}})) &  \text{if \(R - 1 \le \abs{x} \le R\)}.
 \end{cases}
\]
We compute that \(\Deriv v(x)=\frac{R}{R-1}\Deriv u(\frac{R}{R-1}x)\) if \(\abs{x} \leq R-1\), and if \( R-1\leq \abs{x} \leq R\),
\begin{multline*}
\abs{\Deriv v(x)}^2=\bigabs{\Pi_\manifold{N}\bigl(\gamma\bigl(\tfrac{x}{\abs{x}} \bigr)\bigr) -\gamma \bigl(\tfrac{x}{\abs{x}} \bigr)}^2 \\
+\frac{1}{\abs{x}^2}\,\bigabs{(R-\abs{x})\gamma'\bigl(\tfrac{x}{\abs{x}} \bigr)-(\abs{x}-(R-1))\Deriv\Pi_\manifold{N}(\gamma\bigl(\tfrac{x}{\abs{x}} \bigr)\bigl[\gamma'\bigl(\tfrac{x}{\abs{x}} \bigr) \bigr]}^2.
\end{multline*}
By smoothness and compactness the derivatives of \(\Pi_\manifold{N}\) are bounded in \(\manifold{N}_{\delta_\manifold{N}/2}\) and we have
\begin{equation*}
\bigabs{(R-\abs{x})\gamma'\bigl(\tfrac{x}{\abs{x}} \bigr)-(\abs{x}-(R-1))\Deriv\Pi_\manifold{N}(\gamma\bigl(\tfrac{x}{\abs{x}} \bigr))\bigl[\gamma'\bigl(\tfrac{x}{\abs{x}} \bigr) \bigr] }^2  \leq C \bigabs{\gamma' \bigl( \tfrac{x}{\abs{x}} \bigr)}^2.
\end{equation*}
Moreover, in view of \eqref{hyp20}, we estimate
\begin{equation*}
\bigabs{\Pi_\manifold{N}\bigl(\gamma\bigl(\tfrac{x}{\abs{x}} \bigr)\bigr) -\gamma \bigl(\tfrac{x}{\abs{x}} \bigr)}^2 =\dist(\gamma\bigl(\tfrac{x}{\abs{x}} \bigr),\manifold{N})^2 \leq \C F\bigl( \gamma\bigl(\tfrac{x}{\abs{x}} \bigr) \bigr)
\end{equation*}
and
\begin{multline*}
F\Bigl( \Pi_\manifold{N} \bigl(\gamma\bigl(\tfrac{x}{\abs{x}}\bigr)\bigr)+ (R-\abs{x}) \left(\gamma\bigl(\tfrac{x}{\abs{x}}\bigr)- \Pi_\manifold{N}\bigl(\gamma\bigl(\tfrac{x}{\abs{x}}\bigr)\bigr)   \right) \Bigr)  \\
\leq \C \dist \Bigl( \Pi_\manifold{N} \bigl(\gamma\bigl(\tfrac{x}{\abs{x}}\bigr)\bigr)+ (R-\abs{x}) \left(\gamma\bigl(\tfrac{x}{\abs{x}}\bigr)- \Pi_\manifold{N}\bigl(\gamma\bigl(\tfrac{x}{\abs{x}}\bigr) \bigr)  \right),\mathcal{N} \Bigr)^2 \\
\leq \C \bigabs{\Pi_\manifold{N}\bigl(\gamma\bigl(\tfrac{x}{\abs{x}} \bigr)\bigr) -\gamma \bigl(\tfrac{x}{\abs{x}} \bigr)}^2 \leq \C  F\bigl( \gamma\bigl(\tfrac{x}{\abs{x}} \bigr) \bigr).
\end{multline*}

By using a change of variables and integration in polar coordinates we arrive at
\[
 \int_{B_R} \frac{\abs{\Deriv v}^2}{2} + F (v) \le \int_{B_R} \frac{\abs{\Deriv u}^2}{2} + F (u)
 + \Cl{cst_ahS1gaishae1euNae} \biggl(\int_{\Sset^1} \frac{\abs{\gamma'}^2}{R} + \int_{\Sset^1} R\, F(\gamma)\biggr).
\]
It follows thus that
\[
\mathcal{Q}^{R}_{F, \Pi_{\manifold{N}}\compose \gamma}
-\mathcal{Q}^{R}_{F, \gamma}\le \Cr{cst_ahS1gaishae1euNae}  \biggl(\int_{\Sset^1} \frac{\abs{\gamma'}^2}{R} + R\, F (\gamma)\biggr).
\]
The proof of the converse inequality is similar.
\end{proof}

\section{Upper bound on the energy of minimisers}
\label{section_upper_bound}

Thanks to the singular and renormalised energies presented in \S \ref{section_renormalised_energies} and the minimal energy on disks developed in \S \ref{section_minimal_energy_on_balls}, we establish an upper bound on the Ginzburg--Landau energy \(\mathcal{E}^\varepsilon_F (u)\), defined in \eqref{eq_GLenergy}.
In this section, \(\Omega\) is a Lipschitz bounded domain and \(F \in \mathcal{C} (\Rset^\nu, [0,+\infty))\) satisfies \(F^{-1}(\{0\}) = \manifold{N}\) and \eqref{hyp20}.

We first give an upper bound on the infimum of the energy with given Dirichlet boundary datum in terms of the infimum of the geometric renormalised energy. 

\begin{proposition}%
\label{proposition_geometricUpperBound}
Let \(g\in W^{1/2,2}(\partial\Omega,\manifold{N})\), \(k\in\Nset_\ast\), \(a_1,\dotsc,a_k\) be distinct points in \(\Omega\), and let \((\gamma_{1},\dotsc,\gamma_{k})\) be a minimal topological resolution of \(g\). Then, as \(\varepsilon \to 0\),
  \begin{multline*}
    \inf \{ \mathcal{E}^\varepsilon_F (u) \st u \in W^{1, 2} (\Omega, \Rset^\nu) \text{ and } \tr_{\partial \Omega} u = g\}\\
    \le  \Esing (g) \log \frac{1}{\varepsilon} + \mathcal{E}^{\mathrm{geom}}_{g, \gamma_1, \dotsc,\gamma_k}(a_1,\dotsc,a_k)+\sum_{i=1}^k \mathcal{Q}_{F,\gamma_i}
    + o (1).
  \end{multline*}
\end{proposition}

When \(\manifold{N} = \Sset^1\), \cref{proposition_geometricUpperBound}
is due to Bethuel, Brezis and Hélein \cite{Bethuel_Brezis_Helein_1994}*{Lemma VIII.1}.

\begin{proof}%
[Proof of \cref{proposition_geometricUpperBound}]
For every \(\rho \in (0, \Bar{\rho} (a_1, \dotsc, a_k))\),
we consider a map \(u_*\in W^{1, 2} (\Omega \setminus \bigcup_{i = 1}^k \Bar{B}_\rho (a_i), \manifold{N})\)
such \(\tr_{\partial \Omega} u_* = g\) and \(\tr_{\Sset^1} u_* (a_i + \rho\, \cdot) =\gamma_i\) for every \(i \in \{1, \dotsc, k\}\)
and maps \(u_1, \dots, u_k \in W^{1, 2} (B_\rho, \Rset^\nu)\) such that \(\tr_{\Sset^1} u_i (\rho\, \cdot) = \gamma_i\).
We then set
\[
u (x)\defeq
\begin{cases}
  u_* (x)&\text{if \(x\in\Omega\setminus\bigcup_{i=1}^k B_{\rho}(a_i)\),}\\
  u_i(x-a_i)&\text{if \(x\in B_{\rho}(a_i)\) for some \(i\in\{1,\dotsc,k\}\),}
\end{cases}
\]
and we have, since \(F (u_*) = 0\) in \(\Omega \setminus\bigcup_{i=1}^k B_{\rho}(a_i)\),
\[
\mathcal{E}^\varepsilon_F (u)=\int_{\Omega} \frac{\abs{\Deriv u}^2}{2} + \frac{F (u)}{\varepsilon^2}
= \int_{\Omega \setminus\bigcup_{i=1}^k B_{\rho}(a_i)} \frac{\abs{\Deriv u_*}^2}{2} +
\sum_{i = 1}^n \int_{B_\rho} \frac{\abs{\Deriv u_i}^2}{2} + \frac{F (u_i)}{\varepsilon^2}.
\]
By taking the infimum over \(u_*, u_1, \dotsc, u_k\), we obtain by \eqref{eq_def_renorm_geom_rho} and \eqref{scaling_invariance},
\[
\inf \,\bigl\{ \mathcal{E}^\varepsilon_F (u) \st u \in W^{1, 2} (\Omega, \Rset^\nu) \text{ and } \tr_{\partial \Omega} u = g\bigr\}
\le \mathcal{E}^{\mathrm{geom}, \rho}_{g, \gamma_1, \dotsc, \gamma_k} (a_1, \dotsc, a_k)
  + \sum_{i = 1}^k \mathcal{Q}^{\rho/\varepsilon}_{F,\gamma_i}.
\]
By choosing now \(\rho = \sqrt{\varepsilon}\), we obtain
\begin{multline*}
  \inf \, \bigl\{ \mathcal{E}^\varepsilon_F (u) \st u \in W^{1, 2} (\Omega, \Rset^\nu) \text{ and } \tr_{\partial \Omega} u = g\bigr\}
  -
  \sum_{i = 1}^k \frac{\equivnorm{\gamma_i}^2}{4 \pi} \log \frac{1}{\varepsilon}\\
\le \mathcal{E}^{\mathrm{geom}, \sqrt{\varepsilon}}_{g, \gamma_1, \dotsc, \gamma_k} (a_1, \dotsc, a_k)
- \sum_{i = 1}^k \frac{\equivnorm{\gamma_i}^2}{4 \pi} \log \frac{1}{\sqrt{\varepsilon}}
+ \sum_{i = 1}^k \mathcal{Q}^{\frac{1}{\sqrt{\varepsilon}}}_{F,\gamma_i} - \sum_{i = 1}^k \frac{\equivnorm{\gamma_i}^2}{4 \pi} \log \frac{1}{\sqrt{\varepsilon}},
\end{multline*}
and the conclusion follows by letting \(\varepsilon \to 0\) from the definition \eqref{eq_def_renorm_geom} of \(\mathcal{E}_{g, \gamma_1, \dotsc, \gamma_k}^{\mathrm{geom}}  (a_1, \dotsc, a_k)\) and the definition  \eqref{def_topCst} of  \(\mathcal{Q}_{F,\gamma_i}\).
\end{proof}

We also have an upper bound around singularities for  renormalisable maps.
\begin{proposition}
\label{proposition_improvedupperBound}
 For every \(u \in W^{1,2}_{\mathrm{ren}}(\Omega,\manifold{N})\),
if \(\operatorname{sing} (u) = \{(a_1, \gamma_1), \dotsc, (a_k, \gamma_k)\}\), then for every \(\rho \in (0, \Bar{\rho} (a_1, \dotsc, a_k))\), as \(\varepsilon \to 0\),
\begin{multline*}
\inf \, \bigl\{ \mathcal{E}^\varepsilon_F (v) \st v \in W^{1, 2} (\Omega, \Rset^\nu) \text{ and } v = u \text{ in \(\textstyle \Omega \setminus \bigcup_{i = 1}^k B_\rho (a_i)\) } \bigr\} \\
\le \sum_{i = 1}^k \frac{\equivnorm{\gamma_i}^2}{4 \pi} \log \frac{1}{\varepsilon} +   \mathcal{E}^{\mathrm{ren}} (u)+ \mathcal{Q}_{F} (u) + o(1).
\end{multline*}
\end{proposition}
The quantity \(\mathcal{Q}_F(u)\) has been defined in \eqref{defTopCstU}.

\begin{proof}%
  [Proof of \cref{proposition_improvedupperBound}]
For every \(u_1, \dotsc, u_k \in W^{1, 2} (B_\rho, \Rset^\nu)\) such that \(\tr_{\Sset^1} u_i (\rho\, \cdot) = \tr_{\Sset^1}u (a_i + \rho \,\cdot)\) for each \(i \in \{1, \dotsc, k\}\), if we define the function \(v : \Omega \to \Rset^\nu\) by
\[
v (x)=
\begin{cases}
  u(x)&\text{if }x\in\Omega\setminus\bigcup_{i=1}^k \Bar{B}_{\rho}(a_i),\\
  u_i(x-a_i)&\text{if for some \(i\in\{1,\dotsc,k\}\), \(x\in B_{\rho}(a_i)\),}
\end{cases}
\]
then we have
\[
 \mathcal{E}^\varepsilon_F (v)
 = \int_{\Omega \setminus \bigcup_{i = 1}^k B_\rho (a_i)} \frac{\abs{\Deriv u}^2}{2}
 + \sum_{i = 1}^k \int_{B_{\rho}} \frac{\abs{\Deriv u_i}^2}{2} + \frac{F (u_i)}{\varepsilon^2},
\]
and thus by taking the infimum over \(u_1, \dotsc, u_k\), we obtain by \eqref{scaling_invariance},
\begin{multline*}
\inf \{ \mathcal{E}^\varepsilon_F (v) \st v \in W^{1, 2} (\Omega, \Rset^\nu) \text{ and } v = u \text{ in \(\textstyle \Omega \setminus \bigcup_{i = 1}^k B_\rho (a_i)\) } \} - \sum_{i = 1}^k \frac{\equivnorm{\gamma_i}^2}{4 \pi} \log \frac{1}{\varepsilon}\\
\le \int_{\Omega \setminus \bigcup_{i = 1}^k B_\rho (a_i)} \frac{\abs{\Deriv u}^2}{2}
- \sum_{i = 1}^k \frac{\equivnorm{\gamma_i}^2}{4 \pi} \log \frac{1}{\rho}
+ \sum_{i = 1}^k \biggl(\mathcal{Q}^{\rho/\varepsilon}_{F,\tr_{\Sset^1} u (a_i + \rho\, \cdot)}
-  \frac{\equivnorm{\gamma_i}^2}{4 \pi} \log \frac{\rho}{\varepsilon}\biggr).
\end{multline*}
We conclude by the definition \eqref{renormalised_energy_u} of \(\mathcal{E}^{\mathrm{ren}} (u)\), by the definitions \eqref{def_topCst} and \eqref{defTopCstU} of the quantities \(\mathcal{Q}_{F,\gamma_i}(u)\) and \(\mathcal{Q}_F(u)\), and by \cref{proposition_QW_dependence} and \eqref{synhamonicConvergenceSing} in \cref{admissibleMaps}.
\end{proof}

\section{Lower bounds on the energy}
\label{sect:Lower_bounds}
We derive a lower bound for the Ginzburg--Landau energy \(\mathcal{E}^\varepsilon_F (u)\), defined in \eqref{eq_GLenergy}, of maps \(u\) in \(W^{1,2} (\Omega,\Rset^\nu)\) with given boundary datum \(\tr_{\partial \Omega} u = g\) that matches the upper bound of \cref{proposition_improvedupperBound}.
We first prove in \cref{section_neiZoh3ahvoh1iNg6} a lower bound of the form \(\Esing (g)\log \frac{1}{\varepsilon}-C\) for maps in \(W^{1,2}(\Omega,\Rset^\nu)\) and for the Ginzburg--Landau energy.
This lower bound along with a localisation of the energy argument allows us to prove boundedness of sequences which have their energies bounded by \(\Esing (g)\log \frac{1}{\varepsilon}+C\) in \cref{section_ogeePa6nooxeeN2ei}.
We have seen in the previous section that such a bound is satisfied by minimisers of \eqref{eq_GLenergy}.
With the help of the compactness of minimisers we are able to improve the lower bound and obtain the desired result in \cref{section_aejie9eiLoh9haire}. 

In this section, \(\Omega\) is a Lipschitz bounded domain and \(F \in \mathcal{C} (\Rset^\nu, [0,+\infty))\) satisfies \(F^{-1}(\{0\}) = \manifold{N}\) and \eqref{hyp20}.

\subsection{Global lower bound}
\label{section_neiZoh3ahvoh1iNg6}
The global lower bound depends  on the tubular neighbourhood extension energy.
\begin{definition}
\label{def_energy_extension}
  Let \(\Omega\subset\Rset^2\) be a bounded Lipschitz domain and \(g \in W^{1/2,2}(\partial \Omega, \manifold{N})\). We define the \emph{tubular neighbourhood extension energy} of \(g\) to be
  \begin{equation*}
    \mathcal{E}^{\mathrm{ext}}(g)
    \defeq
    \inf\Bigl\{\int_{\partial \Omega \times [0,1]} \frac{\abs{\Deriv v}^2}{2} \st v \in W^{1,2}(\partial \Omega\times [0,1],\manifold{N}) \text{ and } \tr_{\partial \Omega \times \{0 \}}v=g \Bigr\}.
  \end{equation*}
\end{definition}

\begin{proposition}
  \label{proposition_Firstlowerbound}
There exists a constant \(C \in (0, +\infty)\), depending only on \(\Omega\) and \(F\), such that for every \(\varepsilon>0\) and every \(u\in W^{1, 2}(\Omega,\Rset^\nu)\) with \(g\defeq\tr_{\partial \Omega} u\) satisfying \(g \in \manifold{N}\) almost everywhere on \(\partial \Omega\), we have
  \begin{multline*}
  \mathcal{E}^\varepsilon_F(u)+C\mathcal{E}^{\mathrm{ext}} (g)
  - \Esing (g)\log \frac{1}{C\varepsilon \Esing (g)}\\
  \geq\frac{1}{C}\biggl(\frac{\abs{\Deriv (\dist_{\manifold{N}} \compose u)}^2}{2} + \frac{F (u)}{\varepsilon^2}+ \sup_{t > 0} t^2 \mathcal{L}^2 \bigl(\abs{\Deriv u}^{-1} ([t, +\infty))\bigr)\biggr),
  \end{multline*}
where the last term on the left-hand side is understood to vanish when  \(\Esing(g)=0\).
\end{proposition}

When \(\manifold{N} = \Sset^1\), \Cref{proposition_Firstlowerbound}, without the weak estimate on the gradient, is due to Sandier \cite{Sandier_1998}*{Theorem 2}, the corresponding weak estimate being due to Serfaty and Tice \cite{Serfaty_Tice_2008}*{Theorem 2}.
In the general case, the fact that the left-hand side is non-negative is due to Canevari \cite{Canevari_2015}.

\Cref{proposition_Firstlowerbound} will follow from a slightly refined result for smooth maps (see  \Cref{lemma_Firstlowerbound}) together with an approximation argument. The proof of \cref{lemma_Firstlowerbound} follows Sandier's strategy \cite{Sandier_1998}*{Proof of Theorem 2} by an application of the coarea formula and the lower estimates for the Dirichlet energy outside a compact set of maps into a manifold, which depends on the one-dimensional Hausdorff content, whose definition and properties we recall now.

\begin{definition}
\label{def_hausdorff_content}
The \emph{one-dimensional Hausdorff content} of a compact set \(K \subset \Rset^2\) is defined as
\[
 \mathcal{H}^1_\infty (K)
 \defeq
 \inf \biggl\{ \sum_{B \in \mathcal{B}} \diam (B) \st \ K \subset \bigcup_{B \in \mathcal{B}} B \text{ and } \mathcal{B} \text{ is a finite collection of closed disks}\biggr\}.
\]
\end{definition}

The one-dimensional Hausdorff content is an outer measure and is bounded from above by the Hausdorff measure:
\begin{equation}
\label{eq_be1efeethietieQuu}
  \mathcal{H}^1_{\infty} (K) \le \mathcal{H}^1 (K).
\end{equation}

 We also recall the following lemma which will be used repeatedly to transform a covering of some set by disks into a covering by closed disks with disjoint closure (see Lemma 4.1 in \cite{Sandier_Serfaty_2007}).

\begin{lemma}
\label{lemmaMerging}
For every finite set \(\mathcal{B}\) of disks of \(\Rset^2\),
there exists a finite set \(\mathcal{B}'\) of disjoint non-empty closed disks of \(\Rset^2\) such that
\[
 \mathcal{B} = \bigcup_{B' \in \mathcal{B}'} \{B \in \mathcal{B} \st B \subseteq B'\},
\]
and 
\[
  \sum_{B' \in \mathcal{B}'}\diam (B') =\sum_{B \in \mathcal{B}} \diam (B). 
\]
\end{lemma}

We finally rely on the equality between the one-dimensional Hausdorff content of a compact set and of its boundary.

\begin{lemma}
\label{lemma_content_boundary}
If \(K \subseteq \Rset^2\) is compact, then
\(\mathcal{H}^1_\infty (K) = \mathcal{H}^1_\infty (\partial K)
\).
\end{lemma}
\Cref{lemma_content_boundary} does not hold for the \emph{Haudorff measure}; the proof of \cref{lemma_content_boundary} can be seen to work when \(K \subseteq \Rset^n\) is compact and \(n \ge 2\); the equality fails when \(n = 1\) and \(K=[0, 1] \subset \Rset\).
\begin{proof}[Proof of \cref{lemma_content_boundary}]
By monotonicity of the Hausdorff content, we have \(\mathcal{H}^1_\infty (K)
\ge \mathcal{H}^1_{\infty} (\partial K)\).
It remains thus to establish the converse inequality.

We fix \(\eta > 0\).
By definition of the Hausdorff content, there exist points \(a_1, \dotsc, a_k \in \Rset^2\) and radii \(\rho_1, \dotsc, \rho_k \in (0, + \infty)\)
such that \(\partial K \subseteq \bigcup_{i = 1}^k B_{\rho_i} (a_i)\) and \(\sum_{i = 1}^k 2 \rho_i \le  \mathcal{H}^1_\infty (\partial K) + \eta\).
By \cref{lemmaMerging}, we can assume that \(\Bar{B}_{\rho_i} (a_i) \cap \Bar{B}_{\rho_j} (a_j) = \emptyset\) if \(i, j \in \{1, \dotsc, k\}\) with \(i \ne j\).
We claim that \(K \subset \bigcup_{i = 1}^k B_{\rho_i} (a_i)\).
Indeed, assume by contradiction that there exists a point \(x \in K \setminus  \bigcup_{i = 1}^k B_{\rho_i} (a_i)\).
Since the disks \(\Bar{B}_{\rho_1} (a_1), \dotsc, \Bar{B}_{\rho_k} (a_k)\) are pairwise disjoint, the set \(\Rset^2 \setminus  \bigcup_{i = 1}^k B_{\rho_i} (a_i)\) is path-connected.
Since the set \(K\) is compact, we have \(\Rset^2\setminus (K\cup \bigcup_{i = 1}^k B_{\rho_i} (a_i)) \ne\emptyset\) and there exists thus a continuous map \(\gamma \in \mathcal{C} ([0, 1],\Rset^2 \setminus \bigcup_{i = 1}^k B_{\rho_i} (a_i))\) such that \(\gamma (0) = x\) and \(\gamma (1) \not \in K\).
Since the map \(\gamma\) is continuous, there exists some \(t_* \in [0, 1]\) such that \(\gamma (t_*) \in \partial K\) and we would thus have \(\partial K \setminus \bigcup_{i = 1}^k B_{\rho_i} (a_i)\neq \emptyset\), which is a contradiction.
We have thus
\[
 \mathcal{H}^1_\infty (K)
 \le 2 \sum_{i = 1}^k \rho_i \le \mathcal{H}^1_\infty (\partial K) + \eta;
\]
we conclude by letting \(\eta \to 0\).
\end{proof}

We will use the lower estimate on the Dirichlet energy of maps into a manifold proved in \cite{Monteil_Rodiac_VanSchaftingen_RE}*{Theorem 5.1}.
\begin{theorem}
\label{prop_lower_bound_compact}
For every Lipschitz bounded domain \(\Omega\subset \Rset^2\), every compact set \(K\subset \Omega\) such that \(\mathcal{H}^1_{\infty}(K)>0\) and every map \(v \in W^{1, 2} (\Omega \setminus K, \manifold{N})\), we have
\begin{equation}\label{first_prop_lower_bound_compact}
\int_{\Omega \setminus K} \frac{\abs{\Deriv v}^2}{2}\geq \Esing (\tr_{\partial \Omega}v) \log \frac{\dist(K,\partial \Omega)}{2\mathcal{H}^1_{\infty} (K)}.
\end{equation}
More precisely, there exists a constant \(C > 0\) such that
\begin{equation}
\label{weakLorentzEstimate}
\sup_{t>0}t^2\mathcal{L}^2 \bigl(\{x \in \Omega \setminus K \st \abs{\Deriv v} \ge t\}\bigr)
\le
C
\biggl(\int_{\Omega \setminus K} \frac{\abs{\Deriv v}^2}{2}
 -  \Esing (\tr_{\partial \Omega}v) \log \frac{\dist(K,\partial \Omega)}{2 \mathcal{H}^1_{\infty} ({K})} \biggr).
\end{equation}
\end{theorem}
The left-hand side of \eqref{weakLorentzEstimate} is the weak-\(L^2\) quasi-norm of \(\abs{\Deriv v}\). \Cref{prop_lower_bound_compact} has its roots in a corresponding estimate for maps outside a finite collection of disks \cite{Bethuel_Brezis_Helein_1994}*{Corollary II.1}. 

We are now ready to state a slightly refined version of \cref{proposition_Firstlowerbound} in the smooth setting:

\begin{lemma}
\label{lemma_Firstlowerbound}
 There exist constants \(C\in (0\,+\infty)\) and \(\delta\in (0,+\infty)\) depending only on \(\Omega\) and \(F\), such that for every \(\varepsilon>0\) and every map \(u\in \mathcal{C}^{2}(\Bar{\Omega},\Rset^\nu)\) with \(g\defeq \operatorname{tr}_{\partial \Omega}u\) satisfying \(g (\partial \Omega) \subseteq \manifold{N}\) and \(\Esing(g)>0\), we have
 \begin{multline*}
 \mathcal{E}^\varepsilon_F(u)+C \mathcal{E}^{\mathrm{ext}} (g)- \Esing (g)\log \frac{1}{C\varepsilon \Esing (g)}\\
 \geq 
 \frac{1}{C}\biggl(\int_{\Omega} \biggl(\frac{\abs{\Deriv (\dist_{\manifold{N}} \compose u)}^2}{2} + \frac{F (u)}{\varepsilon^2}\biggr)
  + \sup_{t > 0} t^2 \mathcal{L}^2\bigl(\{x\in \Omega \st \abs{\Deriv u(x)}\geq t\}\bigr)\\
 +\Esing (g)\,\frac{1}{\delta}\int_0^{\delta} \Psi \biggl(\frac{\mathcal{H}^1_\infty (K_s) \, s }{C\varepsilon\Esing(g)}\biggl) \dif s
\biggr),
\end{multline*}
where the function \(\Psi : (0, +\infty) \to \Rset_+\) is defined by \(\Psi (\tau) \defeq \tau - 1 - \log \tau\) for each \(\tau \in (0, +\infty)\), and where the sets \(K_s\) are defined for every \(s\in (0,+\infty)\) by
\begin{equation*}
  K_s \defeq \{x \in \Omega \st  \dist(u(x),\manifold{N})\geq s \}.
\end{equation*}
\end{lemma}

\resetconstant
Before proving \Cref{lemma_Firstlowerbound} we extend maps in \(u\in W^{1,2}(\Omega,\Rset^\nu)\) in the following way.
In view of \cref{def_energy_extension}, there exists \(\delta_{\partial \Omega}>0\) such that if we set
\begin{equation}\label{def:tubular_neighboorhood_of_boundary}
\Omega_{\delta_{\partial \Omega}}\defeq\{x \in \Rset^2\st \dist(x, \partial \Omega)< \delta_{\partial \Omega}\},
\end{equation}
then, we can extend the function \(u\in W^{1,2}(\Omega,\Rset^\nu)\) to a function \(u \in W^{1, 2} (\Omega_{\delta_{\partial \Omega}}, \Rset^\nu)\) in such a way that \(u \in \manifold{N}\) almost everywhere in \(\Omega_{\delta_{\partial \Omega}} \setminus \Omega\) and
\begin{equation}
\resetconstant
\label{extensionEstimate}
\int_{\Omega_{\delta_{\partial \Omega}} \setminus \Omega} \frac{\abs{\Deriv u}^2}{2} \le \Cl{cst_Xa7bu0Yiucheephie} \mathcal{E}^{\mathrm{ext}} (g),
\end{equation}
for some constant \(\Cr{cst_Xa7bu0Yiucheephie}\) depending only on \(\partial \Omega\) (see \cite{Monteil_Rodiac_VanSchaftingen_RE}*{Lemma 5.5}).

\emph{In the rest of the present work we will always assume that maps \(u\in W^{1,2}(\Omega,\Rset^\nu)\) are extended to the larger domain \(\Omega_{\delta_{\partial \Omega}}\) as explained above.}

\begin{proof}[Proof of \Cref{lemma_Firstlowerbound}]

We proceed in several steps:

\medskip
\noindent\emph{Step 1. Splitting normal and tangential derivatives.}
We set for \(x\in\Omega \setminus K_{\delta_\manifold{N}}\),
\begin{align*}
 \Deriv^\top u(x) & \defeq P^\top_{\manifold{N}} (\Pi_{\manifold{N}} (u(x)))\compose \Deriv u(x) &
 &\text{ and }&
 \Deriv^\perp u(x) & \defeq P^\perp_{\manifold{N}} (\Pi_{\manifold{N}} (u(x)))\compose \Deriv u(x),
\end{align*}
with the nearest point retraction \(\Pi_{\manifold{N}}\) and the projections \(P^\top_{\manifold{N}} \) and \(P^\perp_{\manifold{N}}\) being defined in \cref{lemma_derivativeNearestPointRetraction}; there holds in particular, within the set \(\Omega\setminus K_{\delta_\manifold{N}}\),
\begin{gather}
\label{estimatesTangentialDerivatives}
\left(1-\frac{\dist_{\manifold{N}} \compose u}{\delta_\manifold{N}}\right)\abs{\Deriv (\Pi_{\manifold{N}}\compose u)}^2
\le 
\abs{\Deriv^\top u}^2
\le
\Cl{cst_nfoei3909j}\abs{\Deriv (\Pi_{\manifold{N}}\compose u)}^2,
\intertext{and}
\label{estimatesNormalDerivatives}
\abs{\Deriv (\dist_{\manifold{N}}\compose u)}^2\le \abs{\Deriv^\perp u}^2.
\end{gather}
We also let \(\delta_F\in(0,\delta_\manifold{N})\) be a constant as in \cref{lemma_compFanddist} so that for all \(y\in \manifold{N}_{\delta_F}\), we have
\begin{equation}
\label{nonDegF}
F(y)\ge \frac{m_F}{2}\dist_{\manifold{N}}(y)^2.
\end{equation}
By orthogonality between \(P^\perp_{\manifold{N}}\) and \(P^\top_{\manifold{N}}\), we have for every \(\delta\in(0,\delta_F]\)
\begin{equation}
\label{normalTangentialSplitting}
\begin{split}
\mathcal{E}^\varepsilon_F(u)&=
\int_{\Omega\setminus K_{{{\delta}}}}
\biggl(\frac{\abs{\Deriv^\perp u}^2}{2}+ \frac{F \compose u}{\varepsilon^2}
\biggr)
+
\int_{\Omega\setminus K_{{{\delta}}}}\frac{\abs{\Deriv^\top u}^2}{2}
+
\int_{K_{{{\delta}}}}
\biggl(\frac{\abs{\Deriv u}^2}{2}+ \frac{F \compose u}{\varepsilon^2}\biggr)\\
&\eqdef\mathbf{(I)+(II)+(III)}.
\end{split}
\end{equation}

\noindent\emph{Step 2. Estimate of \((\mathbf{I})\) from below.}
Since \(u \in \mathcal{C}^2 (\Omega, \Rset^\nu)\), by Sard's lemma and by the implicit function theorem, for almost every \(s \in (0, +\infty)\),
the set \(K_{s} \subset \Omega\) has a \(\mathcal{C}^2\) boundary and
\begin{equation*}
  \partial K_s = \Sigma_s
  \defeq \{x \in \Omega \st \dist(u(x),\manifold{N})= s  \}.
\end{equation*}
Hence, using successively \eqref{estimatesNormalDerivatives}, Young's inequality, \eqref{nonDegF} and the coarea formula, we obtain
\begin{equation*}
  \begin{split}
 \mathbf{(I)}=\int_{\Omega\setminus K_{{{\delta}}}}
\biggl(\frac{\abs{\Deriv^\perp u}^2}{2}+ \frac{F \compose u}{\varepsilon^2}
\biggr)&\ge \int_{\Omega \setminus K_{{{\delta}}}}\biggl(\frac{\abs{\Deriv (\dist_{\manifold{N}} \compose u)}^2}{2}+ \frac{F \compose u}{\varepsilon^2}\biggr)\\
  &\ge\int_{\Omega \setminus K_{{{\delta}}}} \frac{1}{\varepsilon}\abs{\Deriv (\dist_{\manifold{N}} \compose u)} \sqrt{2 F \compose u}\\
  &\ge\sqrt{m_F}\int_{\Omega \setminus K_{{{\delta}}}} \frac{1}{\varepsilon}\abs{\Deriv (\dist_{\manifold{N}}\compose u)} (\dist_{\manifold{N}} \compose u)\\
  &=\sqrt{m_F}\int_0^{{{\delta}}} \frac{\mathcal{H}^1(\Sigma_s) \, s}{\varepsilon} \dif s.
  \end{split}
\end{equation*}
But, by \cref{lemma_content_boundary} and \eqref{eq_be1efeethietieQuu}, we have for almost every \(s>0\), \(\mathcal{H}^1_\infty (K_s) = \mathcal{H}^1_\infty(\Sigma_s)\le \mathcal{H}^1 (\Sigma_s)\); hence
\[
 \mathbf{(I)}\ge\sqrt{m_F}\int_0^{{{\delta}}}\frac{\mathcal{H}^1_\infty (K_s)\, s}{\varepsilon} \dif s.
 \]
Moreover, by Chebyshev's inequality, we have also
\begin{equation*}
\mathbf{(I)}\ge\int_{\Omega\setminus K_{{{\delta}}}} \frac{\abs{\Deriv^\perp u}^2}{2}
\ge
\sup_{t>0}\frac{t^2}{8}\mathcal{L}^2\bigl(\{x\in \Omega\setminus K_{{{\delta}}} \st \abs{\Deriv^\perp u(x)}\geq t/2\}\bigr).
\end{equation*}
We have thus proved that there exists a constant \(\Cl{estimateOfI}>0\) such that
\begin{multline}
  \label{estimateOfI}
  \mathbf{(I)}\ge
  \frac{1}{\Cr{estimateOfI}}\biggl(
\int_0^{{{\delta}}} \frac{\mathcal{H}^1_\infty(K_s) \, s}{\varepsilon} \dif s + \int_{\Omega\setminus K_{{{\delta}}}} \frac{\abs{\Deriv  (\dist_{\manifold{N}} \compose u)}^2}{2} + \frac{F \compose u}{\varepsilon^2}\\
+\sup_{t>0}t^2\mathcal{L}^2\bigl(\{x\in \Omega\setminus K_{{{\delta}}} \st \abs{\Deriv^\perp u(x)}\geq t/2\}\bigr)\biggr).
  \end{multline}

\noindent\emph{Step 3. Estimate of \(\mathbf{(II)}\) from below.} By \eqref{estimatesTangentialDerivatives} and Fubini's theorem, we have
\begin{equation}
\label{firstEstII}
\begin{split}
 \mathbf{(II)}=\int_{\Omega\setminus K_{{{\delta}}}} \frac{\abs{\Deriv^\top u}^2}{2}
 &\ge 
  \int_{\Omega \setminus {K_{{{\delta}}}}}\left( 1-\frac{ \dist_{\manifold{N}} \compose u}{\delta} \right) \frac{\abs{D(\Pi_{\manifold{N}}\compose u)}^2}{2} \\
  & =
   \int_{\Omega \setminus {K_{{{\delta}}}}} \left( \frac{1}{{{\delta}}}\int_{\dist_{\manifold{N}} \compose u}^{{{\delta}}}\dif s\right) \frac{\abs{D(\Pi_{\manifold{N}}\compose u)}^2}{2}
 \\
  & = \frac{1}{{{\delta}}} \int_0^{{{\delta}}} \left( \int_{\Omega \setminus K_s} \frac{\abs{D(\Pi_{\manifold{N}}\compose u)}^2}{2} \right)\dif s.
  \end{split}
\end{equation}
Then, by \eqref{extensionEstimate}, we have for every \(s\in (0,{{\delta}})\),
\begin{equation*}
 \int_{\Omega\setminus K_s}\frac{\abs{D(\Pi_{\manifold{N}}\compose u)}^2}{2}
\ge
\int_{\Omega_{\delta_{\partial \Omega}}\setminus K_s} \frac{\abs{D(\Pi_{\manifold{N}}\compose u)}^2}{2} 
- \Cr{cst_Xa7bu0Yiucheephie}\mathcal{E}^{\mathrm{ext}}(g),
\end{equation*}
while by the lower estimate on the Dirichlet energy of mappings \cref{prop_lower_bound_compact}, since \(\Pi_{\manifold{N}} \compose u \in W^{1, 2} (\Omega_{\delta_{\partial \Omega}} \setminus K_{{{\delta}}}, \manifold{N})\), for every \(t>0\),
\begin{equation*}
\begin{split}
\int_{\Omega_{\delta_{\partial \Omega}}\setminus K_s} \frac{\abs{D(\Pi_{\manifold{N}}\compose u)}^2}{2} 
& \ge 
\Esing (g) \log \frac{\delta_{\partial \Omega}}{2 \mathcal{H}^1_\infty(K_s)} \\
&\qquad+\frac{1}{\Cl{cst_thuusfoep09h}}t^2\mathcal{L}^2 (\{x \in \Omega \setminus K_s \st \abs{\Deriv (\Pi_{\manifold{N}} \compose u)(x)} \ge t\}),
\end{split}
\end{equation*}
for some constant \(\Cr{cst_thuusfoep09h}>0\). Since by \eqref{estimatesTangentialDerivatives}, we have \(\abs{\Deriv^\top u}\le \sqrt{\Cr{cst_nfoei3909j}}\abs{\Deriv (\Pi_{\manifold{N}} \compose u)}\), we have also
\begin{equation*}
\begin{split}
\mathcal{L}^2 (\{x \in \Omega \setminus K_s \st \abs{\Deriv (\Pi_{\manifold{N}} \compose u)(x)} \ge t\})
\geq\mathcal{L}^2 (\{x \in \Omega \setminus K_s\st \abs{\Deriv^\top u(x)} \ge \sqrt{\Cr{cst_nfoei3909j}}\, t\}).
\end{split}
\end{equation*}
We thus arrive at
\begin{equation*}
 \begin{split}
  \int_{\Omega\setminus K_s} \frac{\abs{D(\Pi_{\manifold{N}}\compose u)(x)}^2}{2} 
&\ge
\Esing (g) \log \frac{\delta_{\partial \Omega}}{2 \mathcal{H}^1_\infty(K_s)} 
- \Cr{cst_Xa7bu0Yiucheephie} \mathcal{E}^{\mathrm{ext}} (g)
\\
&\qquad+\frac{1}{\Cr{cst_thuusfoep09h}}t^2\mathcal{L}^2 (\{x \in \Omega \setminus K_s\st \abs{\Deriv^\top u(x)} \ge \sqrt{\Cr{cst_nfoei3909j}}\, t\}).  
 \end{split}
\end{equation*}
By integration with respect to \(s\) over \((0,{{\delta}})\), we obtain in view of \eqref{firstEstII},
\begin{multline*}
 \mathbf{(II)}
  \ge
\Esing (g)\frac{1}{{{\delta}}}\int_0^{{{\delta}}}\log \frac{\delta_{\partial \Omega}}{2\mathcal{H}^1_\infty(K_s)} \dif s  - \Cr{cst_Xa7bu0Yiucheephie}\mathcal{E}^{\mathrm{ext}} (g)\\
  +
 \frac{1}{\Cr{cst_thuusfoep09h}}t^2\frac{1}{{{\delta}}}\int_0^{{{\delta}}}
 \mathcal{L}^2 (\{x \in \Omega \setminus K_s \st \abs{\Deriv^\top u(x)} \ge \sqrt{\Cr{cst_nfoei3909j}}\, t\})\dif s.
  \end{multline*}
By Fubini's theorem, we compute
  \begin{equation*}
  \begin{split}
  \frac{1}{{{\delta}}}\int_0^{{{\delta}}}\mathcal{L}^2 (\{x \in \Omega \setminus K_s \st \abs{ \Deriv^\top u} \ge \sqrt{\Cr{cst_nfoei3909j}}\, t\})\dif s
&=\int_{\{x\in\Omega\setminus K_{{{\delta}}}\st \abs{ \Deriv^\top u}\ge \sqrt{\Cr{cst_nfoei3909j}}\, t\}}\left(1-\frac{\dist_{\manifold{N}}\compose u}{\delta}\right)\\
&\geq \frac 12\mathcal{L}^2\bigl(\{x\in\Omega\setminus K_{{{\delta}}/2}\st \abs{\Deriv^\top u(x)}\ge \sqrt{\Cr{cst_nfoei3909j}}\, t\}\bigr)
\end{split}
 \end{equation*}
and by the change of variable \(s=2\sqrt{\Cr{cst_nfoei3909j}}\, t\),
\begin{multline*}
\sup_{t>0}t^2\mathcal{L}^2\bigl(\{x\in\Omega\setminus K_{{{\delta}}/2}\st \abs{\Deriv^\top u}\ge \sqrt{\Cr{cst_nfoei3909j}}\, t\}\bigr)\\
\geq
\sup_{s>0}\frac{s^2}{4\Cr{cst_nfoei3909j}}\mathcal{L}^2\bigl(\{x\in\Omega\setminus K_{{{\delta}}/2}\st \abs{\Deriv^\top u(x)}\ge s/2\}\bigr).
\end{multline*}

Hence, we have proved
  \begin{multline}
   \label{estimateOfII}
 \mathbf{(II)}
  \ge
 \Esing (g)\frac{1}{\delta}
\int_0^{{{\delta}}}
\log \frac{\delta_{\partial \Omega}}{2\mathcal{H}^1_\infty(K_s)}  \dif s - \Cr{cst_Xa7bu0Yiucheephie}\mathcal{E}^{\mathrm{ext}} (g)\\
  +
  \frac{1}{8\Cr{cst_nfoei3909j}\Cr{cst_thuusfoep09h}}\sup_{t>0}t^2\mathcal{L}^2\bigl(\{x\in\Omega\setminus K_{{{\delta}}/2}\st \abs{\Deriv^\top u(x)}\ge t/2\}\bigr).
 \end{multline}

\noindent\emph{Step 4. Estimate of \(\mathbf{(III)}\) from below.} By \eqref{estimatesNormalDerivatives} and Chebyshev's inequality, we have
\begin{multline}
\label{estimateOfIII}
\mathbf{(III)}=\int_{K_{{{\delta}}}}
\biggl(\frac{\abs{\Deriv u}^2}{2}+ \frac{F \compose u}{\varepsilon^2}\biggr)
\ge\frac 12\biggl(\int_{K_{{{\delta}}}}
\biggl(\frac{\abs{\Deriv (\dist_{\manifold{N}} \compose u)}^2}{2}+ \frac{F \compose u}{\varepsilon^2}\biggr)
+
\int_{K_{{{\delta}}}}
\frac{\abs{\Deriv u}^2}{2}\biggr)
\\
\ge
\frac{1}{8}\biggl(\int_{K_{{{\delta}}}}
\biggl(\frac{\abs{\Deriv (\dist_{\manifold{N}} \compose u)}^2}{2}+ \frac{F \compose u}{\varepsilon^2}\biggr)
+\sup_{t>0}t^2\mathcal{L}^2(\{x \in K_{{{\delta}}} \st \abs{\Deriv u(x)}\geq t\})\biggr).
\end{multline}

\noindent\emph{Step 5. Putting things together.} We first observe that for every \(t>0\),
\begin{equation*}
\begin{split}
t^2\mathcal{L}^2(\{x\in K_{{{\delta}}}\cup\Omega\setminus K_{{{\delta}}/2}\st \abs{\Deriv u(x)}\geq t\})
&\leq  
t^2\mathcal{L}^2(\{x\in K_{{{\delta}}} \st \abs{\Deriv u(x)}\geq t\})\\  
&\quad + t^2 \mathcal{L}^2(\{ x \in \Omega \setminus K_{{{\delta}}/2} \st \abs{\Deriv^\perp u(x) }\geq t/2 \})\\
&\quad +t^2 \mathcal{L}^2(\{ x \in \Omega \setminus K_{{{\delta}}/2} \st \abs{\Deriv^\top u(x)}\geq t/2 \}),
\end{split}
\end{equation*}
which in view of \eqref{normalTangentialSplitting}, by adding \eqref{estimateOfI}, \eqref{estimateOfII} and \eqref{estimateOfIII}, gives the existence of a constant \(\Cl{cst_ifoewj032do0e2}>0\) such that
\begin{multline}
\label{est_mfi03qk9rfeokp}
\mathcal{E}^\varepsilon_F(u)\ge 
\Esing (g)\frac{1}{\delta}\int_0^{{{\delta}}}
\biggl(\frac{\delta\mathcal{H}^1_\infty(K_s) \, s}{\Cr{estimateOfI}\varepsilon\Esing (g)} +\log \frac{\delta_{\partial \Omega}}{2\mathcal{H}^1_\infty(K_s)} \biggr)\dif s 
- \Cr{cst_Xa7bu0Yiucheephie}\mathcal{E}^{\mathrm{ext}} (g)
\\
+\frac{1}{\Cr{cst_ifoewj032do0e2}}\biggl(
\int_{\Omega} \frac{\abs{\Deriv  (\dist_{\manifold{N}} \compose u)}^2}{2} + \frac{F \compose u}{\varepsilon^2}\\
+\sup_{t>0}t^2\mathcal{L}^2\bigl(\{x\in K_{{{\delta}}}\cup\Omega\setminus K_{{{\delta}}/2} \st \abs{\Deriv u(x)}\geq t\}\bigr)\biggr).
\end{multline}
Applying the identity \(\tau=1+\log\tau+\psi(\tau)\) to \(\tau=\frac{\delta\mathcal{H}^1_\infty(K_s) \, s}{\Cr{estimateOfI}\varepsilon\Esing (g)}\), we obtain
\begin{equation*}
\frac{\delta\mathcal{H}^1_\infty(K_s) \, s}{\Cr{estimateOfI}\varepsilon\Esing (g)} +\log \frac{\delta_{\partial \Omega}}{2\mathcal{H}^1_\infty(K_s)}
=
 1+\log \frac{\delta\delta_{\partial \Omega} s}{ 2\Cr{estimateOfI}\varepsilon \Esing (g)}
+ \Psi \left(\frac{{{\delta}}\mathcal{H}^1_\infty (K_s) \, s }{ \Cr{estimateOfI} \varepsilon \Esing (g)}\right),
\end{equation*}
and we compute that
\begin{equation*}
\frac{1}{\delta}\int_0^{{{\delta}}} \left(1+\log \frac{\delta\delta_{\partial \Omega} s}{ 2\Cr{estimateOfI}\varepsilon \Esing (g)}\right)\dif s
=
\log \frac{\delta^2\delta_{\partial \Omega}}{ 2\Cr{estimateOfI}\varepsilon \Esing (g)}.
\end{equation*}
Hence, there exists a constant \(C>0\) such that
\begin{multline}
\label{justBeforeEnd}
\mathcal{E}^\varepsilon_F(u)+C\mathcal{E}^{\mathrm{ext}} (g)-
\Esing (g)\log \frac{1}{C\varepsilon \Esing (g)}
\\
\ge \frac{1}{C}\biggl(
\int_{\Omega} \frac{\abs{\Deriv  (\dist_{\manifold{N}} \compose u)}^2}{2} + \frac{F \compose u}{\varepsilon^2}+\sup_{t>0}t^2\mathcal{L}^2\bigl(\{x\in K_{{{\delta}}}\cup\Omega\setminus K_{{{\delta}}/2} \st \abs{\Deriv u(x)}\geq t\}\bigr)\\
+\Esing (g)\frac{1}{\delta}\int_0^{\delta} \Psi \biggl(\frac{\mathcal{H}^1_\infty (K_s) \, s }{C\varepsilon\Esing(g)}\biggl) \dif s
\biggr).
\end{multline}
Since we have
\begin{equation*}
\begin{split}
\mathcal{L}^2\bigl(\{x\in \Omega\st \abs{\Deriv u(x)}\geq t\}\bigr)
\le&
\mathcal{L}^2\bigl(\{x\in K_{\delta_F}\cup\Omega\setminus K_{\delta_F/2} \st \abs{\Deriv u(x)}\geq t\}\bigr)\\
&+
\mathcal{L}^2\bigl(\{x\in K_{\delta_F/2} \cup\Omega\setminus K_{\delta_F/4}\st \abs{\Deriv u(x)}\geq t\}\bigr),
\end{split}
\end{equation*}
the desired estimate follows by taking the average of \eqref{justBeforeEnd} for \(\delta\in\{\delta_F,\frac{\delta_F}{2}\}\).
\end{proof}

\begin{proof}[Proof of \cref{proposition_Firstlowerbound}]
  If the Ginzburg--Landau functional is continuous with respect to the \(W^{1,2}\) strong convergence,
  the conclusion follows from \cref{lemma_Firstlowerbound} and an approximation argument.

If the Ginzburg--Landau functional is not continuous, we consider a non-decreasing sequence \((F^\ell)_{\ell \in \Nset}\)
of bounded and continuous functions coinciding with \(F\) in a neighbourhood of \(\manifold{N}\) and converging to \(F\) almost everywhere. 
The proposition holds for each of these functions,  since \(\mathcal{E}_{F_\ell}^\varepsilon\) is continuous for the \(W^{1,2}\) strong convergence by Lebesgue's dominated convergence. The constants appearing only depend on \(\delta_{F_\ell}\) and \(m_{F_\ell}\), which by construction of \(F_\ell\) are indepedent on \(\ell\). 
Hence the conclusion holds with the same constant for each \(F_\ell\), and the conclusion then holds by Lebesgue's monotone convergence theorem.
\end{proof}

As a consequence of \cref{proposition_Firstlowerbound}, we obtain a lower estimate on Ginzburg--Landau energy in terms of the Ginzburg--Landau energy on the boundary.

\begin{corollary}
\label{corollary_lowerBndBoundary}
There exists a constant \(C \in (0, +\infty)\), depending only on \(\Omega\) and \(F\), such that for every \(\varepsilon>0\) and every \(u\in W^{1, 2}(\Omega,\Rset^\nu)\) with \(g \defeq \tr_{\partial \Omega} u \in W^{1, 2} (\partial \Omega, \Rset^\nu)\) and \(\dist(g, \manifold{N}) \le \delta_{\manifold{N}}\), we have
  \begin{equation*}
  \mathcal{E}^\varepsilon_F(u)+C\mathcal{E}^{\varepsilon}_{F, \partial \Omega}  (g)
  \ge \Esing (\Pi_{\manifold{N}} \compose g)\log \frac{1}{C\varepsilon \Esing (\Pi_{\manifold{N}} \compose g)}.
  \end{equation*}
\end{corollary}

Here, \(\mathcal{E}^{\varepsilon}_{F, \partial \Omega}\) is the boundary Ginzburg--Landau functional, defined for \(g \in W^{1, 2} (\partial \Omega, \Rset^\nu)\) by \[
\mathcal{E}^{\varepsilon}_{F, \partial \Omega} (g)
\defeq \int_\Omega \frac{\abs{g'}^2}{2}+\frac{F(g)}{\varepsilon^2}.
\]

\begin{proof}[Proof of \cref{corollary_lowerBndBoundary}]
\resetconstant
We define \(\Tilde{\Omega} \defeq \Omega_\delta\) and \(\Tilde{u}: \Tilde{\Omega} \to \Rset^\nu\) by 
\begin{equation*}
 \Tilde{u} (x)
 \defeq 
 \begin{cases}
  u (x) & \text{if \(x \in \Omega\)},\\
  (1 - \frac{\dist(x, \partial \Omega)}{\delta}) u (\Pi_{\partial \Omega} (x)) + \frac{\dist(x, \partial \Omega)}{\delta}\Pi_{\manifold{N}} u (\Pi_{\partial \Omega} (x))
  & \text{otherwise}.
 \end{cases}
\end{equation*}
where \(\Pi_{\partial \Omega}\) is a retraction of \(\Omega_\delta \setminus \Omega\) on \(\partial \Omega\).
By assumption \eqref{hyp20} on \(F\), we have 
\begin{equation}
\label{eq_hohX0mohkieth4poov4mie7n}
 \mathcal{E}^{\varepsilon} (\Tilde{u}) \le  \mathcal{E}^\varepsilon_F(u) + \C \mathcal{E}^{\varepsilon}_{F, \partial \Omega} (g).
\end{equation}
Moreover, if \(\Tilde{g} \defeq \tr_{\partial \Tilde{\Omega}} \Tilde{u}\), we have 
\begin{equation}
\label{eq_wo3eNei2ohpi5Cakoodochu4}
  \mathcal{E}^{\varepsilon} (\Tilde{g}) \le \C  \mathcal{E}^{\varepsilon}_F (g).
\end{equation}
The conclusion then follows from \cref{proposition_Firstlowerbound} applied to the function \(\Tilde{u}\) and from the estimates \eqref{eq_hohX0mohkieth4poov4mie7n} and \eqref{eq_wo3eNei2ohpi5Cakoodochu4}.
\end{proof}

As a consequence of \cref{proposition_Firstlowerbound}, we obtain the finiteness of the quantity \(\mathcal{Q}_{F, \gamma}\) defined in \eqref{def_topCst}, when \(\gamma\) is an atomic minimising geodesic, i.e.\ when \(\Esing (\gamma)=\equivnorm{\gamma}^2/(4 \pi)\).

\begin{corollary}
\label{corollary_lowerBndBalls}
Let \(F \in \mathcal{C} (\Rset^\nu, [0,+\infty))\).
If \(F^{-1}(\{0\}) = \manifold{N}\), if \(F\) satisfies \eqref{hyp20}
and if \(\gamma\in\mathcal{C}^1(\Sset^1,\manifold{N})\) is an atomic minimising geodesic, then
\(
\mathcal{Q}_{F, \gamma}
>-\infty
\).
\end{corollary}

\begin{proof}%
\resetconstant
By \cref{proposition_Firstlowerbound} applied to \(\Omega=B_1\), the unit disk with centre \(0\) in \(\Rset^2\),
there exists a constant \(\Cl{cst_aw9azae3Ieje1eeshubei5La}\in(0,+\infty)\) such that for every \(\varepsilon>0\) and every \(u\in W^{1,2} (B_1,\Rset^\nu)\)
we have
\begin{equation*}
  \int_{B_1} \frac{\abs{\Deriv u}^2}{2} +\frac{F(u)}{\varepsilon^2} \ge
  \Esing (\tr_{\partial \Omega}u)\log \frac{1}{\Cr{cst_aw9azae3Ieje1eeshubei5La}\varepsilon \Esing (\tr_{\partial \Omega} u)}
  -
\Cr{cst_aw9azae3Ieje1eeshubei5La}\mathcal{E}^{\mathrm{ext}} (\tr_{\partial \Omega}u).
\end{equation*}
By taking the infimum over \(u\) such that \(\tr_{\partial B_1} u = \gamma\), we obtain, in view of \eqref{scaling_invariance}, with \(\rho=\frac{1}{\varepsilon}\),
\[
\mathcal{Q}^{\rho}_{F,\gamma}-\Esing (\gamma)\log \rho \ge  \Esing (\gamma)\log \frac{1}{\Cr{cst_aw9azae3Ieje1eeshubei5La}\Esing (\gamma)}
  -
  \Cr{cst_aw9azae3Ieje1eeshubei5La}\mathcal{E}^{\mathrm{ext}} (\gamma).
\]
The claim follows from \eqref{def_topCst} since, by assumption, \(\Esing (\gamma)=\frac{\equivnorm{\gamma}^2}{4 \pi}\).
\end{proof}

\subsection{Localised lower bound on the energy}
\label{section_ogeePa6nooxeeN2ei}
The next proposition provides some information on the localisation of the energy of mappings satisfying a logarithmic bound.

\begin{proposition}
\label{theorem_localisation}
There exists a constant \(C \in (0,+\infty)\) such that for every \(\kappa \in (0, +\infty)\), \(\eta \in (0, 1/C)\), \(\gamma \in (0, 1)\), \(\varepsilon\in (0,+\infty)\) and \(g\in W^{1/2,2}(\partial\Omega,\manifold{N})\) such that
\begin{align*}
\Esing(g)>0,\quad C e^{C\gamma\kappa}(\Esing (g)\varepsilon)^{1-\gamma}&\le \gamma\eta,
&
C \Esing (g)\varepsilon &< \gamma \eta,&
C \varepsilon \kappa &\le 1&
&\text{and}&
 \mathcal{E}^{\mathrm{ext}} (g) &\le \kappa,
\end{align*}
if \(u\in W^{1, 2}(\Omega,\Rset^\nu)\) satisfies \(\tr_{\partial \Omega}u=g\) and
\begin{equation}
\label{eq_neephohy5ne5eLiu9OM1biim}
  \int_{\Omega} \frac{\abs{\Deriv u}^2}{
    2} + \frac{F (u)}{\varepsilon^2}
  \leq \Esing (g)\log\frac{1}{\varepsilon {\Esing (g)}}+\kappa,
\end{equation}
then, if we still denote by \(u\) the extension to \( \Omega_{\delta_{\partial \Omega}}\) satisfying \eqref{extensionEstimate}, there exists a collection of disks \(\mathcal{B}\) in \(\Rset^2\) with
\begin{enumerate}[(i)]
  \item \label{it_OoPhah4ouneukoot6thos4cu} for every \(B\in \mathcal{B}\), \(\diam(B) \le 2\eta\) and \( \Bar{B}\subset \Omega_{\delta_{\partial \Omega}}\),
  \item for every \(B\in \mathcal{B}\), \(\dist_{\manifold{N}} \compose \tr_{\partial B} u < \delta_{\manifold{N}}\), the map
  \(\Pi_{\manifold{N}} \compose \tr_{\partial B} u\) is not homotopic to a constant and the maps \((\Pi_{\manifold{N}} \compose \tr_{\partial B} u)_{B\in \mathcal{B}}\) are a topological resolution of \(g=\tr_{\partial\Omega} u\),
  \item
  \label{it_uNamohyo4Ciece8aegah1que}
  for  every subset \(\mathcal{B}' \subset \mathcal{B}\),
  \begin{equation*}
\int\limits_{\Omega \cap \bigcup_{B \in \mathcal{B}'} B} \frac{\abs{\Deriv u}^2}{
    2} + \frac{F (u)}{\varepsilon^2}
  \ge
  \sum_{B \in \mathcal{B}'}
  \Esing (\Pi_{\manifold{N}} \compose \tr_{\partial B}u) \log \frac{\gamma \eta}{C  \Esing (g) \varepsilon}
  - C (\kappa + \Esing (g)),
\end{equation*}
\item
\label{it_dieYair4co5Ui2oowipheisa}
one has
\begin{equation*}
\frac{\syst (\manifold{N})^2}{4 \pi} \# \mathcal{B}
\le
 \sum_{B\in \mathcal{B}} \Esing (\Pi_\manifold{N} \compose \tr_{\partial B} u)
 \le
 \Esing (g)
 + \frac{(\log \frac{C}{\gamma \eta} + C \varepsilon) \Esing (g) + (1+C \varepsilon)\kappa}
 {
\log \frac{\gamma \eta}{C \varepsilon \Esing (g)}
 }.
\end{equation*}
\end{enumerate}
\end{proposition}

In the statement of \cref{theorem_localisation}, \(\syst (\manifold{N})\) denotes the systole of the manifold \(\manifold{N}\) defined in \eqref{eq_systole} as the shortest length of a closed geodesic which is not homotopic to a constant.

\Cref{theorem_localisation} has its roots in lower bounds for minimisers of the Ginzburg--Landau energy when \(\manifold{N} = \Sset^1\) \cite{Bethuel_Brezis_Helein_1994}*{Theorem V.2}; localised lower bounds for \(\manifold{N} = \Sset^1\)
are originally due to Sandier \cite{Sandier_1998}*{Theorem 3\textprime} and Jerrard \cite{Jerrard_1999}*{Theorem 1.2}.

We follow in our proof the Jerrard’s strategy \cite{Jerrard_1999} (see also the recent work by Ignat and Jerrard \cite{Ignat_Jerrard_2017}).
As a first tool to prove \cref{theorem_localisation}, we have a Sobolev type embedding theorem with dependence on \(\varepsilon\) for maps defined on \(\Sset^1_r\defeq\partial B_r\), the circle of radius \(r\) centred at the origin in \(\Rset^2\) (see also \cite{Jerrard_1999}*{Lemma 2.3}).

\begin{lemma}
\label{lemma_W12_Linfty_1d_epsilon}
There exists a constant \(C > 0\) such that for every \(r > 0\), every \(h \in W^{1, 2} (\Sset^1_r, \Rset)\) and every \(\varepsilon \in (0,r]\), 
\[
\norm{h}_{L^\infty (\Sset^1_r)}^2
\le C
\int_{\Sset^1_r} \varepsilon \abs{h'}^2 + \frac{1}{\varepsilon}h^2.
\]
\end{lemma}
\begin{proof}
By Morrey--Sobolev embedding, the function \(h\) is continuous on \(\Sset^1_r\). 
By the mean value theorem, there exists a point \(a \in \Sset^1_r\) such that \( h(a)^2=\frac{1}{2\pi r}\int_{\Sset^1_r} h^2\). By the fundamental theorem of calculus we can write 
\begin{equation*}
h(x)^2=h(a)^2 +\int_{t_a}^{t_x}\bigl((h \compose \gamma)^2\bigr)'\dif t
\end{equation*}
 where \(\gamma\) is a smooth path on \(\Sset^1_r\) such that \(\gamma(t_a)=a\), \(\gamma(t_x)=x\) and \(\abs{\gamma'}\equiv 1\). Thus, for any \(C>0\), by using Young's inequality \( 2\abs{h' h}\leq C \varepsilon\abs{h'}^2+\frac{\abs{h}^2}{C\varepsilon}\) and by recalling that \( \varepsilon \leq r\) we find
\[
\norm{h^2}_{L^\infty (\Sset^1_r)}
\le 
\frac{1}{2\pi r}\int_{\Sset^1_r}h^2+\int_{\Sset^1_r}\abs{(h^2)'}\le 
\int_{\Sset^1_r} C\varepsilon \abs{h'}^2 + \Bigl(\frac{1}{2\pi}+\frac{1}{C}\Bigr)\frac{h^2}{\varepsilon}.
\]
The conclusion follows by taking \(C=\frac{1+\sqrt{1+16\pi^2}}{4\pi}\) which solves \(C=\frac{1}{2\pi}+\frac{1}{C}\). 
\end{proof}

The next tool for the proof of \cref{theorem_localisation}, is a lower bound on the Ginzburg--Landau energy on circles at scales larger than \(\varepsilon\).
(When \(\manifold{N} = \Sset^1\), see \cite{Jerrard_1999}*{Proof of Proposition 3.1, Claim 1}).

\begin{lemma}
\label{lemma_lowerbound_GL_circle} 
There exists a constant \(c_1 > 0\), such that for every \(r>0\), for every \(u \in W^{1, 2} (\Sset^1_r, \Rset^\nu)\) such that \(\dist (u,\manifold{N}) <\delta_{\manifold{N}}\) almost everywhere in \(\Sset^1_r\) and for every \(\varepsilon < r\), one has
\[
\int_{\Sset_r^1} \frac{\abs{u'}^2}{2}
+ \frac{F (u)}{\varepsilon^2}
\ge \frac{1}{\frac{\varepsilon}{c_1} + \frac{4 \pi r}{\equivnorm{\Pi_{\manifold{N}} \compose u}^2}}.
\]
\end{lemma}

We remark that the right-hand side in the inequality of \cref{lemma_lowerbound_GL_circle} is an increasing function of \(c_1\). The proof of \cref{lemma_lowerbound_GL_circle} relies on the following elementary inequality.

\begin{lemma}
  \label{lemma_oki8aemoh3quilohK}
For every \(z \in [0, 1]\) and \(\alpha \in (0, +\infty)\), one has
\[
\frac{1 - z}{\alpha} + z^2 \ge \frac{1}{\alpha + 1}
\]

\end{lemma}
\begin{proof}
If \(\alpha \ge \frac{1}{2}\), then the left-hand side in the desired inequality is minimal for \(z=\frac{1}{2\alpha}\in [0,1]\) and we thus obtain that for every \(z\in[0,1]\), \(\frac{1 - z}{\alpha} + z^2 \ge \frac{1}{\alpha}- \frac{1}{4\alpha^2} \ge \frac{1}{\alpha + 1}\); if \(\alpha < \frac{1}{2}\), we have \(\frac{1 - z}{\alpha} + z^2 \ge  2(1-z)+z^2 = 1 +(1-z)^2  \geq 1 \ge \frac{1}{\alpha + 1}\).
\end{proof}

\begin{proof}%
  [Proof of \cref{lemma_lowerbound_GL_circle}]
\resetconstant
Since by assumption \(\dist (u,\manifold{N}) < \delta_{\manifold{N}}\) almost everywhere on \(\Sset^1_r\) and since the function \(F\) satisfies the non-degeneracy assumption \eqref{hyp20}, we have by \cref{lemma_derivativeNearestPointRetraction} and by \eqref{eq_xoishee4IuHi8eefahyeithe},
\[
\frac{\abs{u'}^2}{2} + \frac{F (u)}{\varepsilon^2} \ge \left(1 - \frac{\dist_{\manifold{N}} \compose u}{\delta_{\manifold{N}}}\right) \frac{\abs{(\Pi_{\manifold{N}} \compose u)'}^2}{2}
  + \frac{\abs{(\dist_{\manifold{N}} \compose u)'}^2}{2} + \frac{m_F}{2\varepsilon^2}(\dist_{\manifold{N}} \compose u)^2
  ,
\]
almost everywhere on \(\Sset^1_r\).
If we set \(\theta \defeq\norm{\dist_{\manifold{N}} \compose u}_{L^\infty (\Sset^1_r)} \in [0, \delta_{\manifold{N}}]\),
we have on the one hand, by definition of \(\theta\) and by the characterisation of \( \equivnorm{\Pi_{\manifold{N}} \compose u}\) (see \eqref{eq_Eojei1cooruef6uiP}),
\begin{equation}
\label{eq_eixieghiengazohwed0Cai3r}
\int_{\Sset^1_r} \left(1 - \frac{\dist_{\manifold{N}} \compose u}{\delta_{\manifold{N}}}\right) \frac{\abs{(\Pi_{\manifold{N}} \compose u)'}^2}{2}
 \ge
 \left(1 - \frac{\theta}{\delta_{\manifold{N}}}\right) \frac{\equivnorm{\Pi_{\manifold{N}} \compose u}^2}{4 \pi r},
\end{equation}
and on the other hand, by \cref{lemma_W12_Linfty_1d_epsilon},
\begin{equation}
\label{eq_ulah4Aaghuph2eemohcohkee}
 \int_{\Sset^1_r}\frac{\abs{(\dist_{\manifold{N}} \compose u)'}^2}{2} + \frac{m_F }{2 \varepsilon^2}(\dist_{\manifold{N}} \compose u)^2
 \ge \frac{\theta^2}{\Cl{cst_Pa6ighooxeeChae4l} \varepsilon} ,
\end{equation}
for some constant \(\Cr{cst_Pa6ighooxeeChae4l} > 0\).
It follows thus from \eqref{eq_eixieghiengazohwed0Cai3r} and \eqref{eq_ulah4Aaghuph2eemohcohkee} that if \(c_1 \le \delta_{\manifold{N}}^2/\Cr{cst_Pa6ighooxeeChae4l} \), by applying \cref{lemma_oki8aemoh3quilohK}  with \(z=\frac{\theta }{\delta_\manifold{N}}\),
since \(\theta \le \delta_{\manifold{N}}\)
\[
\begin{split}
\int_{\Sset_r^1} \frac{\abs{u'}^2}{2}
+ \frac{F (u)}{\varepsilon^2}
  &\ge   \left(1 - \frac{\theta}{\delta_{\manifold{N}}}\right)  \frac{\equivnorm{\Pi_{\manifold{N}} \compose u}^2}{4 \pi r}
  + \frac{\theta^2}{\Cr{cst_Pa6ighooxeeChae4l} \varepsilon}\\
  &\ge \frac{c_1}{\varepsilon}
  \biggl(  \Big(1 - \frac{\theta}{\delta_{\manifold{N}}}\Big)  \frac{\equivnorm{\Pi_{\manifold{N}} \compose u}^2 \varepsilon}{4 \pi r c_1}  + \Bigl(\frac{\theta}{\delta_{\manifold{N}}}\Bigr)^2 \biggr)\\
  &\ge \frac{c_1}{\varepsilon ( \frac{4 \pi r c_1}{\equivnorm{\Pi_{\manifold{N}} \compose u}^2 \varepsilon} + 1)}
  = \frac{1}{\frac{\varepsilon}{c_1} + \frac{4 \pi r}{\equivnorm{\Pi_{\manifold{N}} \compose u}^2}}.
\qedhere
  \end{split}
\]
\end{proof}

A last tool is the following lower bound on the energy inside the disk \(B_r\) of radius \(r\) centered at the origin, with non-trivial boundary conditions.

\begin{lemma}
\label{lemma_ball_small_energy_trivial}
There exists a constant \(c_2> 0\),
such that if \(r >0\), if \(u \in W^{1, 2} (B_r, \Rset^\nu)\) satisfies \(\norm{\dist (\tr_{\partial B_r} u(\cdot),\manifold{N})}_{L^\infty(\partial B_r)}< \delta_{\manifold{N}}\) and if
\(
\int_{ B_r} \abs{\Deriv u}^2 \le c_2
\),
then the map \(\Pi_{\manifold{N}} \compose \tr_{\partial B_r} u\) is  homotopic to a constant map.
\end{lemma}
\begin{proof}
  \resetconstant
We have by chain rule and the trace theorem
\[
\norm{\Pi_{\manifold{N}} \compose \tr_{\partial B_r} u}_{\dot{W}^{1/2, 2} (\partial B_r)}
\le
\C
\norm{\tr_{\partial B_r} u}_{\dot{W}^{1/2, 2} (\partial B_r)}
\le \C \norm{\Deriv u}_{L^2 ( B_r)}.
\]
On the other hand, if \(\norm{\Pi_{\manifold{N}} \compose \tr_{\partial B_r} u}_{\dot{W}^{1/2, 2} (\partial B_r)}\) is small enough, then \(\Pi_{\manifold{N}} \compose \tr_{\partial B_r} u\) is homotopic in \(\VMO (\partial B_r, \manifold{N})\) to a constant map see \cite{Brezis_Nirenberg_1995}*{Lemma A.19}.
\end{proof}

\begin{proof}%
[Proof of \cref{theorem_localisation}]
  \resetconstant
We first consider the case where \(u \in \mathcal{C}^2 (\Bar{\Omega},\Rset^\nu)\) with \(\operatorname{tr}_{\partial \Omega}u=g\).
We recall that, in view of \cref{def_energy_extension}, we have assumed that the function \(u\) is extended to a function \(u \in W^{1, 2} (\Omega_{\delta_{\partial \Omega}}, \Rset^\nu)\) in such a way that \(u \in \manifold{N}\) almost everywhere in \(\Omega_{\delta_{\partial\Omega}} \setminus \Omega\) and
\[
\int_{\Omega_{\delta_{\partial\Omega}} \setminus \Omega} \frac{\abs{\Deriv u}^2}{2} \le \C \mathcal{E}^{\mathrm{ext}} (g)
.
\]
By \cref{lemma_Firstlowerbound}, there exist constants \(\Cl{cst_iecaethohp4ahleeK}\) and \(\delta\in(0,\delta_{\manifold{N}})\), depending on \(F\) and \(\Omega\) only, such that
\begin{equation*}
\Esing (g)\,\frac{1}{\delta}
\int_0^{\delta} \Psi \left(\frac{\mathcal{H}^1_\infty (K_s) \, s }{\Cr{cst_iecaethohp4ahleeK}\varepsilon\Esing (g)}\right) \dif s
\le \Cr{cst_iecaethohp4ahleeK}\bigl(\kappa + \Esing (g)\bigr).
\end{equation*}
Since for every \(\tau \in (0, +\infty)\), one has \(\Psi (\tau) = \tau - 1 - \log \tau \ge \frac{\tau}{2} - \log 2\), we deduce that
\begin{equation}\label{eq:majoration_Hausdorff_content}
  \frac{1}{\delta}\int_0^{\delta}
  \frac{\mathcal{H}^1_\infty (K_s) \, s }{2 \Cr{cst_iecaethohp4ahleeK} \varepsilon }\dif s
  \le
  \Cr{cst_iecaethohp4ahleeK}(\kappa + \Esing (g))
  + \Esing (g)\log 2
  \leq 
    (\Cr{cst_iecaethohp4ahleeK}+\log 2)(\kappa + \Esing (g))
\end{equation}
and then, by monotonicity of the Hausdorff content, that
\begin{equation}
\label{geometricOutcome}
    \mathcal{H}^1_{\infty} (K_{\delta})
  \le
    \frac{2}{\delta^2}
    \int_0^{\delta} \mathcal{H}^1_\infty (K_s) \, s \dif s
  \leq
  \Cl{cst_EebaRaaph9zeech6a} \varepsilon(\kappa+\Esing (g)).
\end{equation}
Since the set \(K_{\delta} \subset \Omega \subset \Omega_{\delta_{\partial \Omega}}\) is compact, by definition of the Hausdorff content (\cref{def_hausdorff_content}) and by \cref{lemmaMerging}, there exists
a family of disks \(\mathcal{B}_0\) with disjoint closures such that 
\[
K_{\delta} \subset
\bigcup_{B_\rho(a) \in \mathcal{B}_0} \Bar{B}_\rho(a),
\]
and
\begin{equation}
  \label{eq_weequ7saegh4IevuV}
 \sum_{B_\rho(a)\in \mathcal{B}_0} 2\rho \le 2 \mathcal{H}^1_{\infty} (K_{\delta}) \le 2 \Cr{cst_EebaRaaph9zeech6a} \varepsilon(\kappa+\Esing (g)).
\end{equation}
In particular, if we assume that
\begin{equation}
\label{eq_choopookahphaiChooph6Iot}
2\Cr{cst_EebaRaaph9zeech6a} \varepsilon(\kappa+\Esing (g)) \le \delta_{\partial \Omega}/2,
\end{equation}
and if, without loss of generality, all the disks of \(\mathcal{B}_0\) intersect \(K_{\delta}\), then the disks of \(\mathcal{B}_0\) are all contained in \(\Omega_{\delta_{\partial \Omega}/2}\).
We define
\[
\Bar{s}
\defeq
\sup \biggl\{s \in [0, +\infty) \st
\frac{s }{\log (1 + s)} \varepsilon \biggl(\frac{\Esing (g)}{c_0}  \log \frac{\Cl{cst_ooVai2ohbieFii6ahShaish3}}{\Esing (g) \varepsilon} + \Cl{cst_uec8Zu9aiph9eethe} \kappa \biggr)
\le \frac{\delta_{\partial \Omega}}{4} \biggr\},
\]
where \(c_0=\min\{c_1,c_2\}\) with \(c_1,c_2\)  defined in \Cref{lemma_lowerbound_GL_circle}, \Cref{lemma_ball_small_energy_trivial},
\begin{align}
\label{eq_uyeu2raefeiNgiequohxeeta}
\Cr{cst_ooVai2ohbieFii6ahShaish3}
= e^{2 \Cr{cst_EebaRaaph9zeech6a} c_0} & &
& \text{ and } &
 \Cr{cst_uec8Zu9aiph9eethe} & \defeq \frac{1}{c_0} + 2 \Cr{cst_EebaRaaph9zeech6a}.
\end{align}
We claim that for every \(s \in [0, \Bar{s})\), there exists a collection of disks \(\mathcal{B} (s)\) such that
\begin{enumerate}[(a)]
  \item \label{it_mu3einohch1Ied2Ee} the closure of the disks in \(\mathcal{B} (s)\) are disjoint and contained in \(\Omega_{\delta_{\partial \Omega}}\),
  \item \label{it_yeejaCaicei4daen8} if \(t \in [0, s)\), then \(\bigcup_{B_\sigma(b)\in \mathcal{B} (t)} B_\sigma(b) \subset \bigcup_{B_\rho(a) \in \mathcal{B} (s)} B_\rho(a)\),
  \item \label{it_wah0taNgaitu4Feez} for every \(B_\rho(a)\in \mathcal{B} (s)\),
  \(\dist_{\manifold{N}} \compose \tr_{\partial B_\rho(a)} u < \delta_{\manifold{N}}\) and
  \[
\rho \ge \frac{\varepsilon s}{c_0} \Esing (\tr_{\partial B_\rho(a)} u),
  \]
  \item  \label{it_eivai0aiphe5Anoed} for every \(B_\rho(a) \in \mathcal{B} (s)\),
  \[
  \int_{B_\rho (a)} \frac{\abs{\Deriv u}^2}{2} + \frac{F (u)}{\varepsilon^2}
  \ge \frac{c_0}{\varepsilon} \biggl( \rho \frac{\log (1 + s)}{s} - \sum_{\substack{ B_\sigma(b)\in \mathcal{B}_0\\ B_\sigma(b) \subset B_\rho(a)}}  \sigma\biggr).
  \]
\end{enumerate}

In order to construct this collection of disks \(\mathcal{B} (s)\)  for \(s \in [0, \Bar{s})\) we first set \(\mathcal{B} (0) \defeq \mathcal{B}_0\).
We have showed that \eqref{it_mu3einohch1Ied2Ee} holds for \(s = 0\) provided \eqref{eq_choopookahphaiChooph6Iot} holds; the assertion \eqref{it_yeejaCaicei4daen8} holds trivially when \(s=0\) and \eqref{it_wah0taNgaitu4Feez} holds since every connected component of \(K_{\delta}\) is contained in a unique disk of \(\mathcal{B}_0\). Finally for \eqref{it_eivai0aiphe5Anoed}, we observe that when \(s \to 0\) the limit of the right-hand side  vanishes.
By continuity, we can take \(\mathcal{B} (s) = \mathcal{B}(0)\) for \(s > 0\) close enough from \(0\).
We assume now that the assertions \eqref{it_mu3einohch1Ied2Ee}, \eqref{it_yeejaCaicei4daen8}, \eqref{it_wah0taNgaitu4Feez} and \eqref{it_eivai0aiphe5Anoed} are satisfied for some \(s_* \in (0,\Bar{s})\).
We define then the set of disks
\[
\mathcal{B}_* \defeq \{B_\rho (a) \in \mathcal{B} (s_*) \st \text{equality holds in \eqref{it_wah0taNgaitu4Feez}} \}.
\]
These disks are referred to as minimising disks. 

The first step consists in an expansion phase: we let the radii of the minimising disks grow in the following way. We define, for \(s\geq s_*\)
\[
\mathcal{B} (s) \defeq \{B_{\rho s/s_*} (a) \st B_\rho (a) \in \mathcal{B}_* \} \cup \left( \mathcal{B} (s_*) \setminus \mathcal{B}_* \right)
\]
and the number
\begin{multline*}
s^* \defeq \sup \bigl\{\sigma \in [s_*, \Bar{s}] \st \text{for each \(s \in [s_*, \sigma)\) \eqref{it_mu3einohch1Ied2Ee} holds,}\\
\shoveright{\text{
    strict inequality holds in \eqref{it_wah0taNgaitu4Feez} for each \(B_\rho (a) \in \mathcal{B} (s_*) \setminus \mathcal{B}_*\)}\quad}\\
\text{
  and \(\varepsilon \not \in (\rho, \Bar{s}\rho/s_*)\)
  for each \(B_\rho (a) \in \mathcal{B} (s)\)} \bigr\}.
\end{multline*}
We check that for \(s_*\leq s \leq \Bar{s}\) the families of disks \(\mathcal{B}(s)\) satisfy  \eqref{it_mu3einohch1Ied2Ee}, \eqref{it_yeejaCaicei4daen8} , \eqref{it_wah0taNgaitu4Feez},  \eqref{it_eivai0aiphe5Anoed}. By construction, the assertions \eqref{it_mu3einohch1Ied2Ee} and \eqref{it_wah0taNgaitu4Feez} hold for every \(s \in [s_*, s^*)\). Property \eqref{it_yeejaCaicei4daen8} is also satisfied.
We now prove \eqref{it_eivai0aiphe5Anoed}. If \( B_\rho(a) \in \mathcal{B}(s_*)\setminus \mathcal{B}_*\),  \eqref{it_eivai0aiphe5Anoed} is true by assumption. If \(B_\rho (a) \in \mathcal{B}_*\), 
we first consider the case where \(\rho s^*/s_* \le \varepsilon\). 
Then since equality holds in \eqref{it_wah0taNgaitu4Feez} and since maps homotopic to a constant have zero singular energy \(\Esing\), the map
\(\tr_{\partial B_{\rho} (a)} \Pi_{\manifold{N}} \compose u\) is not homotopic to a constant. Hence
for every \(s \in [s_*, s^*]\), the map \(\tr_{\partial B_{\rho s/s_*} (a)} \Pi_{\manifold{N}} \compose u\) is not homotopic to a constant and satisfies by \cref{lemma_ball_small_energy_trivial}, since \(\rho s/s_* \le \varepsilon\),
\[
\begin{split}
\int_{B_{\rho s/s_*} (a)} \!\!\abs{\Deriv u}^2 + \frac{F (u)}{\varepsilon^2}
\ge \int_{B_{\rho s/s_*} (a)} \abs{\Deriv u}^2
\ge c_0 &\ge \frac{c_0 \rho s}{\varepsilon s^*}\\
&\ge \frac{c_0}{\varepsilon} \biggl( \frac{\rho s}{s_*} \frac{\log (1 + s)}{s} - \sum_{\substack{B_\sigma (b) \in \mathcal{B}_0\\ B_{\sigma} (b) \subset B_{s \rho/s_*} (a)}} \sigma\biggr).
\end{split}
\]
In the last inequality we have used the inequality \(\log(1+s) \leq s\). 
If \(B_\rho(a) \in \mathcal{B}_*\) and \(\rho s^*/s_* >\varepsilon\)  we have, from the definition of \(s^*\), that \(\rho >\varepsilon\). Then, if \(B_\rho(a) \in \mathcal{B}_*\),
we apply \cref{lemma_lowerbound_GL_circle}, since
\(\dist_{\manifold{N}} \compose u < \delta_\manifold{N}\) on \(B_{\rho s/s_*} (a) \setminus B_{\rho} (a)
\subset \Omega_{\delta_{\partial \Omega}} \setminus K_{\delta_{\manifold{N}}}\) and  since \(\Esing(\Pi_{\manifold{N}} \compose \tr_{\partial B_{t} (a)}u) = \frac{\rho c_0}{s_* \varepsilon}\) for \(t\in (\rho, \frac{\rho s}{s_*})\)
\begin{equation}
\label{eq_gai7gophaingeefaqu2Eiv7u}
\begin{split}
 \int_{B_{\rho s/s_*} (a) \setminus B_{\rho} (a)}\frac{\abs{\Deriv u}^2}{2} + \frac{F (u)}{\varepsilon^2}
&\ge \int_\rho^{\rho s/s_*} \frac{1}{\frac{\varepsilon}{c_0} + \frac{t}{\Esing(\Pi_{\manifold{N}} \compose \tr_{\partial B_{t} (a)}u)}} \dif t\\
& = \frac{c_0 \rho}{\varepsilon s_*}
\int_\rho^{\rho s/s_*} \frac{1}{\frac{\rho}{s_*} + t} \dif t = \frac{c_0 \rho}{\varepsilon s_*} \log \frac{1 + s}{1 + s_*}\\
&\ge \frac{c_0 \rho}{\varepsilon} \left(\frac{\log(1 + s)}{s} - \frac{\log (1 + s_*)}{s_*}\right).
\end{split}
\end{equation}
 We use that from \eqref{it_eivai0aiphe5Anoed}, for any \(B_\rho(a) \in \mathcal{B}(s_*)\) we have
\begin{equation}
\label{eq_xohf0aiThie3sheikaiF0wai}
\rho\frac{\log(1+s_*)}{s_*}\leq \sum_{\substack{B_\sigma (b) \in \mathcal{B}_0\\ B_{\sigma} (b) \subset B_\rho (a)}} \sigma +\frac{\varepsilon}{c_0}\left(\int_{B_\rho (a)} \frac{\abs{\Deriv u}^2}{2} + \frac{F (u)}{\varepsilon^2} \right)
\end{equation}
and therefore by \eqref{eq_gai7gophaingeefaqu2Eiv7u} and \eqref{eq_xohf0aiThie3sheikaiF0wai}
\[
\int_{B_{\rho s/s_*} (a)} \frac{\abs{\Deriv u}^2}{2} + \frac{F (u)}{\varepsilon^2}
\ge \frac{c_0}{\varepsilon} \biggl( \rho \frac{\log (1 + s)}{s} - \sum_{\substack{B_\sigma (b) \in\mathcal{B}_0\\ B_{\sigma} (b) \subset B_\rho (a)}} \sigma\biggr).
\]
Moreover, we deduce from \eqref{it_eivai0aiphe5Anoed}, from our assumption \eqref{eq_neephohy5ne5eLiu9OM1biim} and from \eqref{eq_weequ7saegh4IevuV} that
\begin{equation}
  \label{eq_eepai5Doozie6quoo}
\begin{split}
\sum_{B_\rho (a) \in \mathcal{B} (s)} \rho
&\le
\frac{s}{\log (1 + s)}\biggl(
\frac{\varepsilon}{c_0}
\int_{\Omega} \frac{\abs{\Deriv u}^2}{2} + \frac{F (u)}{\varepsilon^2}
+ \sum_{B_\sigma (b) \in \mathcal{B}_0} \sigma
\biggr)\\
& \le
\frac{s }{\log (1 + s)} \varepsilon \biggl(\frac{\Esing (g)}{c_0}  \log \frac{\Cr{cst_ooVai2ohbieFii6ahShaish3}}{\Esing (g) \varepsilon} + \Cr{cst_uec8Zu9aiph9eethe} \kappa \biggr),
\end{split}
\end{equation}
in view of the definition of \(\Cr{cst_ooVai2ohbieFii6ahShaish3}\) and \(\Cr{cst_uec8Zu9aiph9eethe}\) in \eqref{eq_uyeu2raefeiNgiequohxeeta}. Thus we find a collection the desired collection of disks \(\mathcal{B}(s)\) for \(0\leq s <s^*\).
In order to define \(\mathcal{B} (s^*)\), we  set
\[
\mathcal{B}^* \defeq \{B_{\rho s^*/s_*} (a) \st B_\rho (a) \in \mathcal{B}_* \} \cup \mathcal{B} (s_*) \setminus \mathcal{B}_*.
\]
We first note that by \eqref{eq_eepai5Doozie6quoo}, since \(s < \Bar{s}\), we have \(\bigcup_{B_\rho (a) \in \mathcal{B}^*} B_\rho (a) \subset \Omega_{\delta_{\manifold{N}}/2}\). We also note that the family \(\mathcal{B}^*\) satisfies all the desired properties except that some disks can have  boundaries that intersect each other. If this is the case we perform then a disk merging procedure by \cref{lemmaMerging} and we define \(\mathcal{B} (s^*)\) to be the resulting disk collection.
By \eqref{it_wah0taNgaitu4Feez}, for every \(B_{\rho} (a) \in \mathcal{B} (s^*)\), we have
\[
\rho = \sum_{\substack{B_{\sigma} (b) \in \mathcal{B}^*\\ B_{\sigma} (b) \subset B_{\rho} (a)}} \sigma
\ge \sum_{\substack{B_{\sigma} (b) \in \mathcal{B}^*\\ B_{\sigma} (b) \subset B_{\rho} (a)}} \frac{\varepsilon s^*}{c_0}
\Esing( \tr_{\partial B_{\sigma} (b)} v)
\ge \frac{\varepsilon s^*}{c_0} \Esing( \tr_{\partial B_{\rho} (b)} v),
\]
so that assertion \eqref{it_wah0taNgaitu4Feez} still holds for the modified collection of disks.
We also have, since \(\mathcal{B}^*\) satisfies \eqref{it_eivai0aiphe5Anoed},
\[
\begin{split}
\int_{B_{\rho} (a)} \frac{\abs{\Deriv u}^2}{2} + \frac{F (u)}{\varepsilon^2}
  &\ge
  \sum_{\substack{B_{\sigma} (b) \in \mathcal{B}^*\\ B_{\sigma} (b) \subset B_{\rho} (a)}}
  \int_{B_{\sigma} (b)} \frac{\abs{\Deriv u}^2}{2} + \frac{F (u)}{\varepsilon^2}\\
  &\ge  \sum_{\substack{B_{\sigma} (b) \in \mathcal{B}^*\\ B_{\sigma} (b) \subset B_{\rho} (a)}} \frac{c_0}{\varepsilon} \biggl( \sigma \frac{\log (1 + s)}{s} - \sum_{\substack{B_\tau (c) \in \mathcal{B}_0\\ B_{\tau} (c) \subseteq B_{\sigma} (b)}} \tau \biggr)\\
  &= \frac{c_0}{\varepsilon} \biggl( \rho \frac{\log (1 + s)}{s} -  \sum_{\substack{B_\tau (c) \in \mathcal{B}_0\\ B_{\tau} (c) \subseteq B_{\rho} (a)}} \tau \biggr),
\end{split}
\]
and hence assertion \eqref{it_eivai0aiphe5Anoed} also holds for the modified collection of disks. We can then continue alternatively with  expansion phases and merging steps.
Since at each step either the number of disks decreases or the number of disks with equality in \eqref{it_wah0taNgaitu4Feez} increases, we fill the full announced interval of \([0, \Bar{s})\) in a finite number of steps.

In order to conclude if
\begin{equation}
\label{eq_tixua3iebaeRoFaiguw8puaJ}
\eta \ge 2 \varepsilon \Esing (g)/(\gamma c_0),
\end{equation}
we set
\[
\Tilde{s} \defeq  \frac{c_0 \gamma \eta}{\varepsilon \Esing (g)} - 1 \geq 1
\]
so that,
\[
\begin{split}
  \frac
  {\Tilde{s}}
  {\log (1 + \Tilde{s})}&
  \varepsilon \biggl(\frac{\Esing (g)}{c_0}  \log \frac{\Cr{cst_ooVai2ohbieFii6ahShaish3}}{\Esing (g) \varepsilon} + \Cr{cst_uec8Zu9aiph9eethe} \kappa \biggr)\\
&\le
  \frac
  {\frac{c_0 \eta \gamma}{\varepsilon \Esing (g)} - 1}
  {\log \frac{c_0\gamma \eta}{ \varepsilon \Esing (g)}}
  \varepsilon \biggl(\frac{\Esing (g)}{c_0}  \log \frac{\Cr{cst_ooVai2ohbieFii6ahShaish3}}{\Esing (g) \varepsilon} + \Cr{cst_uec8Zu9aiph9eethe} \kappa \biggr)
 \le
\gamma \eta
  \frac
  {\log \frac{\Cr{cst_ooVai2ohbieFii6ahShaish3} e^{\Cr{cst_uec8Zu9aiph9eethe} \kappa}}{\Esing (g) \varepsilon}}
    {\log  \frac{c_0 \gamma \eta}{\Esing(g)\varepsilon} }
    \le  \eta,
\end{split}
\]
provided
\begin{equation}
\label{eq_oWooghahngoich8veedooxu9}
 (\Cr{cst_ooVai2ohbieFii6ahShaish3} e^{\Cr{cst_uec8Zu9aiph9eethe} \kappa})^\gamma
 (\Esing (g) \varepsilon)^{1 - \gamma} \le c_0 \gamma \eta.
\end{equation}
It follows that if \(\eta \le \delta_{\partial \Omega}/4\), \(\Tilde{s} \in [0, \Bar{s}]\).
Moreover, we have by \eqref{eq_eepai5Doozie6quoo},
\begin{equation}
\label{eq_oqueeShae9bee2Obi7Xuofai}
\sum_{B_\rho (a) \in \mathcal{B} (\Tilde{s})} \rho \le  \eta.
\end{equation}
We now define the collection \(\mathcal{B} \defeq \{ B_\rho (a) \in \mathcal{B} (\Tilde{s}) \st \Esing (\Pi_\mathcal{N} \compose\tr_{\partial B_r (a)} u) > 0\}\).
We then have for every \(B_\rho (a) \in \mathcal{B}\), by \eqref{it_wah0taNgaitu4Feez}
and by \eqref{it_eivai0aiphe5Anoed},
\begin{equation*}
\int_{B_\rho (a)} \frac{\abs{\Deriv u}^2}{2} + \frac{F (u)}{\varepsilon^2}
 \ge  \Esing (\Pi_\mathcal{N} \compose\tr_{\partial B_\rho (a)} u) \log \left(\frac{c_0 \gamma \eta}{\varepsilon  \Esing (g)}\right)  - \sum_{\substack{B_{\sigma} (b) \in \mathcal{B}_0\\ B_{\sigma} (b) \subseteq B_{\rho} (a)}} \sigma.
\end{equation*}
Hence, for every subcollection of disks \(\mathcal{B}' \subset \mathcal{B}\),  by summing and by \eqref{eq_weequ7saegh4IevuV},
we obtain
\begin{multline}
  \label{eq_shou1eejaij3MooHe}
    \int_{\bigcup_{B_\rho (a) \in \mathcal{B}'} B_\rho (a)} \frac{\abs{\Deriv u}^2}{2} + \frac{F (u)}{\varepsilon^2}\\
    \ge \sum_{B_\rho (a) \in \mathcal{B}'} \Esing (\Pi_\mathcal{N} \compose \tr_{\partial B_\rho (a)} u) \log \left(\frac{c_0 \gamma \eta}{\varepsilon \Esing (g)}\right) -  2 \Cr{cst_EebaRaaph9zeech6a}(\Esing (g) + \kappa) \varepsilon.
\end{multline}
By \eqref{eq_paih1ieb8eiNgo4fumaequ6u}, our assumption \eqref{eq_neephohy5ne5eLiu9OM1biim} and by \eqref{eq_shou1eejaij3MooHe}, we deduce that
\begin{equation}
\label{eq_xeeNgeo2eishie6jah9Bool2}
\begin{split}
\frac{\syst (\manifold{N})^2}{4 \pi} \# \mathcal{B}
&\le
 \sum_{B_\rho (a) \in \mathcal{B}} \Esing (\Pi_{\manifold{N}} \compose\tr_{\partial B_\rho (a)} u)
 \\
 &\le
 \Esing (g)
 + \frac{(2 \Cr{cst_EebaRaaph9zeech6a} \varepsilon + \log \frac{1}{c_0 \gamma \eta})\Esing (g)
 + (2\Cr{cst_EebaRaaph9zeech6a} \varepsilon + 1) \kappa}
 {
\log \frac{c_0 \gamma \eta}{\varepsilon \Esing (g)}
 }.
 \end{split}
\end{equation}
The proposition is proved when \(u \in \mathcal{C}^2 (\Bar{\Omega})\),
with \eqref{it_OoPhah4ouneukoot6thos4cu} following from \eqref{eq_oqueeShae9bee2Obi7Xuofai}, with a constant \(C\)  in the conditions coming from the conditions \eqref{eq_choopookahphaiChooph6Iot}, \eqref{eq_tixua3iebaeRoFaiguw8puaJ} and \eqref{eq_oWooghahngoich8veedooxu9};
 the conclusion \eqref{it_uNamohyo4Ciece8aegah1que} follows from  \eqref{eq_shou1eejaij3MooHe} and  \eqref{it_dieYair4co5Ui2oowipheisa} from \eqref{eq_xeeNgeo2eishie6jah9Bool2}.

In the general case where \(u \in W^{1, 2} (\Omega, \Rset^\nu)\),
we first consider the case where the function \(F\) is bounded and continuous, so that the Ginzburg--Landau functional is continuous for the strong convergence in \(W^{1,2}\).
We consider a sequence \((u_n)_{n \in \Nset}\) in \(\mathcal{C}^2 (\Bar{\Omega})\) converging strongly to \(u\) in \(W^{1, 2} (\Omega, \Rset^\nu)\).
We apply the proposition to \(u_n\) and let \(\mathcal{B}_n\) be the associated disks.
By \eqref{it_dieYair4co5Ui2oowipheisa}, up to a subsequence, \((\Pi_{\manifold{N}} \compose \tr_{\partial B_\rho (a)}u_n)_{B_\rho (a) \in \mathcal{B}_n}\) can be chosen to remain in the same homotopy class and  \(\# \mathcal{B}_n\) can be chosen to be constant.

If \(F\) is not bounded, then we apply the proposition to a sequence of  bounded functions \(\Tilde{F}_l \in \mathcal{C}(\Rset^\nu, [0, +\infty))\) such that \(\Tilde{F}_l \le F\), \(\Tilde{F}_l\) converges to \(F\) everywhere in \(\Rset^\nu\) and \(\Tilde{F}_l = F\) on a neighbourhood of \(\manifold{N}\). Since the constants only depend on the behaviour of \(F\) in a neighbourhood of \(\manifold{N}\), the constants can be taken independently on \(l \in \Nset\) and the conclusion follows by Leguesgue's monotone convergence theorem.
\end{proof}

\section{Energy convergence}
\label{section_energy_convergence}
We investigate  first in \S \ref{section_aejie9eiLoh9haire} 
the convergence of sequences whose Ginzburg--Landau energy satisfies a logarithmic bound. 
This bound being satisfied for minimisers in view of \cref{proposition_geometricUpperBound}, 
we apply this result to minimisers and get additional properties in \S \ref{theorem_main1}. 

In this section \(\Omega\) is a Lipschitz bounded domain and \(F \in \mathcal{C} (\Rset^\nu, [0,+\infty))\) satisfies \(F^{-1}(\{0\}) = \manifold{N}\) and \eqref{hyp20}.

\subsection{Convergence of bounded sequences}
\label{section_aejie9eiLoh9haire}

The main result about convergence of sequences whose Ginzburg--Landau energy satisfies a logarithmic bound is

\begin{theorem}
\label{theorem_compactnessweak}

Let \(g\in W^{1/2,2}(\partial\Omega,\manifold{N})\), \((u_n)_{n \in \Nset}\) be a sequence in \(W^{1,2}(\Omega,\Rset^\nu)\) with \(\tr_{\partial \Omega}u_n=g\)  and \((\varepsilon_n)_{n \in \Nset}\)
be a sequence in \((0, +\infty)\) converging to \(0\) such that
\begin{equation}
\label{eq_borne_energie}
\sup_{n \in \Nset} \int_{\Omega}  \frac{\abs{\Deriv u_n}^2}{2} + \frac{F (u_n)}{\varepsilon_n^2}
- \Esing (g) \log \frac{1}{\varepsilon_n}
< + \infty .
\end{equation}
Then up to a subsequence, there exists a map \(u_* \in W^{1, 2}_{\mathrm{ren}} (\Omega, \manifold{N})\), such that if we write \(\operatorname{sing} (u_*) =  \{(a_1,\gamma_1), \dotsc, (a_k,\gamma_k)\}\), we have
\begin{enumerate}[(i)]
  \item \label{it_zeqpfibhaedq}
 the sequence \((u_n)_{n \in \Nset}\) converges to \(u_*\) weakly in \(W^{1, 2}_{\textrm{loc}} (\Bar{\Omega} \setminus \{a_1, \dotsc, a_k\}, \Rset^\nu)\) and almost everywhere in \(\Omega\),
  \item \label{it_aebzherçbahçfh}\(\Esing (g) = \sum_{i = 1}^k \frac{\equivnorm{\gamma_i}^2}{4 \pi}\),
 \item \label{it_iji8Jienie1cad0ie} \(\displaystyle \sup_{n \in \Nset} \int_{\Omega} \frac{
 \abs{\Deriv (\dist_{\manifold{N}} \compose u_n)}^2}{2} + \frac{F (u_n)}{\varepsilon_n^2}
 + \sup_{t > 0} t^2 \mathcal{L}^2 (\abs{\Deriv u_n}^{-1}([t, +\infty))) < +\infty\),
 \item \label{it_eShexuRaiChe7oochae6ooda} one has, narrowly as measures on \(\Omega\),
 \[
 \frac{\abs{\Deriv u_n}^2}{2 \log \frac{1}{\varepsilon_n}}\underset{n\to\infty}{\weakto}\sum_{i = 1}^k \frac{\equivnorm{\gamma_i}^2}{4 \pi}\delta_{a_i},
 \]
\item \label{it_Eejai5sungeiF7sho} \(\displaystyle
   \mathcal{E}^\mathrm{ren} (u_*) + \mathcal{Q}_F(u_*)
   \le \liminf_{n \to \infty}
   \int_{\Omega} \frac{\abs{\Deriv u_n}^2}{2} + \frac{F (u_n)}{\varepsilon_n^2}
   - \Esing (g) \log \frac{1}{\varepsilon_n}
 \), 
 
 \item \label{it_vaid7ieb1rub4Che6} for every \(\rho \in (0, \Bar{\rho} (a_1, \dotsc, a_k))\),
 \[
     \mathcal{E}^\mathrm{ren} (u_*)
   +
     \mathcal{Q}_F(u_*)
 \le
  \int_{\Omega \setminus \bigcup_{i = 1}^k B_\rho (a_i)} \frac{\abs{\Deriv u^\ast}^2}{2}
  +
 \liminf_{n \to \infty}
       \int_{\bigcup_{i = 1}^k B_\rho (a_i)} \frac{\abs{\Deriv u_n}^2}{2} + \frac{F (u_n)}{\varepsilon_n^2}
          - \Esing (g) \log \frac{1}{\varepsilon_n}.
 \]
\end{enumerate}
\end{theorem}

\Cref{theorem_compactnessweak} follows from \cref{theorem_localisation} as in \cites{Sandier_1998,Jerrard_1999}.

\begin{proof}[Proof of \cref{theorem_compactnessweak}]
\resetconstant

The boundedness assertion \eqref{it_iji8Jienie1cad0ie} follows immediately from the lower bound for the Ginzburg-Landau energy of \cref{proposition_Firstlowerbound}.
Since \(g \in W^{1/2, 2} (\partial \Omega, \manifold{N})\),
there exists a map \(w \in W^{1, 2} (\Omega_{\delta_{\partial \Omega}}, \manifold{N})\), with \(\Omega_{\delta_{\partial \Omega}}\) defined in \eqref{def:tubular_neighboorhood_of_boundary} , such that \(\tr_{\partial \Omega} w = g\). 
For each \(n \in \Nset\), we define the function \(\Bar{u}_n \in W^{1, 2} (\Omega_{\delta_{\partial \Omega}}, \Rset^\nu)\)
in such a way that \(\Bar{u}_n \vert_{\Omega} = u_n\) and \(\Bar{u}_n \vert_{\Omega_{\delta_{\partial \Omega}} \setminus \Omega} = w\).

We let \(\Cl{cst_ungaiB8aiXai8thaiv5aghoo} \in (0, +\infty)\) be a constant as in \cref{theorem_localisation} and we consider a sequence \((\eta_p)_{p \in \Nset}\) in \((0, +\infty)\) converging to \(0\). Since whatever the constants \(\kappa\in(0,+\infty)\), \(\gamma\in(0,1)\), and for each \(p \in \Nset\), there exists \(n_p\in\Nset\) such that for every \(n \ge n_p\),
\begin{align*}
\Cr{cst_ungaiB8aiXai8thaiv5aghoo} e^{\gamma  \Cr{cst_ungaiB8aiXai8thaiv5aghoo} \kappa}
(\Esing (g)\varepsilon_n)^{1-\gamma} & \le \gamma \eta_p, &
\Cr{cst_ungaiB8aiXai8thaiv5aghoo} \Esing (g)\varepsilon_n &\le \frac 12\gamma \eta_p,
&\text{ and }&
\Cr{cst_ungaiB8aiXai8thaiv5aghoo} \kappa \varepsilon_n \le 1,
\end{align*}
we have by \cref{theorem_localisation} and by the energy bound \eqref{eq_borne_energie} a finite collection \(\mathcal{B}_{n, p}\) of disjoint disks of radii less than \(\eta_p\) such that for every \(n\ge n_p\) and every \(B\in \mathcal{B}_{n, p}\), we have \( \Bar{B} \subset \Omega_{\delta_{\partial \Omega}}\) and \(\dist_{\manifold{N}} \compose \tr_{\partial B} \Bar{u}_n < \delta_{\manifold{N}}\), such that for some constant \(\Cl{nfdoksaiewoqn}>-\infty\) independant of \(p,n\) (which may depend on \(\gamma,\kappa\) and \(\Esing (g)\)), we have in view of \eqref{it_uNamohyo4Ciece8aegah1que} and \eqref{it_dieYair4co5Ui2oowipheisa} in \cref{theorem_localisation}, for every \(\mathcal{B}' \subseteq \mathcal{B}_{n, p}\)
\begin{equation}
\label{equation_ti9Iegheegei7uephaiquahr}
   \int_{\bigcup_{B \in \mathcal{B}'} B}
  \frac{\abs{\Deriv \Bar{u}_n}^2}{2} + \frac{F (\Bar{u}_n)}{\varepsilon_n^2}
  \geq \sum_{B \in \mathcal{B}'}
  \Esing (\Pi_{\manifold{N}} \compose \tr_{\partial B} \Bar{u}_n) \log \frac{\eta_p}{\varepsilon_n}-\Cr{nfdoksaiewoqn},
\end{equation}
with \(\Esing (\Pi_{\manifold{N}} \compose \tr_{\partial B} \Bar{u}_n)\)
and such that the maps \((\Pi_{\manifold{N}} \compose \tr_{\partial B} \Bar{u}_n)_{B\in \mathcal{B}_{n,p}}\) form a topological resolution of \(g\).

Since the manifold \(\manifold{N}\) is compact, in view of  \cref{proposition_equiv_norm_discrete}, the set \(\{\equivnorm{\gamma} \st \gamma \in \VMO (\Sset^1, \manifold{N})\}\) is discrete, and thus there exists \(\delta > 0\) such that if \((\gamma_1, \dotsc, \gamma_\ell)\) is a topological resolution of \(g\),
and
\(
 \sum_{i = 1}^{\ell} \Esing (\gamma_i)
\le \Esing (g) + \delta,
\)
then
\(
\sum_{i = 1}^{\ell} \Esing (\gamma_i)
= \Esing (g).
\)
By \cref{theorem_localisation} \eqref{it_dieYair4co5Ui2oowipheisa}, and taking \(n_p\) larger if necessary, we can thus assume that
\[
\# \mathcal{B}_{n, p} \frac{\syst(\manifold{N})^2}{4 \pi}
\le
\sum_{B\in \mathcal{B}_{n, p}}
\Esing (\Pi_{\manifold{N}} \compose \tr_{\partial B} \Bar{u}_n)
= \Esing (g),
\]
so that \((\Pi_{\manifold{N}} \compose \tr_{\partial B} \Bar{u}_n)_{B \in \mathcal{B}_{n, p}}\) is a minimal topological resolution of \(\tr_{\partial \Omega} u_n=g\). By \eqref{equation_ti9Iegheegei7uephaiquahr}, it follows that
\begin{equation}
\label{equation_foepkmflmfoewapfmdls}
  \int_{\bigcup_{B\in \mathcal{B}_{n, p}} B}
  \frac{\abs{\Deriv \Bar{u}_n}^2}{2} + \frac{F (\Bar{u}_n)}{\varepsilon_n^2}
  \geq
  \Esing (g) \log \frac{\eta_p}{\varepsilon_n}-\Cr{nfdoksaiewoqn},
\end{equation}
which with our assumption \eqref{eq_borne_energie}, yields
\begin{equation}
  \label{eq_localisation2}
  \int_{\Omega_{\delta_{\partial \Omega}}\setminus \bigcup_{B\in \mathcal{B}_{n, p}} B}
    \frac{\abs{\Deriv \Bar{u}_n}^2}{2} + \frac{F (\Bar{u}_n)}{\varepsilon_n^2}
  \leq
    \Esing (g)\log \frac{1}{\eta_p}+\Cl{cst_ooHei2tahgae8Tei3}.
\end{equation}

Let \(\mathcal{C}_{n, p}\) denote the set of centres of the disks in \(\mathcal{B}_{n, p}\).
Up to a subsequence in \(n\) and by a diagonal argument, we can assume that for each \(p \in \Nset\), the sequence \((\mathcal{C}_{n, p})_{n\in\mathbb{N}}\) converges in Hausdorff distance to a finite set \(\mathcal{C}_p\) in \(\Bar{\Omega}\) of cardinality at most \( 4 \pi \Esing (g)/\operatorname{sys} (\manifold{N})^2\).
Taking a subsequence, we can assume further that \((\mathcal{C}_p)_{p \in \Nset}\) converges in Hausdorff distance to a finite set \(\mathcal{C} = \{a_1, \dotsc, a_k\} \subset \Bar{\Omega}\), with \(k \le 4 \pi \Esing (g)/\operatorname{sys} (\manifold{N})^2\).
(The sets \(\mathcal{C}_{p, n}\), \(\mathcal{C}_p\) and \(\mathcal{C}\) being possibly empty, we understand that a sequence converges to the empty set in Hausdorff distance whenever it is eventually a constant sequence of empty sets.)

Moreover, 
we have for every \(p, q \in \Nset\), for \(n \in \Nset\) large enough by \eqref{equation_ti9Iegheegei7uephaiquahr},
\begin{equation}
\begin{split}
\int_{\Omega_{\delta_{\partial \Omega}}}\frac{\abs{\Deriv \Bar{u}_n}^2}{2} + \frac{F (\Bar{u}_n)}{\varepsilon_n^2}
&\ge \sum_{B \in \mathcal{B}_{n, p}}
  \int_{B}
  \frac{\abs{\Deriv \Bar{u}_n}^2}{2} + \frac{F (\Bar{u}_n)}{\varepsilon_n^2}
  + \sum_{\substack{B \in \mathcal{B}_{n, q}\\
  B \cap \bigcup_{B' \in  \mathcal{B}_{n, p}}B' = \emptyset}}
  \int_{B}
  \frac{\abs{\Deriv \Bar{u}_n}^2}{2} + \frac{F (\Bar{u}_n)}{\varepsilon_n^2}\\
  &\geq
  \sum_{B \in \mathcal{B}_{n, p}}
  \Esing (\Pi_{\manifold{N}} \compose \tr_{\partial B} \Bar{u}_n) \log \frac{\eta_p}{\varepsilon_n}\\
  &\qquad 
  + \sum_{\substack{B \in \mathcal{B}_{n, q}\\
  B \cap \bigcup_{B' \in  \mathcal{B}_{n, p}}B' = \emptyset}}
  \Esing (\Pi_{\manifold{N}} \compose \tr_{\partial B} \Bar{u}_n) \log \frac{\eta_q}{\varepsilon_n}
  - 2 \Cr{nfdoksaiewoqn}\\
  &\geq
  \Esing (g) \log \frac{\eta_p}{\varepsilon_n}\\
& \qquad  + \# \{B \in \mathcal{B}_{n, q} \st  B \cap {\textstyle \bigcup_{B ' \in  \mathcal{B}_{n, p}}B' = \emptyset}\} \tfrac{\syst(\manifold{N})^2}{4 \pi} \log \frac{\eta_q}{\varepsilon_n} - 2\Cr{nfdoksaiewoqn},
  \end{split}
\end{equation}
from which it follows that if \(n\) is large enough,
we have \(\dist_{\mathcal{H}} (\mathcal{C}_{n, p}, \mathcal{C}_{n, q}) \le  \eta_p + \eta_q\).
Passing to the limit, we have \(\dist_{\mathcal{H}} (\mathcal{C}_{p}, \mathcal{C}_{q}) \le  \eta_p + \eta_q\) and then \(\dist_{\mathcal{H}} (\mathcal{C}_{p}, \mathcal{C})  \le \eta_p\).
In particular, for every \(p \in \Nset\), for \(n \in \Nset\) large enough, we have \(\dist_{\mathcal{H}} (\mathcal{C}_{n, p}, \mathcal{C})  \le 2\eta_p\), and thus 
\begin{gather}
\label{eq_eePhahw5ohng0zie0}
\int_{\Omega_{\delta_{\partial \Omega}}\setminus \bigcup_{i = 1}^k B_{3 \eta_p} (a_i)}
    \frac{\abs{\Deriv \Bar{u}_n}^2}{2} + \frac{F (\Bar{u}_n)}{\varepsilon_n^2}
  \leq
    \Esing (g)\log \frac{1}{3 \eta_p}+\Cl{cst_if1aphoogi7eiPe5i},
\end{gather}
with \(\Cr{cst_if1aphoogi7eiPe5i} \defeq \Cr{cst_ooHei2tahgae8Tei3} +  \Esing (g) \log 3\).
By  weak compactness, Rellich's compactness theorem
and a diagonal argument,
the sequence \((\Bar{u}_n)_{n \in \Nset}\)  converges almost everywhere to some \(\Bar{u}_*  : \Omega_{\delta_{\partial \Omega}} \to \Rset^\nu\)
and weakly to \(\Bar{u}_*\) in \(W^{1, 2} (\Omega_{\delta_{\partial \Omega}} \setminus \bigcup_{i = 1}^k \Bar{B}_{\rho} (a_i), \Rset^\nu)\) for every \(\rho > 0\).
We have \(\Bar{u}_* = w\) on \(\Omega \setminus \Omega_{\delta_{\partial \Omega}}\).
We define \(u_* = \Bar{u}_* \vert_{\Omega}\).
Since by Fatou's lemma,
\[
 \int_{\Omega \setminus \bigcup_{i = 1}^k B_{3 \eta_p} (a_i)} F (u_*) \le \lim_{n \to \infty} \varepsilon_n^2 \int_{\Omega \setminus \bigcup_{i = 1}^k B_{3 \eta_p} (a_i)} \frac{F (u_n)}{\varepsilon_n^2} = 0,
\]
we have \(u_* \in \manifold{N}\) almost everywhere in \(\Omega\).
Moreover, for every \(p \in \Nset\), we have by \eqref{eq_eePhahw5ohng0zie0} and by lower semicontinuity
\begin{equation}
\label{eq_cacooPoc8euQuoh6N}
\begin{split}
  \int_{\Omega_{\delta_{\partial \Omega}}\setminus \bigcup_{i = 1}^k B_{3 \eta_p} (a_i)}
  \frac{\abs{\Deriv\Bar{u}_*}^2}{2} &\le \liminf_{n \to \infty} \int_{\Omega_{\delta_{\partial \Omega}}\setminus \bigcup_{i = 1}^k B_{3 \eta_p} (a_i)} \frac{\abs{\Deriv \Bar{u}_n}^2}{2} + \frac{F (\Bar{u}_n)}{\varepsilon_n^2}\\
  &\le \Esing (g)\log \frac{1}{3 \eta_p}+\Cr{cst_if1aphoogi7eiPe5i}.
  \end{split}
\end{equation}
By \cite{Monteil_Rodiac_VanSchaftingen_RE}*{Lemma 6.2}, for \(p\) large enough so that
\begin{multline*}
3\eta_p<
\Bar{\rho} \defeq \sup \{r > 0 \st \text{for each \(i \in \{1, \dotsc, k\}\), \(B_r (a_i) \subset \Omega_{\delta_{\partial \Omega}} \)}\\
\text{
  and for each \(j \in \{1, \dotsc, k\} \setminus \{i\}\), \(B_r (a_i) \cap B_r (a_j) = \emptyset\)} \},
\end{multline*}
we have
\begin{multline}\label{eq:aefibhaefzpz}
  \int_{\Omega_{\delta_{\partial \Omega}} \setminus \bigcup_{i = 1}^k B_{3 \eta_p} (a_i)} \frac{\abs{\Deriv \Bar{u}_*}^2}{2}
 \\
 \ge
 \sum_{i = 1}^k  \frac{\equivnorm{\tr_{\partial B_{3 \eta_p} (a_i)} \Bar{u}_*}^2}{4 \pi \nu_{\Bar{\rho}, 3 \eta_p} (a_i)}
  \log \frac{\Bar{\rho}}{3 \eta_p}
 \Biggl(1 -
 \Bigl(\frac{2 \pi \C \mathcal{E}^{\mathrm{ext}} (\tr_{\partial \Omega} \Bar{u}_*)}
 {\equivnorm{\tr_{\partial B_{3 \eta_p}} u}^{2}  \log \frac{\Bar{\rho}}{3 \eta_p}}\Bigr)^{1/2}\Biggr)^2,
\end{multline}
where 
\[
\nu_{\Bar{\rho}, {3 \eta_p}} (a)
\defeq \frac{1}{2 \pi  \log \frac{\Bar{\rho}}{{3 \eta_p}}}\int_{(B_{\Bar{\rho}} (a) \setminus \Bar{B}_{3 \eta_p} (a)) \cap \Omega}
\frac{1}{\abs{x - a}^2} \dif x \leq 1.
\]
Since \((\tr_{\partial B_{3 \eta_p} (a_i)} \Bar{u}_*)_{1 \le i \le k}\) is a topological resolution of \(\tr_{\partial \Omega_{\delta_{\partial \Omega}}} \Bar{u}_*\), and thus of \(g\), we have
\[
\sum_{i = 1}^k
 \frac{\equivnorm{\tr_{\partial B_{3 \eta_p} (a_i)} \Bar{u}_*}^2}{4 \pi \nu_{\Bar{\rho}, 3 \eta_p} (a_i)}
 \geq \sum_{i=1}^k \frac{\equivnorm{\tr_{\partial B_{3 \eta_p} (a_i)} \Bar{u}_*}^2}{4 \pi} \geq \Esing(g).
\]
It thus follows by \eqref{eq_cacooPoc8euQuoh6N} and \eqref{eq:aefibhaefzpz} that \(\lim_{p \to \infty} \nu_{\Bar{\rho}, 3 \eta_p} (a_i)=1\)   which implies that \(a_i \in \Omega\) (since \(\Omega\) has Lipschitz boundary). It also follows that \( (\tr_{\partial B_{3 \eta_p} (a_i)} \Bar{u}_*)_{1\leq i\leq k}\) is a minimal topological resolution of \(g\).
Hence, in view of \cref{def_renormalisable} and \eqref{eq_cacooPoc8euQuoh6N}, the map \(u_*\) is renormalisable. Thus, if we  let \( \operatorname{sing} (u_*) =  \{(a_1,\gamma_1), \dotsc, (a_k,\gamma_k)\}\) we obtain \eqref{it_zeqpfibhaedq} and \eqref{it_aebzherçbahçfh}.


We now prove \eqref{it_eShexuRaiChe7oochae6ooda}. By \eqref{it_zeqpfibhaedq} and \eqref{eq_borne_energie}, we deduce that, up to a subsequence, 
\begin{equation}\label{eq:qefpubhaebdq}
\frac{\abs{\Deriv u_n}^2}{2\abs{\log \varepsilon_n}} \rightharpoonup \sum_{i=1}^k \alpha_i \delta_{a_i} \quad\text{in the narrow topology of measures}
\end{equation}
for some constants \(\alpha_1,\dotsc,\alpha_k\in \Rset\). Thanks to the upper bound \eqref{eq_borne_energie} and the lower bound \cref{proposition_Firstlowerbound}, we have that
\[
 \biggl(\int_{\Omega} \frac{\abs{\Deriv u_n}^2}{2} + \frac{F (u_n)}{\varepsilon_n} - \Esing (g)\log \frac{1}{\varepsilon_n}\biggr)_{n \in \Nset}
\]
is bounded. Using this and the fact that, by \eqref{it_iji8Jienie1cad0ie}, \( \frac{1}{\abs{\log \varepsilon_n}}\int_\Omega\frac{F(u_n)}{\varepsilon_n^2}\rightarrow 0\) we obtain that 
\begin{equation}\label{ezpiuhaqefb}
\sum_{i=1}^k \alpha_i=\Esing(g).
\end{equation}
By \eqref{eq_eePhahw5ohng0zie0} and by Fatou's lemma, for almost every \(r \in (0, \Bar{\rho})\), we have for each \(i \in \{1, \dotsc, k\}\),
\begin{equation}
\label{FubiniTypeArgument}
 \liminf_{n \to \infty}
 \int_{\partial B_{r} (a_i)} \frac{\abs{\Deriv u_n}^2}{2} + \frac{F (u_n)}{\varepsilon_n^2} < + \infty.
\end{equation}
By Sobolev embedding and Arzela-Ascoli criterion, we obtain that for \(n\) large enough \(\Pi_{\manifold{N}} \compose \tr_{\partial B_{r} (a_i)} u_n\) and \(\tr_{\partial B_{r} (a_i)} u_*\) are homotopic, and thus in particular 
\begin{equation*}
\Esing (\tr_{\partial B_{r} (a_i)} u_n)
= \Esing(\tr_{\partial B_{r} (a_i)} u) = \Esing(\gamma_i) = \frac{\equivnorm{\gamma_i}^2}{4 \pi}.
\end{equation*}
Hence, taking such an \(r \in (\Bar{\rho}/2, \Bar{\rho})\), we get by \cref{corollary_lowerBndBoundary}
\begin{equation}
 \int_{\partial B_{r} (a_i)} \frac{\abs{\Deriv u_n}^2}{2} + \frac{F (u_n)}{\varepsilon_n^2}
 \ge \biggl(\log \frac{1}{\varepsilon_n}\biggr)\frac{\equivnorm{\gamma_i}^2}{4 \pi} - \C.
\end{equation}
Thus we have \(\alpha_i \ge \frac{\equivnorm{\gamma_i}^2}{4 \pi}\), and in view of \eqref{ezpiuhaqefb}, we obtain \eqref{it_eShexuRaiChe7oochae6ooda}.


The rest of the proof is devoted to assertions \eqref{it_Eejai5sungeiF7sho} and \eqref{it_vaid7ieb1rub4Che6}.
For almost every \(r \in (0, \Bar{\rho})\), for each \(i \in \{1, \dotsc, k\}\) and \(n \in \Nset\),
\(\tr_{\partial B_r (a_i)} u_n = u_n \vert_{\partial B_r (a_i)}\).
Hence if we define \(\gamma_{i, n}^r : \Sset^1 \to \Rset^\nu\)  by \(\gamma_{i, n}^r (x) \defeq u_n (a_i + r x)\), we have by \eqref{FubiniTypeArgument}, for almost every \(r \in (0, \Bar{\rho})\)
\begin{equation}
\label{eq_quah6xie2aifoociG}
 \liminf_{n \to \infty}
 \int_{\Sset^1} \frac{\abs{(\gamma_{i, n}^r)'}^2}{2} + \frac{r^2}{\varepsilon_n^2}F (\gamma_{i, n}^r) < + \infty.
\end{equation}
There exists thus a subsequence \((n_\ell)_{\ell \in \Nset}\) (depending on \(r\)) such that
\begin{equation}\label{eq:afbuhaeboiah}
 \sup_{\ell \in \Nset} 
 \int_{\Sset^1} \frac{\abs{(\gamma_{i, n_\ell}^r)'}^2}{2} + \frac{r^2}{\varepsilon_{n_\ell}^2}F (\gamma_{i, n_\ell}^r) < + \infty.
\end{equation}
In particular, \(F(\gamma_{i,n_\ell}^r) \rightarrow 0\) a.e.\ up to extraction of a subsequence if necessary. Hence, by \eqref{hyp20}, we have \(\dist_{\manifold{N}} \compose \gamma_{i, n_\ell}^r\rightarrow 0\) a.e.\ so that for \(\ell\) large enough, we have \(\dist_{\manifold{N}} \compose \gamma_{i, n_\ell}^r<\delta_\manifold{N}/2\). Thus, by \cref{prop_Q_F_projection} and by \eqref{eq:afbuhaeboiah}, we have
 \begin{equation}
   \lim_{\ell \to \infty}
   \bigabs{\mathcal{Q}^{ r/\varepsilon_{n_\ell}}_{F,\gamma_{i, n_\ell}^r} -  \mathcal{Q}^{ r/\varepsilon_{n_\ell}}_{F,\Pi_{\manifold{N}} \compose \gamma_{i, {n_\ell}}^r}}
   \le \lim_{\ell \to \infty} \C \frac{\varepsilon_{n_\ell}}{r} \int_{\Sset^1} \frac{\abs{(\gamma_{i, n_\ell}^r)'}^2}{2} + \frac{r^2}{\varepsilon_{n_\ell}^2}F (\gamma_{i, n_\ell}^r) = 0,
 \end{equation}
so that
\begin{equation}
\label{eq_voo5ooPithofaChae}
 \liminf_{\ell \to \infty} \mathcal{Q}^{ r/\varepsilon_{n_\ell}}_{F,\gamma_{i, n_\ell}^r} - \frac{\equivnorm{\gamma_i}^2}{4 \pi} \log \frac{r}{\varepsilon_{n_\ell}}
 = \liminf_{\ell \to \infty} \mathcal{Q}^{ r/\varepsilon_{n_\ell}}_{F,\Pi_{\manifold{N}} \compose \gamma_{i, {n_\ell}}^r}- \frac{\equivnorm{\gamma_i}^2}{4 \pi} \log \frac{r}{\varepsilon_{n_\ell}}.
\end{equation}
On the other hand, by \eqref{eq:afbuhaeboiah} and Sobolev embeddings, the sequence \((\gamma_{i, n_\ell}^r)_{\ell\in\Nset}\) converges strongly to \(u_* (a_i + r\,\cdot)\) in \(W^{1/2, 2} (\Sset^1, \Rset^\nu)\), and thus \((\Pi_{\manifold{N}} \compose \gamma_{i, n_\ell}^r)_{\ell\in\Nset}\) converges to \(u_* (a_i + r\,\cdot)\) in \(W^{1/2, 2} (\Sset^1, \manifold{N})\). Hence, by \cite{Monteil_Rodiac_VanSchaftingen_RE}*{Proposition 3.3 (vi)}, \(\lim_{\ell \to \infty} \synhar{\Pi_{\manifold{N}} \compose \gamma_{i, n_\ell}^r}{u_* (a_i +r \cdot)}= 0\). Thus by \eqref{eq_voo5ooPithofaChae} and  \cref{proposition_QW_dependence}
\begin{equation}
\label{eq_wah3ieg5hieNgaeja}
\liminf_{\ell\to \infty} \mathcal{Q}^{ r/\varepsilon_{n_\ell}}_{F,\gamma_{i,n_\ell}^r} - \frac{\equivnorm{\gamma_i}^2}{4 \pi} \log \frac{r}{\varepsilon_{n_\ell}}
 =
  \liminf_{\ell \to \infty} \mathcal{Q}^{ r/\varepsilon_{n_\ell}}_{F,u_* (a_i + r \cdot)} - \frac{\equivnorm{\gamma_i}^2}{4 \pi} \log \frac{r}{\varepsilon_{n_\ell}}.
\end{equation}
Finally by \cref{proposition_QW_dependence} again, we have in view of \eqref{eq_wah3ieg5hieNgaeja}
\[
 \liminf_{\ell \to \infty} \mathcal{Q}^{ r/\varepsilon_{n_\ell}}_{F,\gamma_{i,n_\ell}^r}- \frac{\equivnorm{\gamma_i}^2}{4 \pi} \log \frac{r}{\varepsilon_{n_\ell}}
 \ge \mathcal{Q}_{F,\gamma_i} - \synhar{u_* (a_i + r \cdot)}{\gamma_i}.
\]
It follows thus that
\[
\begin{split}
 \lim_{\ell \to \infty} \int_{\Omega}& \frac{\abs{\Deriv u_{n_\ell}}^2}{2}
 + \frac{F (u_{n_\ell})}{\varepsilon_{n_\ell}^2}- \Esing (g)\log \frac{1}{\varepsilon_{n_\ell}}\\
 &\ge \liminf_{V \to \infty} \int_{\Omega \setminus \bigcup_{i = 1}^k B_r (a_i)} \frac{\abs{\Deriv u_{n_\ell}}^2}{2} + \frac{F (u_{n_\ell})}{\varepsilon_{n_\ell}^2}
 + \sum_{i = 1}^k \liminf_{\ell \to \infty} \int_{B_r (a_i)}  \frac{\abs{\Deriv u_{n_\ell}}^2}{2}
 + \frac{F (u_{n_\ell})}{\varepsilon_{n_\ell}^2} \\ & \quad \quad -\Esing(g) \log \frac{1}{\varepsilon_{n_\ell}}\\
 &\ge \int_{\Omega \setminus \bigcup_{i = 1}^k B_r (a_i)} \frac{\abs{\Deriv u_*}^2}{2} - \Esing (g) \log \frac{1}{r}
 + \sum_{i=1}^k \liminf_{\ell \to \infty} \mathcal{Q}^{r/\varepsilon_{n_\ell}}_{F,\gamma_{i,n_\ell}^r}- \frac{\equivnorm{\gamma_i}^2}{4 \pi} \log \frac{r}{\varepsilon_{n_\ell}}\\
 &\ge \int_{\Omega \setminus \bigcup_{i = 1}^k B_r (a_i)} \frac{\abs{\Deriv u_*}^2}{2} - \Esing (g) \log \frac{1}{r} + \mathcal{Q}_{F} (u_*) - \sum_{i  = 1}^k \synhar{u_* (a_i + r \cdot)}{\gamma_i}.
\end{split}
\]
We reach the conclusion \eqref{it_Eejai5sungeiF7sho} by letting \(r \to 0\).

The proof of \eqref{it_vaid7ieb1rub4Che6} proceeds as the proof of \eqref{it_Eejai5sungeiF7sho} in order to reach
\begin{multline*}
\liminf_{\ell \to \infty} \sum_{i = 1}^k \int_{B_\rho (a_i)} \frac{\abs{\Deriv u_{n_\ell}}^2}{2}
 + \frac{F (u_{n_\ell})}{\varepsilon_{n_\ell}^2}- \Esing (g)\log \frac{1}{\varepsilon_{n_\ell}}\\
 \ge \sum_{i = 1}^k  \int_{B_\rho (a_i) \setminus B_r (a_i)} \frac{\abs{\Deriv u_*}^2}{2} - \Esing (g) \log \frac{1}{r} + \mathcal{Q}_{F} (u_*) - \sum_{i  = 1}^k \synhar{u_*(a_i + r \cdot)}{\gamma_i},
\end{multline*}
the conclusion follows then by letting \(r \to 0\) and additivity of integrals.
\end{proof}
\begin{remark}[\(\Gamma\)--convergence]
For each \(g\in W^{1/2,2}(\partial\Omega,\manifold{N})\) and \(\varepsilon \in (0,+\infty)\), we define \({E}^{g}_\varepsilon\) on set of measurable functions by setting
\begin{equation*}
{E}^{g}_\varepsilon(u)=
\begin{cases}
\displaystyle
  \int_\Omega \Bigl(\frac{\abs{\Deriv u}^2}{2} +\frac{F(u)}{\varepsilon^2}\Bigr) -\Esing (g)\log\frac 1\varepsilon&\text{if \(u\in W^{1,2}(\Omega,\Rset^\nu)\) and
\(\tr_{\partial \Omega} u=g\) on \(\partial\Omega\),}\\
+\infty&\text{otherwise,}
\end{cases}
\end{equation*}
and we define the limit functional \(E^{g}_0\) on the set of measurable functions by setting
\begin{equation}
\label{GammaLimit}
E^{g}_0(u)=
\begin{cases}
\mathcal{E}^\mathrm{ren}(u)+\mathcal{Q}_F(u)
& \text{if \(u \in \mathcal{A}_g (\Omega, \manifold{N})\)},\\
+\infty & \text{otherwise},
\end{cases}
\end{equation}
where \(\mathcal{A}_g(\Omega, \manifold{N})\) is the set of maps \(u\in W^{1, 2}_{\mathrm{ren}}(\Omega, \manifold{N})\) such that \(\tr_{\partial \Omega} u=g\) and \((\gamma_1,\dotsc,\gamma_k)\) is a minimal topological resolution of \(g\), where \(\operatorname{sing}(u)=\{(a_1,\gamma_1),\dotsc,(a_k,\gamma_k)\}\).

The family of functionals \(({E}^{g}_\varepsilon)_{\varepsilon>0}\) \(\Gamma\)-converges as \(\varepsilon\to 0\) to \(E^{g}_0\) in \(L^p(\Omega,\Rset^\nu)\) endowed with the strong topology for every \(p \in [1, +\infty)\), and in \(W^{1, p} (\Omega, \Rset^\nu)\) endowed with the weak or strong topology for every \(p \in [1, 2)\). The upper bound follows from the upper bound \cref{proposition_improvedupperBound} and the lower bound from \cref{theorem_compactnessweak}.
For \(\manifold{N} = \Sset^1\) and for the strong convergence in \(W^{1, 1}\), a \(\Gamma\)--convergence result at leading order, i.e.\ the \(\Gamma\)--convergence of \( \mathcal{E}_F^\varepsilon/\log \frac{1}{\varepsilon}\), is due to Jerrard and Soner \cite{Jerrard_Soner_2002}*{Theorem 4.1}. For \(\manifold{N} = \Sset^1\), a \(\Gamma\)--convergence type result at next order can be found in \cite{Alicandro_Ponsiglione_2014}: if \(Ju=\operatorname{det}\nabla u\) denotes the Jacobian of \(u\), the authors show the \(\Gamma\)--convergence of the energy \(\inf_{\{u\st Ju=J\}}\mathcal{E}_F^\varepsilon(u)-\Esing(g)\log \frac{1}{\varepsilon}\) in the Jacobian variable \(J\in\mathcal{C}^{0,1}(\Omega)'\) endowed with the convergence in the flat norm. Our framework allows us to state the \(\Gamma\)--convergence of \(\mathcal{E}_F^\varepsilon(u)-\Esing(g)\log \frac{1}{\varepsilon}\) in the variable \(u\); this in particular requires to introduce the renormalised energy \(\mathcal{E}^\mathrm{ren}\) of renormalisable maps (not necessarily harmonic away from singularities). In the case of the circle \(\manifold{N}=\Sset^1\), this approach is reminiscent of \cite{GoldmanMerletMillot}.

\end{remark}

\subsection{Convergence of minimisers}
We are now ready to fully state and prove our result about the convergence of minimisers:
\begin{theorem}
\label{theorem_main1}
Let \(g\in W^{1/2,2}(\partial\Omega,\manifold{N})\),
let \((\varepsilon_n)_{n \in \Nset}\) be a sequence in \((0,+\infty)\) converging to \(0\)
and for each \(n \in \Nset\)  let \(u_n \in W^{1,2}(\Omega,\Rset^\nu)\)
be a minimiser of the Ginzburg--Landau energy \(\mathcal{E}^\varepsilon_F\) under the Dirichlet boundary condition \(\tr_{\partial\Omega}u_n=g\).
Then, up to a subsequence, there exists a map \(u_*\in W^{1, 2}_{\mathrm{ren}} (\Omega,\manifold{N})\) such that if we write \(\operatorname{sing} (u_*) =\{(a_1,\gamma_1),\dotsc,(a_k,\gamma_k)\}\), we have
\begin{enumerate}[(i)]
  \item \label{it_aiboo5ahXoh7meingohnoo9i} the sequence \((u_{n})_{n \in \Nset}\) converges almost everywhere to \(u_*\) and strongly in
  \(W^{1, 2}_{\textrm{loc}} (\Bar{\Omega} \setminus \{a_1, \dotsc, a_k\})\) and \(F (u_{n})/\varepsilon_n^2 \to 0\) in \(L^1_{\mathrm{loc}} (\Bar{\Omega} \setminus \{a_1, \dotsc, a_k\})\), 
  \item \label{it_ooChifooNg2zeequohsahsoo}\(\Esing (g) = \sum_{i = 1}^k \frac{\equivnorm{\gamma_i}^2}{4 \pi}\),
  \item \label{it_ud9Ooy6uelooch0jei0oopay} \(\displaystyle \sup_{n \in \Nset} \int_{\Omega} \frac{
 \abs{\Deriv (\dist_{\manifold{N}} \compose u_n)}^2}{2} + \frac{F (u_n)}{\varepsilon_n^2}
 + \sup_{t > 0} t^2 \mathcal{L}^2 (\abs{\Deriv u_n}^{-1}[0, +\infty)) < +\infty\),
\item\label{weakCvJac} one has, narrowly as measures on \(\Omega\),
 \[
 \frac{\abs{\Deriv u_n}^2}{2 \log \frac{1}{\varepsilon_n}}\weakto \sum_{i = 1}^k \frac{\equivnorm{\gamma_i}^2}{4 \pi}\delta_{a_i},
 \]
  \item\label{it_zpghargzqger} \(\tr_{\partial \Omega} u_* = g\) and \(u_*\) is a minimising renormalisable singular harmonic map (see \cref{minimisingRenormalisable}) so that in particular, for every \(\rho\in\overline{\rho}(a_1,\dotsc,a_k)\), \(u_*\in \mathcal{C}^\infty(\Omega\setminus\{a_1,\dotsc,a_k\},\manifold{N})\) is harmonic minimising in \(\Omega\setminus\bigcup_{i=1}^k \Bar{B}_\rho(a_i)\) with respect to its own boundary conditions,
  \item \label{it_eigh6wav2ieS3yahThaipho8} we have the equalities
\begin{equation*}
  \begin{split}
    \lim_{n \to \infty} &\int_{\Omega} \frac{\abs{\Deriv u_{n}}^2}{2} + \frac{F (u_{n})}{\varepsilon_n^2} - \Esing (g)\log \frac{1}{\varepsilon_n}\\
    &= \mathcal{E}^{\mathrm{ren}}(u_*)+\mathcal{Q}_{F}(u_*) = \inf \{\mathcal{E}^{\mathrm{ren}} (u) + \mathcal{Q}_F(u) \st
  u \in W^{1, 2}_{\mathrm{ren}} (\Omega,\manifold{N})\text{ and }\tr_{\partial \Omega} u = g\}\\
  &= \mathcal{E}^{\mathrm{geom}}_{g,\gamma_1,\dotsc,\gamma_k}(a_1,\dotsc,a_k) +\sum_{i=1}^k \mathcal{Q}_{F,\gamma_i}=
 \mathcal{W}_{\min},
  \end{split}
\end{equation*}
where
 \begin{multline}
 \label{eq_ies5lohth9sa7Aiboo5AiKat}
 \mathcal{W}_{\min}
 \defeq
  \inf\biggl\{\mathcal{E}^{\mathrm{geom}}_{g,\eta_1,\dotsc,\eta_\ell}(b_1,\dotsc,b_\ell)+\sum_{i=1}^\ell \mathcal{Q}_{F,\eta_i}
\st
b_1,\dotsc,b_\ell \in \Omega
\text{ are distinct }\\[-1em]
\text{and }
(\eta_1, \dotsc, \eta_\ell)
\text{ is a minimal topological resolution } of g
\biggr\}.
 \end{multline}
\end{enumerate}
\end{theorem}

When \(\manifold{N} = \Sset^1\), the weak \(L^2\) estimate \eqref{it_ud9Ooy6uelooch0jei0oopay} on the gradient is due to Serfaty and Tice \cite{Serfaty_Tice_2008}*{Proposition 1.3}.

\begin{remark}\label{strongWonep}
By \eqref{it_aiboo5ahXoh7meingohnoo9i} and \eqref{it_ud9Ooy6uelooch0jei0oopay},
for every \(p \in [1, 2)\), the sequence \((u_n)_{n \in \Nset}\) converges to \(u\) strongly in \(W^{1, p} (\Omega)\).
When \(\manifold{N} = \Sset^1\),
such a convergence was known for smooth data \cite{Bethuel_Brezis_Helein_1994}*{Lemma X.11}
and \(W^{1/2, 2}\) data \cite{Bethuel_Bourgain_Brezis_Orlandi_2001}.
\end{remark}

\begin{remark}\label{minimisingRenormalisable}
Following \cite{Monteil_Rodiac_VanSchaftingen_RE}*{Definition 7.8}, the map \(u_* \in W^{1, 2}_{\mathrm{ren}} (\Omega, \manifold{N})\) being a minimising renormalisable singular harmonic map means that if we set \(\operatorname{sing} (u_*) = \{(a_1, \gamma_1), \dotsc, (a_k, \gamma_k)\}\), then for every map \(v \in W^{1, 2}_{\mathrm{ren}} (\Omega, \manifold{N})\) with \(\operatorname{sing} (v) = \{(b_1, \gamma_1), \dotsc, (b_k, \gamma_k)\}\) (that is \(\operatorname{sing}(v)\) differs from \(\operatorname{sing}(u^\ast)\) only by the position of the points, but not by the \(\gamma_i\)), one has
\[
 \mathcal{E}^{\mathrm{ren}} (u_*) \le \mathcal{E}^{\mathrm{ren}} (v).
\]
In particular \(u_*\) is a stationary renormalisable harmonic map, which means that \(u_*\) is harmonic away from its singularities and that its stress-energy energy tensor has vanishing flux around every singularity, or equivalently the residue of its Hopf differential vanishes at every singularity \cite{Monteil_Rodiac_VanSchaftingen_RE}*{Lemma 7.7 and Proposition 7.9}.
\end{remark}

When \(\manifold{N} = \Sset^1\), \cref{theorem_main1} is essentially due to Bethuel, Brezis and Hélein \cite{Bethuel_Brezis_Helein_1994} for star-shaped domains, and to Struwe for simply-connected domains \cite{Struwe_1994}.
The existence of finitely many singularities and the strong convergence in the case of a general compact manifold \(\manifold{N}\) was proved by Canevari \cite{Canevari_2015} for  general smooth bounded domains

\begin{proof}[Proof of \cref{theorem_main1}]
Since \((u_n)_{n \in \Nset}\) is a sequence of minimisers, it follows from \cref{proposition_geometricUpperBound} that
we have
\begin{equation}
\label{eq_ii1sai1gaeKi2eiTh6Izool7}
  \limsup_{n \to \infty} \int_{\Omega} \frac{\abs{\Deriv u_n}^2}{2} + \frac{F(u_n)}{\varepsilon_n^2}- \Esing(g)\log\frac{1}{\varepsilon_n}\\
  \le
  \mathcal{W}_{\min}.
\end{equation}
By \eqref{eq_ii1sai1gaeKi2eiTh6Izool7} and by \cref{theorem_compactnessweak}, up to a subsequence, there exists a family of points \((a_1,\dotsc,a_k)\) in \(\Omega\) such that \((u_{n})_{n\in\Nset}\) converges weakly in \(W^{1,2}_{\textrm{loc}}(\overline{\Omega}\setminus\{a_1,\dotsc,a_k\},\Rset^\nu)\) to some limit \(u_*\in W^{1,2}_{\textrm{ren}}(\Omega,\manifold{N})\) and
\begin{equation}
\label{eq_ii6wieBeel6EePhie8icaija}
 \mathcal{E}^{\mathrm{ren}}(u_*)+ \mathcal{Q}_F(u_*)
 \le \mathcal{W}_{\min}.
\end{equation}
Note that from \cref{theorem_compactnessweak}, \eqref{it_ooChifooNg2zeequohsahsoo}, \eqref{it_ud9Ooy6uelooch0jei0oopay} and \eqref{weakCvJac} hold. Furthermore \eqref{it_aiboo5ahXoh7meingohnoo9i} also holds if the strong convergence is replaced by the weak convergence.  

Since the map \(u_*\) is renormalisable, by \cref{admissibleMaps} and by \eqref{defTopCstU},
\begin{equation}
\label{eq_vaiX0jei2AhVaor1Epoo0Lah}
 \mathcal{E}^{\mathrm{ren}}(u_*)+ \mathcal{Q}_F(u_*)
 \ge \mathcal{E}^{\mathrm{geom}}_{\gamma_1, \dotsc, \gamma_k} (a_1, \dotsc, a_k) + \sum_{i = 1}^k \mathcal{Q}_{F,\gamma_i},
\end{equation}
where we have set \(\operatorname{sing}(u_*)=\{(a_1,\gamma_1),\dotsc,(a_k,\gamma_k)\}\). It follows thus from \eqref{eq_ies5lohth9sa7Aiboo5AiKat}, \eqref{eq_ii6wieBeel6EePhie8icaija} and \eqref{eq_vaiX0jei2AhVaor1Epoo0Lah} that
\begin{equation}
 \mathcal{E}^{\mathrm{ren}}(u_*) + \mathcal{Q}_F(u_*)
 =  \mathcal{W}_{\min} = \mathcal{E}^{\mathrm{geom}}_{\gamma_1, \dotsc, \gamma_k} (a_1, \dotsc, a_k) + \sum_{i = 1}^k \mathcal{Q}_{F,\gamma_i}.
\end{equation}
By \cref{proposition_improvedupperBound}, since \((u_n)_{n \in \Nset}\) is a sequence of almosts-minimisers and since \(\Esing(g)=\sum_{i=1}^k\frac{\equivnorm{\gamma_i}^2}{4\pi}\), we have also,
\begin{multline}
\label{eq_roh1mahzu4teigh3YaiD7soo}
  \limsup_{n \to \infty} \int_{\Omega} \frac{\abs{\Deriv u_n}^2}{2} + \frac{F(u_n)}{\varepsilon_n^2}- \Esing(g)\log\frac{1}{\varepsilon_n}\\
  \le
  \inf \{\mathcal{E}^{\mathrm{ren}} (u) + \mathcal{Q}_F(u) \st
  u \in W^{1, 2}_{\mathrm{ren}} (\Omega, \manifold{N}) \text{ and }
  \tr_{\partial \Omega} u = g\},
\end{multline}
which together with \eqref{it_Eejai5sungeiF7sho} in \cref{theorem_compactnessweak} yields
\begin{equation}
\label{eq_eelohngiequaich7BohN2phu}
\begin{split}
 \mathcal{E}^{\mathrm{ren}}(u_*)+ \mathcal{Q}_F(u_*) &= \lim_{n \to \infty} \int_{\Omega} \frac{\abs{\Deriv u_n}^2}{2} + \frac{F(u_n)}{\varepsilon_n^2}- \Esing(g)\log\frac{1}{\varepsilon_n}\\
 &= \inf \{\mathcal{E}^{\mathrm{ren}} (u) + \mathcal{Q}_F(u) \st
  u \in W^{1, 2}_{\mathrm{ren}} (\Omega, \manifold{N}) \text{ and }
  \tr_{\partial \Omega} u = g\}.
  \end{split}
\end{equation}
Thus we have proved \eqref{it_eigh6wav2ieS3yahThaipho8}. In particular, \(u_*\) is a minimising renormalisable singular harmonic map. It follows that if \(\rho\in\overline{\rho}(a_1,\dotsc,a_k)\) and if the map \(u\in W^{1,2}_{\textrm{loc}}(\Omega\setminus\{a_1,\dotsc,a_k\},\manifold{N})\) satisfies \(\tr_{\partial \Omega} u=g\) on \(\partial\Omega\) and \(u=u_*\) in \(\cup_{i=1}^k B_\rho(a_i)\), then \(\operatorname{sing}(u)=\operatorname{sing}(u_*)\) and \(\mathcal{E}^{\mathrm{ren}}(u)\geq \mathcal{E}^{\mathrm{ren}}(u_*)\); hence
\[
\int_{\Omega\setminus\bigcup_{i=1}^k B_\rho(a_i)}\frac{\abs{\Deriv u}^2}{2}
- \int_{\Omega\setminus\bigcup_{i=1}^k B_\rho(a_i)}\frac{\abs{\Deriv u_*}^2}{2}
=
\mathcal{E}^{\mathrm{ren}}(u)-\mathcal{E}^{\mathrm{ren}}(u_*)
\geq 0,
\]
so that \(u_*\) is harmonic minimising in \(\Omega\setminus\bigcup_{i=1}^k \Bar{B}_\rho(a_i)\) with respect to its own boundary conditions and, in particular, \(u_*\in \mathcal{C}^\infty(\Omega\setminus\{a_1,\dotsc,a_k\},\manifold{N})\) by the result of \cite{Morrey_1948}; this proves \eqref{it_zpghargzqger}.

Finally, by \eqref{eq_eelohngiequaich7BohN2phu} and \cref{theorem_compactnessweak} \eqref{it_vaid7ieb1rub4Che6}, we have for every \(\rho \in (0, \Bar{\rho} (a_1, \dotsc, a_k))\),
     \begin{multline*}
 \limsup_{n \to \infty}
 \int_{\Omega \setminus \bigcup_{i = 1}^k B_\rho (a_i)
} \frac{\abs{\Deriv u_n}^2}{2} + \frac{F (u_n)}{\varepsilon_n^2}\\
\leq
 \limsup_{n \to \infty}\Bigl(\int_{\Omega} \frac{\abs{\Deriv u_n}^2}{2} + \frac{F (u_n)}{\varepsilon_n^2}- \Esing(g)\log\frac{1}{\varepsilon_n}\Bigr)\\
\hfill -\liminf_{n\to\infty}\Bigl(\int_{\bigcup_{i = 1}^k B_\rho (a_i)
} \frac{\abs{\Deriv u_n}^2}{2} + \frac{F (u_n)}{\varepsilon_n^2}- \Esing(g)\log\frac{1}{\varepsilon_n}\Bigr)\\
 \le
 \int_{\Omega \setminus \bigcup_{i = 1}^k B_\rho (a_i)
} \frac{\abs{\Deriv u_\ast}^2}{2},
\end{multline*}
which implies the announced strong convergence in \eqref{it_aiboo5ahXoh7meingohnoo9i}.
\end{proof}

\section{An explicit computation of the renormalised energy}
\label{section_simple_boundary_conditions}

Although the geometric renormalised energy of singularities and the renormalised energy of renormalisable maps are defined via a shrinking holes approach and are thus quite implicit, if \(\Omega\) is simply-connected and \(g\) is a reparametrisation of a minimising atomic geodesic in \(\manifold{N}\), the geometric renormalised energy of a single singularity coincides strikingly with \(\manifold{N}=\Sset^1\) \cite{Monteil_Rodiac_VanSchaftingen_RE}*{Theorem 10.1}. When \(\Omega =B_1\) this geometric renormalised energy can be explicitly  computed and this allows one  to locate asymptotic singularities for strictly atomic minimising geodesic as boundary conditions (see \eqref{minimisingGeodesicBC} and \eqref{strictlyAtomicBC} in \cref{theorem_explicit}), as Bethuel, Brezis and H\'elein did for \(\manifold{N} = \Sset^1\) \cite{Bethuel_Brezis_Helein_1994}*{Theorem 0.4} in response to a question of Matano.
 
\begin{theorem}
  \label{theorem_explicit}
Let \(\Omega\) be a Lipschitz bounded domain and let \(F \in \mathcal{C} (\Rset^\nu, [0,+\infty))\) satisfy \(F^{-1}(\{0\}) = \manifold{N}\) and \eqref{hyp20}. Assume that
\begin{enumerate}[(a)]
 \item\label{minimisingGeodesicBC} \(g:\mathbb{S}^1 \rightarrow \manifold{N}\) is a minimising geodesic,
 \item\label{strictlyAtomicBC} if \((\gamma_1, \dotsc, \gamma_k)\) is a minimal topological resolution of \(g\) 
 then \(k = 1\) and \(\gamma_1\) is homotopic to \( g\),
 \item\label{homotopicSynharmonic} every map homotopic to \(g\) is synharmonic to \(g\).
\end{enumerate}
If for each \(\varepsilon\in(0,+\infty)\), \(u_\varepsilon\) is a minimiser of \(\mathcal{E}^\varepsilon_F\) in \(W^{1,2}(B_1 ,\Rset^\nu)\) under the condition \(\tr_{\partial\Omega}u_\varepsilon=g\), then
  \begin{equation*}
    u_\varepsilon \underset{\varepsilon\to 0}{\rightarrow} u_* \quad\text{in } W_{\textrm{loc}}^{1,2}(B_1 \setminus \{ 0\},\Rset^\nu),
  \end{equation*}
with \(u_* (x) = g(x/\abs{x})\).
\end{theorem}

\begin{proof}
By our assumptions and \cref{theorem_main1}, given any sequence \((\varepsilon_n)_{n\in\Nset}\) in \((0,+\infty)\) converging to \(0\), up to extraction of a  subsequence, there exists a map \(u_\ast\in W^{1,2}_\mathrm{ren}(\Omega,\manifold{N})\) such that \(\operatorname{sing} (u_*) = \{(a,\gamma)\}\) for some \(a\in\Omega\) and some minimising geodesic \(\gamma:\Sset^1\to\manifold{N}\) homotopic to \(g\), and such that \((u_{\varepsilon_n})_{n\in\Nset}\) converges to \(u_*\) in \(W^{1,2}_{\mathrm{loc}}( B\setminus \{a \},\Rset^\nu)\) with
\[
 \mathcal{E}^{\mathrm{ren}}(u_*)+Q_{F, \gamma} = \inf\, \{\mathcal{E}_{g,\gamma}^{\mathrm{geom}}(x)+\mathcal{Q}_{F, \gamma} \st x \in B_1\}=\mathcal{E}_{g,\gamma}^{\mathrm{geom}}(a)+\mathcal{Q}_{F, \gamma}.
\]
By \eqref{homotopicSynharmonic}, \(\gamma\) and \(g\) are synharmonic, which implies that \(\mathcal{E}_{g,\gamma}^{\mathrm{geom}}=\mathcal{E}_{g,g}^{\mathrm{geom}}\) in view of \eqref{synharmonicDependanceGeometric}. But, from  \cite{Monteil_Rodiac_VanSchaftingen_RE}*{Theorem 10.1}, we know that \(\mathcal{E}_{g,g}^{\mathrm{geom}}\) has a unique minimiser at \(x = 0\), and thus \(a = 0\).
By the characterisation \eqref{renormalised_energy_u} of the renormalised energy of the renormalisable map \(u_*\),
  \begin{equation}
    0 = \mathcal{E}^{\mathrm{ren}}(u_*) = \sup_{\rho\in (0, 1)} \int_{B_1\setminus B_\rho} \frac{\abs{\Deriv u_*}^2}{2}-\frac{\equivnorm{g}^2}{4\pi} \log\frac{1}{\rho}.
  \end{equation}
  In view of the elementary lower bound
  \[
  \int_{B_1\setminus B_\rho} \frac{\abs{\Deriv u_*}^2}{2}
  \geq 
  \frac{\equivnorm{g}^2}{4\pi} \log\frac{1}{\rho}
  +\frac{1}{2}\int_{B_1\setminus B_\rho}\abs{\frac{\partial u_*}{\partial r}}^2,
  \]
this implies that for almost every \(x \in B_1\), \(u_* (x) =g\bigl(\frac{x}{\abs{x}}\bigr)\).
  Since the limit is independent of the subsequence, the convergence holds for the whole family.
\end{proof}

\section{Convergence of solutions to the Ginzburg--Landau equation}
\label{section_solutions}
We consider now solutions to the Ginzburg--Landau equation \eqref{eq_GinzburgLandau} arising at least formally as the Euler--Lagrange equation of the Ginzburg--Landau energy \eqref{eq_GLenergy}. 

\subsection{Boundedness and Euler--Lagrange equation for minimisers}
\label{section_Euler_Lagrange}
We first show that under fairly general and reasonable conditions, minimisers of the Ginzburg--Landau energy \(\mathcal{E}^\varepsilon_F\) are weak solutions to the Ginzburg--Landau equation \eqref{eq_GinzburgLandau}.

\begin{proposition}
\label{proposition_equation}
Let \(\Omega\) be a Lipschitz bounded domain and let \(F \in \mathcal{C}^1 (\Rset^\nu, [0,+\infty))\) such that \(F^{-1}(\{0\}) = \manifold{N}\). If \(u \in W^{1, 2} (\Omega, \Rset^\nu)\) is a minimiser for the Ginzburg--Landau energy \(\mathcal{E}^\varepsilon_F\) under its own Dirichlet boundary conditions, and if \(\nabla F (u) \in L^1_{\mathrm{loc}} (\Omega)\), then for every \(\varphi \in \mathcal{C}^1_c (\Omega, \Rset^\nu)\),
\[
 \int_{\Omega} \Deriv u \cdot \Deriv \varphi + \frac{\nabla F (u)}{\varepsilon^2} \cdot \varphi = 0.
\]
\end{proposition}

\Cref{proposition_equation} does not make any assumption on \(F\) beyond continuous differentiability; in particular \(F\) needs not satisfy \eqref{hyp20}.

The proof of \cref{proposition_equation} follows a truncation argument due to Bousquet \citelist{\cite{Ponce_2016}*{proof of Theorem 4.23}\cite{Bousquet_2013}}. 

\begin{proof}[Proof of \cref{proposition_equation}]
Let \(\theta \in \mathcal{C}^1 (\Rset_+)\) such that \(\theta = 1\) on \([0, 1]\) and \(\theta = 0\) on \([2, +\infty)\).
For every \(R > 0\), we consider the function \(\eta_R \defeq \theta(\abs{u}/R)\).
We have, since \(\eta_R \in W^{1, 2} (\Omega)\) and since \(u\) is bounded on the set \( \{\eta_R \neq 0\} \),
\[
\begin{split}
0 =
  \lim_{t \to 0} \frac{\mathcal{E}^\varepsilon_F (u + t \eta_R \varphi) - \mathcal{E}^\varepsilon_F (u)}{t}
  &= \int_{\Omega} \eta_R \Deriv u \cdot \Deriv \varphi + 
  \Deriv u \cdot (\Deriv\eta_R) \varphi
  + \frac{\nabla F (u)}{\varepsilon^2} \cdot \varphi \eta_R\\
  &= \int_{\Omega}  \Bigl(\Deriv u \cdot \Deriv \varphi + \frac{\nabla F (u)}{\varepsilon^2} \cdot \varphi\Bigr)\theta (\tfrac{\abs{u}}{R})\\
 &\qquad +\sum_{i=1}^2 \int_{\Omega} (\partial_i u\cdot\varphi) \bigl(\partial_i u\cdot\frac{u}{\abs{u}}\bigr)\frac{\theta'(\tfrac{\abs{u}}{R}) }{R}.
  \end{split}
\]
Letting \(R \to +\infty\), we conclude in view of Lebesgue's dominated convergence theorem.
\end{proof}

The condition \(\nabla F (u) \in L^1_{\mathrm{loc}} (\Omega)\) in \cref{proposition_equation} can be obtained by establishing an apriori bound on the minimiser.

\begin{proposition}
\label{bounded_minimiser}
Let \(\Omega\) be a Lipschitz bounded domain and let \(F \in \mathcal{C}^1 (\Rset^\nu, [0,+\infty))\) such that \(F^{-1}(\{0\}) = \manifold{N}\). If there exists a map \(\Psi : \mathcal{C}^{0, 1} (\Rset^\nu, \Rset^\nu)\) such that 
\begin{enumerate}[(a)]
 \item \label{it_aipai9cae5ohxahn3Jong6oo} \(\Psi\) is non-expansive in \(\Rset^\nu\), i.e., \(\abs{\Psi(x)-\Psi(y)}\leq \abs{x-y}\),
 \item \label{it_Aequ9iek9ahh4xuye1aithed} \(F \compose \Psi \le F\) in \(\Rset^\nu\),
 \item \label{it_eroh3wae2ko7quooVuuThaed} \(\Psi = \id\) on \(\manifold{N}\),
\end{enumerate}
then for every \(g \in W^{1/2, 2} (\partial \Omega, \manifold{N})\),
if \(u\) is a minimiser for the Ginzburg--Landau energy under the boundary condition \( \tr_{\partial \Omega} u=g\),
then \(u \in \Bar{K}_\Psi\) almost everywhere in \(\Omega\),
where 
\[
 K_\Psi \defeq \biggl\{x \in \Rset^\nu \st \limsup_{h \to 0} \frac{\abs{\Psi (x + h) - \Psi (x)}} {\abs{h}} = 1\biggr\}.
\]
\end{proposition}

In particular, if \(F (R z/\abs{z}) \le F (z)\) for some \(R > 0\) and every \(z \in \Rset^\nu \setminus B_R\), taking 
\[\Psi (z) \defeq \begin{cases}  z  &\text{ if } \abs{z}<R \\
R z/\abs{z} &\text{ if } \abs{z} \geq R,
\end{cases}
\] we conclude that any minimiser of the Ginzburg--Landau energy \eqref{eq_GLenergy} satisfies \(\norm{u}_{L^\infty (\Omega)}\le R\). This is the case in particular when \(\manifold{N} = \Sset^1\) and \(F (z) = (1 - \abs{z}^2)^2/4\) \cite{Bethuel_Brezis_Helein_1993}*{Proposition 2}.

When the set \(K_\Psi\) is bounded, \cref{proposition_equation} also implies that \(u\) is a weak solution of the Ginzburg--Landau equation.

\begin{proof}[Proof of \cref{bounded_minimiser}]
If \(u\) is a minimiser, we set \(v = \Psi \compose u\). By \eqref{it_eroh3wae2ko7quooVuuThaed}, we have \(\tr_{\partial \Omega} v = \Psi \compose \tr_{\partial \Omega} u = g\) on \(\partial \Omega\).
Now, by \eqref{it_Aequ9iek9ahh4xuye1aithed} and since \(u\) is a minimiser, we have 
\[
\int_{\Omega} \frac{\abs{\Deriv u}^2}{2}
\le \int_{\Omega} \frac{\abs{\Deriv v}^2}{2}
+ \int_{\Omega} \frac{F (v)}{\varepsilon^2} - \int_{\Omega} \frac{F (u)}{\varepsilon^2}
\le \int_{\Omega} \frac{\abs{\Deriv v}^2}{2}.
\]
By \eqref{it_aipai9cae5ohxahn3Jong6oo} and by the chain rule for distributional derivatives \cite{Ambrosio_DalMaso_1990}, we have \(\abs{\Deriv v}^2 \le \abs{\Deriv u}^2\) almost everywhere in \(\Omega\),
and either \(\abs{\Deriv u}^2 = \abs{\Deriv v}^2 = 0\) or \(\abs{\Deriv v}^2 < \abs{\Deriv u}^2\) on \(u^{-1} (\Rset^\nu \setminus K_\Psi)\).
By optimality, this means that \(\Deriv u = 0\) a.e.\ on \(u^{-1} (\Rset^\nu \setminus K_\Psi)\); hence, by the chain rule, \(\Deriv \left( \dist( u(x),K_\psi)\right)=0\) a.e.\ on \(u^{-1} (\Rset^\nu \setminus K_\Psi)\). 
Since the weak derivative of \(\dist( u(\cdot),K_\psi)\) also vanishes a.e.\ on the zero level set, i.e.\ on \(u^{-1}(\Bar K_\psi)\), this implies that \(\Deriv \left( \dist( u(x),K_\psi)\right)=0\) a.e.\ in \(\Omega\). By \eqref{it_eroh3wae2ko7quooVuuThaed}, \(\manifold{N} \subset K_\psi\) and thus by the trace condition we find \( \dist( u,K_\psi)=0\) almost everywhere in \(\Omega\) which implies the conclusion.
\end{proof}

\subsection{Uniform convergence to the manifold}
Given a boundary datum \(g \in W^{1/2, 2} (\partial \Omega, \manifold{N})\), we show that the asymptotic vanishing of the penalisation term in the Ginzburg--Landau equation for a sequence of solutions implies that the distance to the manifold vanishes asymptotically uniformly.

\begin{theorem}
  \label{proposition_uniform_convergence_N}
 Let \(\Omega\) be a Lipschitz bounded domain, let \(g\in W^{1/2,2}(\partial \Omega,\manifold{N})\), and assume that \(F \in \mathcal{C}^1 (\Rset^\nu, [0,+\infty))\) satisfies \(F^{-1}(\{0\}) = \manifold{N}\) and \eqref{hyp20}. If \((\varepsilon_n)_{n \in \Nset}\) is a sequence in \((0,+\infty)\) converging to \(0\), if for every \(n \in \Nset\), \(u_n \in W^{1, 2} (\Omega, \Rset^\nu)\) is a solution to the Ginzburg--Landau equation \eqref{eq_GinzburgLandau} such that \(\tr_{\partial \Omega} u_n  = g\), and if for some \(a\in \Omega\) and \(\rho>0\), we have
  \[
  \lim_{n \to \infty}  \int_{\Omega \cap B_{\rho} (a)}\frac{ F (u_n)}{\varepsilon_n^2}= 0,
  \]
  and
  \[
   \sup_{n \in \Nset} \, \norm{\Deriv u_n}_{L^2 (\Omega \cap B_\rho (a))}
   +  \varepsilon_n^2\norm{\Delta u_n}_{L^\infty (\Omega \cap B_{\rho} (a))}
   < + \infty,
  \]
  then
\[
\lim_{n \to \infty} \norm{\dist (u_n, \manifold{N})}_{L^\infty (\Omega \cap B_{\rho/2} (a))} = 0.
\]
\end{theorem}

The assumptions of boundedness on \(\Deriv u_n\) in \(L^2\) and of convergence of \(F(u_n)/\varepsilon_n^2\) in \(L^1\) hold for sequences of minimisers away from singularities (\cref{theorem_main1}). The uniform bound on \(\Delta u_n\) follows from an a priori bound on \(u_n\) (see \cref{bounded_minimiser}) and the local boundedness of \(\nabla F\) in view of the Ginzburg--Landau equation \eqref{eq_GinzburgLandau}; it could also follow from the global boundedness of \(\nabla F\).

The uniform convergence to the vacuum manifold \(\mathcal{N}\) away from singularities was known for \(\manifold{N} = \Sset^1\) \cite{Bethuel_Brezis_Helein_1993}*{Step B.2}.
The result is reminiscent of uniform convergence of the modulus of \(W^{1, 2}\)--converging sequences of functions
whose Laplacian and whose modulus on the boundary are controlled \cite{Berlyand_Mironescu_Rybalko_Sandier_2014}*{Lemma 2.13}.

The next lemma states that harmonic functions tend uniformly to the image of their trace when we approach the boundary.

\begin{lemma}
  \label{lemma_harmonic_extension_average_close}
If \(\Omega\subset\Rset^2\) has a Lipschitz boundary, if \(v \in W^{1, 2} (\Omega, \Rset^\nu)\) satisfies \(-\Delta v = 0\) in \(\Omega\) and if \(\tr_{\partial\Omega}v \in \manifold{N}\) almost everywhere in \(\partial \Omega\), then
\[
\lim_{\substack{x \in \Omega\\\dist(x,\partial \Omega)\to 0}}\dist (v (x),\manifold{N}) = 0.
\]
\end{lemma}

\Cref{lemma_harmonic_extension_average_close} follows from the corresponding property for harmonic extensions of functions of vanishing mean oscillation (VMO) \cite{Brezis_Nirenber_1996}*{Theorem A3.2} and the embedding of \(W^{1/2, 2} (\partial \Omega,\Rset^\nu)\) in \(\VMO(\partial \Omega,\Rset^\nu)\) (see \cite{Berlyand_Mironescu_Rybalko_Sandier_2014}*{Lemma 2.12} for \(\manifold{N} = \Sset^1\)).
We give a direct proof when \(v\) is is \(W^{1, 2} (\Omega,\Rset^\nu)\).

\begin{proof}[Proof of \cref{lemma_harmonic_extension_average_close}]
  \resetconstant
Since the function \(v\) is harmonic, by the maximum principle, \(v\) is bounded. There exists a constant \(\Cl{cst_the8Ookaacheiy4phohw9Aew}<+\infty\) such that for every \(y \in \Omega\) (see for example \cite{Gilbarg_Trudinger_1983}*{Theorem 2.10}),
\begin{equation}\label{eq_meiX8ii4Quoilaeng6dohG2L}
  \abs{\Deriv v (y)}
  \le \frac{\Cr{cst_the8Ookaacheiy4phohw9Aew}\norm{v}_{L^\infty(\Omega)}}{\dist (y, \partial \Omega)}.
\end{equation}
For every \(x \in \Omega\), we let \(r \defeq \dist (x, \partial \Omega)\).
If \(0 < \eta  < 1\), we have 
\begin{equation}
\label{eq_Nah0iuwagh8seihaiGhae2Ee}
\dist (v (x), \manifold{N})
\le
\fint_{B_{\eta r} (x)} \abs{v (y) - v (x)}\dif y +  
\fint_{B_{\eta r} (x)} \dist (v (y), \manifold{N}) \dif y.
\end{equation} 
In view of  \eqref{eq_meiX8ii4Quoilaeng6dohG2L}, we have 
\begin{equation}
\label{eq_oaK1ka9Liegh8aeHieNa2aef}
 \fint_{B_{\eta r} (x)} \abs{v (y) - v (x)}\dif y
 \le \frac{\Cr{cst_the8Ookaacheiy4phohw9Aew}\norm{v}_{L^{\infty} (\Omega)}} {\frac{1}{\eta} - 1}.
\end{equation}
Next, since the set \(\Omega\) has a Lipschitz boundary and \(\tr_{\partial\Omega} \dist (v, \manifold{N}) = 0\), we have the following Poincaré inequality
\begin{equation}
\label{eq_Iat2Leis7eeJ2Ohphohx5oVo}
\int_{B_{2 r}(x) \cap \Omega} \dist  (v (y), v (\partial \Omega))^2
\le \C r^2 \int_{B_{2r}(x) \cap \Omega} \abs{D v}^2.
\end{equation}
It follows from \eqref{eq_Iat2Leis7eeJ2Ohphohx5oVo} that 
\begin{equation}
\label{eq_ohmi6Eph6dohnge1so7ochai}
 \fint_{B_{\eta r} (x)} \dist (v (y), \manifold{N}) \dif y
 \le 
 \biggl(\fint_{B_{\eta r} (x)} \dist (v (y), \manifold{N})^2 \dif y\biggr)^\frac{1}{2}
 \le \frac{\C}{\eta} \biggl(\int_{B_{2r}(x) \cap \Omega} \abs{D v}^2\biggr)^\frac{1}{2}.
\end{equation}
In order to conclude we observe that when \(\eta\) is small enough, the first-term in the right-hand side of \eqref{eq_Nah0iuwagh8seihaiGhae2Ee} can be made arbitrarily small by \eqref{eq_oaK1ka9Liegh8aeHieNa2aef}, while for any given \(\eta>0\) the second term in the right-hand side of \eqref{eq_Nah0iuwagh8seihaiGhae2Ee} goes to \(0\) as \(\dist(x,\partial \Omega)\to 0\)
in view of 
\eqref{eq_ohmi6Eph6dohnge1so7ochai} and Lebesgue's dominated convergence theorem since \(v\in W^{1,2}(\Omega,\Rset^\nu)\).
\end{proof}


\begin{lemma}[Regularity estimate]
\label{lemma_regularity_D2_Deltainfty}
If \(\Omega\) is a Lipschitz bounded domain and if \(w \in W^{1,2}(\Omega,\Rset^\nu)\) is such that \( \Delta w \in L^\infty(\Omega,\Rset^\nu)\) and \(\tr_{\partial \Omega}w = 0\), then for every \(\rho>0\), and for every \(a\in\Omega\), 
\[
\norm{\Deriv w}_{L^\infty (\Omega \cap B_{\rho/2} (a))}
\le
C (\rho) \bigl(\norm{\Deriv w}_{L^2 (\Omega \cap B_\rho (a))}+ \norm{\Delta w}_{L^\infty (\Omega \cap B_{\rho} (a))}\bigr)^\frac{1}{2}
\norm{\Deriv w}_{L^2 (\Omega\cap B_{\rho} (a))}^\frac{1}{2}.
\]
\end{lemma}

\Cref{lemma_regularity_D2_Deltainfty} is reminiscent of the \(L^\infty\) estimates \cite{Bethuel_Brezis_Helein_1993}*{Lemma A.1}.

\begin{proof}[Proof of \cref{lemma_regularity_D2_Deltainfty}]
\resetconstant
We fix \(p > 2\).
Since \(\partial \Omega\) has a Lipschitz boundary, there exists \(\rho_0>0\) such that if \(2\dist (x, \partial \Omega) \le r \le \rho_0\), then \(B_r(x) \cap \Omega\) is homeomorphic to a half-disk and has uniformly Lipschitz boundary. By a finite covering argument, we can assume that \(\rho \le \rho_0\).

If \(x \in \Omega\) and \(r \in (0,  \rho) \setminus (\dist (x, \partial \Omega), 2 \dist (x, \partial \Omega))\),
we have by classical Calder\'on--Zygmund estimates and a scaling argument,
\begin{equation}
 \norm{\Deriv^2 w}_{L^p (\Omega \cap B_{r/2} (x))}
 \le \C \biggl(\frac{\norm{w}_{L^p (\Omega \cap B_{r} (x))}}{r^2}
 + \norm{\Delta w}_{L^p (\Omega \cap B_r (x))}\biggr);
\end{equation}
combining it with the two-dimensional Sobolev inequality
\[\norm{w}_{L^p(\Omega \cap B_r(x))}\leq \C r^{{2}/{p}} (r^{-1}\norm{w}_{L^2(\Omega \cap B_r(x))}+
\norm{\Deriv w}_{L^2(\Omega \cap B_r(x))} ),
\]
and the Poincar\'e inequality 
\[
 \norm{w}_{L^2(\Omega \cap B_r(x)))}\leq \C r\norm{\Deriv w}_{L^2(\Omega \cap B_r(x))}
 \] 
(which is true when \(r\geq 2 \dist (x, \partial \Omega)\) since \(w=0\) on \(\partial \Omega\), and also when \(r\leq\dist (x, \partial \Omega)\) if we assume without loss of generality that \(\int_{\Omega \cap B_r(x)} w = 0\)), we obtain 
\begin{equation}
\label{eq_shoh3shahpu9foo5Ou7aWaix}
  \norm{\Deriv^2 w}_{L^p (\Omega \cap B_{r/2} (x))}
 \le \C \biggl(\frac{\norm{\Deriv w}_{L^2 (\Omega \cap B_{r} (x))}}{r^{2 - \frac{2}{p}}}
 + \norm{\Delta w}_{L^p (\Omega \cap B_r (x))}\biggr);
\end{equation}
moreover, by the Morrey--Sobolev embedding \(W^{1,p} \subset \mathcal{C}^{0, 1 - 2/p}\) and the Cauchy--Schwarz inequality,
\begin{equation}
\label{eq_ohTae3So4ahShive1aqu0wig}
\begin{split}
\abs{\Deriv w (x)}
 &\le \fint_{\Omega \cap B_{r/2} (x)} \abs{\Deriv w (x) - \Deriv w (y)} \dif y + \fint_{\Omega \cap B_{r/2} (x)} \abs{\Deriv w}\\
 &\le \C \biggl(\norm{\Deriv^2 w}_{L^p(\Omega \cap B_{r/2} (x))} r^{1 - \frac{2}{p}} + \frac{\norm{\Deriv w}_{L^2(\Omega \cap B_{r/2} (x))}}{r}\biggr),
 \end{split}
\end{equation}
and it follows thus from \eqref{eq_shoh3shahpu9foo5Ou7aWaix}, \eqref{eq_ohTae3So4ahShive1aqu0wig} and \(\norm{\Delta w}_{L^p (\Omega \cap B_r (x))} \leq (\pi r^2)^{1/p}\norm{\Delta w}_{L^\infty (\Omega \cap B_r (x))}\) that
\begin{equation}
\label{eq_rohd4vies9ohkeuYiem5OoJe}
\abs{\Deriv w (x)} \le \Cl{cst_ouw2weeQuaiy5aroo9goh6wu} \biggl(r \norm{\Delta w}_{L^\infty (\Omega \cap B_{r/2} (x))} + \frac{\norm{\Deriv w}_{L^2(\Omega \cap B_{r/2} (x))}}{r} \biggr).
\end{equation}

Now, if \(x\in \Omega\cap B_{\rho/2}(a)\) then, for every  \(r \in (0,  \rho) \setminus (\dist (x, \partial \Omega), 2 \dist (x, \partial \Omega))\), we have the inclusion \(\Omega\cap B_{r/2}(x)\subset\Omega\cap B_{\rho}(a)\), and we deduce from \eqref{eq_rohd4vies9ohkeuYiem5OoJe} that
\begin{equation}
\label{eq_ieornajfnkefionaio}
\abs{\Deriv w (x)} \le \Cr{cst_ouw2weeQuaiy5aroo9goh6wu} \biggl(r \norm{\Delta w}_{L^\infty (\Omega \cap B_{\rho}(a) )} + \frac{\norm{\Deriv w}_{L^2(\Omega \cap B_{\rho} (a))}}{r} \biggr).
\end{equation}

We observe now that \eqref{eq_rohd4vies9ohkeuYiem5OoJe} holds also for \(r \in (0, \rho)\cap [\dist (x, \partial \Omega), 2\dist (x, \partial \Omega)]  \) with \( 2\Cr{cst_ouw2weeQuaiy5aroo9goh6wu}\) instead of \(\Cr{cst_ouw2weeQuaiy5aroo9goh6wu}\) and we conclude by taking 
\[
 r \defeq \min \Biggl( \rho/2, \sqrt{\frac{\norm{\Deriv w}_{L^2(\Omega \cap B_{\rho} (a))}}{\norm{\Delta w}_{L^\infty (\Omega \cap B_\rho (a))}}}\Biggr)
 \qedhere
\]
if \(\norm{\Deriv w}_{L^2(\Omega \cap B_{\rho} (a))}\norm{\Delta w}_{L^\infty (\Omega \cap B_\rho (a))}>0\), and \(r\defeq \rho/2\) otherwise.
\end{proof}


\begin{proof}%
[Proof of \cref{proposition_uniform_convergence_N}]%
\resetconstant%
Following \cite{Bethuel_Brezis_Helein_1993}*{Proof of Step B.1},
let \(v\in W^{1,2}(\Omega,\Rset^\nu)\) be a solution to the Dirichlet problem
\[
\left\{
\begin{aligned}
  \Delta v & = 0 && \text{in \(\Omega\)},\\
  v &= g && \text{on \(\partial \Omega\)}.
\end{aligned}
\right.
\]
For each \(n  \in \Nset\), we define the function \(w_n \defeq  u_n - v\), which satisfies by assumption on \(u_n\) and by construction of \(v\),
\[
\left\{
\begin{aligned}
  \Delta w_n & = \frac{\nabla F (u_n)}{\varepsilon_n^2} & & \text{in \(\Omega\)},\\
  w_n& = 0 & & \text{on \(\partial \Omega\)}.
\end{aligned}
\right.
\]
By \cref{lemma_regularity_D2_Deltainfty} and by assumption, we have
\begin{multline}
  \label{eq_ejaeDoo8aequ9tena}
\norm{\Deriv w_n}_{L^\infty (\Omega\cap B_{\rho/2}(a))}\\
\le 
\C \bigl(\norm{\Deriv w_n}_{L^2 (\Omega\cap B_\rho(a))}+ \norm{\Delta w_n}_{L^\infty (\Omega\cap B_\rho(a))}\bigr)^\frac{1}{2}
\norm{\Deriv w_n}_{L^2 (\Omega\cap B_\rho(a))}^\frac{1}{2} 
\le  \frac{\Cl{cst_Vahch0su1noi5oing}}{\varepsilon_n}.
\end{multline}

Let now \(\delta \in (0,\frac{\delta_{\manifold{N}}}{4})\).
By \cref{lemma_harmonic_extension_average_close}, there exists \(r > 0\), such that
if \(\dist (x, \partial \Omega) \le r\), then \(\dist (v (x), \manifold{N})\le \delta/2\).
If moreover \(x\in \Omega\cap B_{\rho/2}(a)\) and
\(\dist (x, \partial \Omega) \le \varepsilon_n\delta/(2\Cr{cst_Vahch0su1noi5oing})<r\), then, thanks to \eqref{eq_ejaeDoo8aequ9tena},
we have \(\abs{w_n (x)} \le \delta/2\). Hence for \(n\) large enough, as \(u_n=w_n+v\),
\begin{equation}
\label{eq_hr8402io90jio}
\text{for all \(x\in \Omega\cap B_{\rho/2}(a)\) such that
\(\dist (x, \partial \Omega) \le \varepsilon_n\delta/(2\Cr{cst_Vahch0su1noi5oing})\)},\ \dist (u_n (x), \manifold{N}) \le \delta.
\end{equation}

We consider now a point \(x\in \Omega\cap B_{\rho/2}(a)\) such that \(\dist (x, \partial \Omega) > \varepsilon_n\delta/(4 \Cr{cst_Vahch0su1noi5oing})\); we have by classical estimates on harmonic extensions (see \eqref{eq_meiX8ii4Quoilaeng6dohG2L}) and by \eqref{eq_ejaeDoo8aequ9tena}
\begin{equation}
\label{eq_rohlua3ovohcoh2Ia6Euguci}
\abs{\Deriv u_n (x)} \le \frac{\Cl{cst_Gee0Sahcaix2Eif3uD6sahk5}}{\varepsilon_n \delta}.
\end{equation}
We assume now by contradiction that there exists a sequence \( (a_n)_{n \in \Nset}\) in \(B_{\rho/2}(a)\cap \Omega\) such that \( \dist(u_n(a_n), \mathcal{N}) \geq 2\delta\). By continuity, we can assume that \(\dist(u_n (a_n),\mathcal{N}) = 2\delta\). By \eqref{eq_hr8402io90jio}, we have in particular \(\dist (a_n, \partial \Omega)>\varepsilon_n\delta/(2\Cr{cst_Vahch0su1noi5oing})\) and so for \(n\) large enough,
\[
B_{\varepsilon_n\delta/(4 \Cr{cst_Vahch0su1noi5oing})}(a_n)\subset\{x\in\Omega\cap B_{\rho}(a)\st \dist (x, \partial \Omega)>\varepsilon_n\delta/(4\Cr{cst_Vahch0su1noi5oing})\}.
\]
Since the distance to a closed set is non-expansive, using \eqref{eq_rohlua3ovohcoh2Ia6Euguci}, we have if \(x\in B_{\varepsilon_n\delta/(4 \Cr{cst_Vahch0su1noi5oing})}(a_n)\),
\[
 \abs{\dist (u_n (x), \manifold{N}) - \dist (u_n (a_n), \manifold{N})}
 \le \abs{u_n (x) - u_n (a_n)}
 \le \frac{\Cr{cst_Gee0Sahcaix2Eif3uD6sahk5}}{\varepsilon_n \delta} \abs{x - a_n}.
\]
Hence, for every \(x\in  B_{\Cl{cst_9403jfo930fji}\varepsilon_n\delta^2}(a_n)\) with \(\Cr{cst_9403jfo930fji}\defeq\inf\{1/(\delta_\manifold{N} \Cr{cst_Vahch0su1noi5oing});1/\Cr{cst_Gee0Sahcaix2Eif3uD6sahk5}\}\), we have
\[
  \delta \le \dist (u_n (x), \manifold{N})\le 3 \delta < \delta_{\manifold{N}}.
\]
Hence, we have if \(n\) is large enough, using \eqref{hyp20},
\[
\frac{m_F}{2}\Cr{cst_9403jfo930fji}^2\pi\delta^6 \le \frac{m_F}{2}\int_{B_{\Cr{cst_9403jfo930fji}\varepsilon_n\delta^2}(a_n)} \frac{\dist (u_n, \manifold{N})^2}{\varepsilon_n^2}
 \le \int_{\Omega \cap B_\rho (a)} \frac{F (u_n)}{\varepsilon_n^2},
\]
which cannot hold by assumption if \(n \in \Nset\) is large enough since the right-hand side goes to zero.
\end{proof}


\subsection{Weak convergence of solutions}
The next result shows that, under some assumptions, the limit of weakly converging sequences of solutions to the Ginzburg--Landau equation are harmonic maps. For a related result when \(\manifold{N}\) is a compact manifold of dimension \(1\) we refer to \cite{Mironescu_Shafrir_2017}. 

\begin{theorem}
\label{proposition_convergence_solutions}
Let \(\Omega\) be a Lipschitz bounded domain and assume that \(F \in \mathcal{C}^1 (\Rset^\nu, [0,+\infty))\) satisfies \(F^{-1}(\{0\}) = \manifold{N}\) and \eqref{hyp20}. Assume also that \(g\in W^{1/2,2}(\partial \Omega,\manifold{N})\), that \((\varepsilon_n)_{n \in \Nset}\) is a sequence in \((0,+\infty)\) converging to \(0\) and that for every \(n \in \Nset\), \(u_n \in W^{1, 2} (\Omega, \Rset^\nu)\) is a solution to the Ginzburg--Landau equation \eqref{eq_GinzburgLandau} with \(\tr_{\partial \Omega} u_n  = g\). If for some \(a\in \Omega\) and \(\rho>0\), we have
\begin{enumerate}[(i)]
\item\label{eq_nnif40289ue2e} \((u_n \vert_{\Omega \cap B_\rho (a)})_{n \in \Nset}\) converges weakly to some limit \(u\) in \(W^{1, 2} (\Omega \cap B_\rho (a), \Rset^\nu)\),
\item 
\label{eq:aavdaeoihvqihf}
\(\lim_{n \to \infty} \norm{\dist (u_n, \manifold{N})}_{L^\infty (\Omega \cap B_{\rho} (a))} = 0\),
\item\label{eq_Vohyeajee6eig6ohvoothamu}\(\lim_{n \to \infty} \int_{\Omega\cap B_\rho (a)} \frac{\abs{D \Pi_{\manifold{N}}(u_n) [\nabla F (u_n)]}} {\varepsilon_n^2} = 0\),
\end{enumerate}
then \(u\) is a \(\manifold{N}\)-valued harmonic map in \(\Omega \cap B_\rho (a)\).
\end{theorem}

For the classical Ginzburg--Landau equation, which corresponds to \eqref{eq_GinzburgLandau} when \(F(z)=(1-\abs{z}^2)^2\), we have \(\nabla F(z)=-4(1-\abs{z}^2)z\), \(\Pi_\manifold{N}(z)=\frac{z}{\abs{z}}\) and \(\Deriv \Pi_\manifold{N}(z)=\frac{\operatorname{id}}{\abs{z}}-\frac{z\otimes z}{\abs{z}^3}\); hence,
\[
D \Pi_{\manifold{N}}(u_n) [\nabla F (u_n)] =-4(1-\abs{u_n}^2)\frac{u_n}{\abs{u_n}}+4(1-\abs{u_n}^2)\frac{(u_n\cdot u_n)u_n}{\abs{u_n}^3}= 0
\]
which implies  \eqref{eq_Vohyeajee6eig6ohvoothamu}. 
We recover from \cref{theorem_compactnessweak}, \cref{proposition_uniform_convergence_N} and \cref{proposition_convergence_solutions} that solutions to the classical Ginzburg--Landau satisfying an a priori bound (see \cref{bounded_minimiser}) and an upper-bound of the form \eqref{eq_borne_energie} converge to a harmonic map with values into \(\mathbb{S}^1\) outside a finite set of singularities when \(\varepsilon\) goes to zero \cite{Bethuel_Brezis_Helein_1994}*{Theorem X.1}.

When \(F \in \mathcal{C}^3 (\Rset^\nu)\), the condition \eqref{eq_Vohyeajee6eig6ohvoothamu} in \cref{proposition_convergence_solutions} follows from
\(\lim_{n \to \infty} \int_{\Omega\cap B_\rho (a)} \frac{F (u_n)}{\varepsilon_n^2} = 0\), in view of the next lemma:

\begin{lemma}
  \label{lemma_isotropic_around_N}
  Let \(F \in \mathcal{C}^3 (\manifold{N}_{\delta_{\manifold{N}}},[0,+\infty))\). If \(F^{-1} (\{0\}) = \manifold{N}\) and \(F\) satisfies \eqref{hyp20}, then there exist constants \(C\in (0,+\infty)\) and \(\delta\in (0,\delta_\manifold{N})\) such that for every \(z \in \manifold{N}_{\delta}\),
  \begin{equation}
    \abs{\Deriv\Pi_{\manifold{N}}(z) [\nabla F (z)]}
    \le C F(z).
\end{equation}
\end{lemma}

\begin{proof}[Proof of \cref{lemma_isotropic_around_N}]%
\resetconstant%
By a second-order Taylor expansion of \(\Deriv F(z)\), we have for \(v \in \Rset^\nu\)
\begin{multline}
\label{eq_fae7reoD2ohneeF7Ahth8Yie}
\Deriv F (z) [\Deriv \Pi_{\manifold{N}}(z)[v]] 
= \Deriv F (\Pi_{\manifold{N}} (z)) [\Deriv \Pi_{\manifold{N}}(z)[v]]\\
+ \Deriv^2F(\Pi_\manifold{N}(z))[z-\Pi_\manifold{N}(z),
\Deriv\Pi_{\manifold{N}}(z)[v]]
+O(\abs{v}\dist(z,\manifold{N})^2).
\end{multline}
We first have for every \(z \in \manifold{N}_{\delta}\),
\begin{equation}
\label{eq_lie3yieJ1eeph2Oodeeng2ie}
\Deriv F (\Pi_{\manifold{N}}(z)) = 0, 
\end{equation}
so that the first term in the right-hand side of \eqref{eq_fae7reoD2ohneeF7Ahth8Yie} vanishes.
Differentiating \eqref{eq_lie3yieJ1eeph2Oodeeng2ie}, we get for \(v, w \in \Rset^\nu\) and \(z \in \manifold{N}_\delta\),
\begin{equation}
\label{eq_haeJa5aik8yohj5oowughaja}
 \Deriv^2 F (\Pi_{\manifold{N}}(z))[w, \Deriv \Pi_{\manifold{N}}(z) [v]] = 0,
\end{equation}
so that the second term in the right-hand side of \eqref{eq_fae7reoD2ohneeF7Ahth8Yie} also vanishes.
We deduce from \eqref{eq_fae7reoD2ohneeF7Ahth8Yie}, \eqref{eq_lie3yieJ1eeph2Oodeeng2ie} and \eqref{eq_haeJa5aik8yohj5oowughaja} that 
\[
\abs{
 \Deriv \Pi_{\manifold{N}} (z)^*[\nabla F (z)]
 }
 \le \C \dist (z, \manifold{N})^2.
\]
Since \(\Deriv \Pi_{\manifold{N}} (z)\) is self-adjoint and \(\nabla F(z)=0\) when \(z \in \manifold{N}\), we have 
\[
 \abs{\Deriv \Pi_{\manifold{N}}(z)^*[\nabla F (z)] - \Deriv \Pi_{\manifold{N}} (z)[\nabla F (z)]}
 \le \C \dist (z, \manifold{N})^2, \text{ for all } z \in \manifold{N}_{\delta}\qedhere
\]
and the conclusion follows, in view of \eqref{hyp20}.
\end{proof}

We begin the proof of \cref{proposition_convergence_solutions} with the following geometrical identity for the nearest-point projection:
\begin{lemma}
\label{lemma_projection}
For every \(y \in \manifold{N}\), \(h \in \Rset^\nu\) and \(w \in T_y \manifold{N}\), we have 
\begin{equation}
 w \cdot \Deriv^2 \Pi_{\manifold{N}}(y)[h, h]
 =w\cdot \Deriv^2 \Pi_{\manifold{N}} (y)[h, \Deriv \Pi_{\manifold{N}} (y)[h]]
 + \Deriv\Pi_{\manifold{N}} (y)[h] \cdot 
 \Deriv^2 \Pi_{\manifold{N}} (y)[h, w].
\end{equation}
\end{lemma}
\begin{proof}
Setting \(h^\top \defeq \Deriv \Pi_{\manifold{N}} (y)[h]\) and \(h^\perp \defeq h - h^\top\), we need to prove that the following quantity vanishes: 
\begin{multline}
\label{eq_seiTe4oJaovejaehuchahgho}
  w\cdot \Deriv^2 \Pi_{\manifold{N}} (y)[h, h]- w  \cdot \Deriv^2 \Pi_{\manifold{N}} (y)[h, h^\top] - h^\top \cdot \Deriv^2 \Pi_{\manifold{N}} (y)[h, w]\\
  =w \cdot \Deriv^2 \Pi_{\manifold{N}} (y)[h^\top, h^\perp] +w \cdot  \Deriv^2 \Pi_{\manifold{N}} (y)[h^\perp, h^\perp] \\
  - h^\top \cdot \Deriv^2 \Pi_{\manifold{N}} (y)[h^\top, w] 
  - h^\top \cdot \Deriv^2 \Pi_{\manifold{N}} (y)[h^\perp, w].
\end{multline}
Since \(h^\perp \in T_y^\perp \manifold{N}\), and since \(\Pi_{\manifold{N}}(y+t h^\perp)=\Pi_\manifold{N}(y)\) for all \(t\) small enough , we have, by differentiating twice:
\begin{equation}
\label{eq_Chee2pheeBeihoPoo6As1jai}
\Deriv^2 \Pi_{\manifold{N}} (y)[h^\perp, h^\perp] = 0.
\end{equation}
By the connection between the nearest-point projection and the second fundamental form \cite{Moser_2005}*{lemma 3.2}, we have
\begin{equation}
\label{eq_iocimoXaereeghooShool0ph}
 w \cdot \Deriv^2 \Pi_{\manifold{N}} (y)[h^\perp, h^\top]
 = - h^\perp \cdot B_y(w, h^\top)
 = - h^\top \cdot \Deriv^2 \Pi_{\manifold{N}} (y)[h^\perp, w].
\end{equation}
Finally, since \(w \in T_y \manifold{N}\), \(h^\top \in T_y \manifold{N}\) and by using that \(\Deriv^2\Pi_\manifold{N}(y):T_y\manifold{N} \otimes T_y\manifold{N} \rightarrow T_y^\perp\manifold{N}\) we have 
\begin{equation}
\label{eq_sheash8Aezeohai2ohR6kash}
 h^\top \cdot \Deriv^2 \Pi_{\manifold{N}} (y)[h^\top, w] = 0.
\end{equation}
In view of \eqref{eq_Chee2pheeBeihoPoo6As1jai} ,\eqref{eq_iocimoXaereeghooShool0ph}
and
\eqref{eq_sheash8Aezeohai2ohR6kash}, the right-hand side of \eqref{eq_seiTe4oJaovejaehuchahgho} vanishes and the conclusion follows.
\end{proof}

\begin{lemma}
  \label{lemma_project_equations}
For every \(y\in \manifold{N}_{\delta_{\manifold{N}}}\), the map \(\alpha_{\manifold{N}} (y) \defeq \Deriv \Pi_{\manifold{N}} (y) \Deriv \Pi_{\manifold{N}} 
(y)^*: T_{\Pi_\manifold{N}(y)}\manifold{N} \rightarrow T_{\Pi_\manifold{N}(y)}\manifold{N}\) is invertible. Moreover, if a map \(u \in W^{2, 1}_{\mathrm{loc}}(\Omega, \Rset^\nu)\) satisfies \(\norm{\dist(u, \manifold{N})}_{L^\infty(\Omega)}< \delta_{\manifold{N}}\), then we have
\begin{equation*}
  \left|\alpha_{\manifold{N}}(u)^{-1} \Deriv\Pi_{\manifold{N}} (u)[\Delta u ]
 -\operatorname{div} [\alpha_{\manifold{N}}(u)^{-1} D(\Pi_{\manifold{N}} \compose u)]\right|\\
 \le C \abs{u - \Pi_{\manifold{N}}(u)} \abs{\Deriv u}^2.
\end{equation*}
\end{lemma}

Here, we recall that \(\Deriv\Pi_\manifold{N}(y)^*: T_{\Pi_
\manifold{N}(y)}\manifold{N}\rightarrow T_{\Pi_
\manifold{N}(y)}\manifold{N}\subset\Rset^\nu \) stands for the adjoint of \(\Deriv\Pi_\manifold{N}(y)\) which is defined by 
\[
\Deriv \Pi_{\manifold{N}}(y)[v] \cdot w 
 = v\cdot \Deriv\Pi_{\manifold{N}}(y)^*[w] \quad\text{for all \(v\in \Rset^\nu\) and \(w \in T_{\Pi_{\manifold{N}}(y)}\manifold{N}\).}
 \]
 
\Cref{lemma_project_equations} is a generalisation of the decomposition when \(\manifold{N} = \Sset^1\) of \(u\)
into its modulus and argument \cite{Bethuel_Brezis_Helein_1993}*{(51)-52)}, which is connected to the substitution in the Schrödinger equation to obtain Madelung equations, see e.g.\ \cite{Carles_Danchin_Saut_2012} and references therein. 

\begin{proof}[Proof of \cref{lemma_project_equations}]
\resetconstant
First of all, we have that \(\alpha_\manifold{N}(y)\) is invertible since \(\Deriv\Pi_\manifold{N}(y):T_{\Pi_
\manifold{N}(y)}\manifold{N}\rightarrow T_{\Pi_
\manifold{N}(y)}\manifold{N}\) is onto. 

Let \(i \in \{1, 2\}\). We have on the one hand
\begin{equation}
\label{eq_dohXu2oon6Cei5phiehu8ang}
\begin{split}
\partial_i^2 (\Pi_{\manifold{N}} \compose u)
&  = \partial_i\left(\alpha_{\manifold{N}}(u)\alpha_{\manifold{N}} (u)^{-1} \partial_i (\Pi_{\manifold{N}} \compose u)\right)  \\
&= \alpha_{\manifold{N}} (u) \partial_i \bigl( \alpha_{\manifold{N}} (u)^{-1} \partial_i (\Pi_{\manifold{N}} \compose u)\bigr)
+ \partial_i (\alpha_{\manifold{N}} (u)) \alpha_{\manifold{N}} (u)^{-1} \partial_i(\Pi_{\manifold{N}} \compose u),
\end{split}
\end{equation}
with
\begin{equation}
\label{eq_obae3meixoh8lo9Ieh3ooghu}
 \partial_i (\alpha_{\manifold{N}} (u))
 = \Deriv^2 \Pi_{\manifold{N}} (u)[\partial_i u] \compose \Deriv \Pi_{\manifold{N}} (u)^*
 + \Deriv \Pi_{\manifold{N}} (u) \compose \bigl(\Deriv^2 \Pi_{\manifold{N}} (u)[\partial_i u]\bigr)^*.
\end{equation}
On the other hand, we have
\begin{equation}
\label{eq_BohG5quaoso7beec7yeatiex}
 \partial_i^2 (\Pi_{\manifold{N}} \compose u) = \partial_i (\Deriv \Pi_{\manifold{N}} (u)[\partial_i u]) =
 \Deriv^2 \Pi_{\manifold{N}} [\partial_i u, \partial_i u] +
 \Deriv \Pi_{\manifold{N}}(u) [\partial_i^2 u],
\end{equation}
and therefore by \eqref{eq_dohXu2oon6Cei5phiehu8ang}, \eqref{eq_obae3meixoh8lo9Ieh3ooghu} and \eqref{eq_BohG5quaoso7beec7yeatiex}, we have
\begin{multline}
\label{eq_ohthieY1iibibiega4tiv7Ce}
 \Deriv \Pi_{\manifold{N}} (u)[\partial^2_i u]
 - \alpha_{\manifold{N}} (u) \partial_i\bigl( \alpha_{\manifold{N}} (u)^{-1} \partial_i (\Pi_{\manifold{N}} \compose u)\bigr)\\
 = \bigl(\Deriv^2 \Pi_{\manifold{N}} (u)[\partial_i u] \compose \Deriv \Pi_{\manifold{N}} (u)^*
 + \Deriv \Pi_{\manifold{N}} (u) \compose \Deriv^2 \Pi_{\manifold{N}} (u)[\partial_i u]^*\bigr)[\alpha_{\manifold{N}} (u)^{-1} \partial_i(\Pi_{\manifold{N}} \compose u)]\\-
 \Deriv^2 \Pi_{\manifold{N}} (u)[\partial_i u, \partial_i u].
\end{multline}
Since the left-hand side of \eqref{eq_ohthieY1iibibiega4tiv7Ce} lies in \(T_{\Pi_\manifold{N}(u)} \manifold{N}\), it suffices to estimate the projection of the right-hand side of \eqref{eq_ohthieY1iibibiega4tiv7Ce} on \(T_{\Pi_\manifold{N}(u)} \manifold{N}\).

Since \(\Deriv\Pi_\manifold{N}(\Pi_\manifold{N}(u))\) is the orthogonal projection onto the tangent space \(T_{\Pi_{\manifold{N}}(u)}\manifold{N}\), we have that both \(\Deriv\Pi_\manifold{N}(\Pi_\manifold{N}(u))=\Deriv\Pi_\manifold{N}(\Pi_\manifold{N}(u))^\ast\) and \(\alpha_\manifold{N}(\Pi_\manifold{N}(u))\) are the identity on \(T_{\Pi_\manifold{N}(u)}\manifold{N}\). Hence, by using a Taylor expansion, we have for every \(v \in T_{\Pi_\manifold{N}(u)}\mathcal{N}\),
\begin{multline}\label{eq:arupvhaae}
v\cdot\Deriv^2 \Pi_{\manifold{N}} (u)[\partial_i u] \left[\Deriv \Pi_{\manifold{N}} (u)^*\alpha_{\manifold{N}} (u)^{-1} \partial_i(\Pi_{\manifold{N}} \compose u)\right] \\
=v\cdot\Deriv^2\Pi_\manifold{N}\left(\Pi_\manifold{N}(u) \right)\left[\partial_iu,\Deriv\Pi_\manifold{N}\left( \Pi_\manifold{N}(u)\right)[\partial_iu]\right]+O(\abs{v}\abs{u-\Pi_\manifold{N}(u)}\abs{\partial_iu}^2),
\end{multline}
\begin{multline}\label{eq:ouaegva}
v \cdot \Deriv \Pi_{\manifold{N}} (u) \compose \Deriv^2 \Pi_{\manifold{N}} (u)[\partial_i u]^*[\alpha_{\manifold{N}} (u)^{-1} \partial_i(\Pi_{\manifold{N}} \compose u)] \\
=  \Deriv^2\Pi_\manifold{N}(u)[\partial_iu,\Deriv\Pi_{\manifold{N}}(u)^*[v]] \cdot \alpha_{\manifold{N}}(u)^{-1} \partial_i(\Pi_\manifold{N}\compose u) \\
= \Deriv^2\Pi_\manifold{N}(\Pi_\manifold{N}(u))[\partial_iu,v] \cdot \partial_i (\Pi_\manifold{N}\compose u)+O(\abs{v}\abs{u-\Pi_\manifold{N}(u)}\abs{\partial_iu}^2)
\end{multline}
and, in view of \cref{lemma_projection},
\begin{equation}\label{eq:azgiahzg}
\begin{split}
v\cdot \Deriv^2\Pi_\manifold{N}(u)[\partial_iu,\partial_iu]
&=v\cdot \Deriv^2\Pi_{\manifold{N}}(\Pi_\manifold{N}(u))[\partial_iu,\partial_iu]+O(\abs{v}\abs{u-\Pi_\manifold{N}(u)}\abs{\partial_iu}^2)\\
&= v \cdot \Deriv^2\Pi_\manifold{N}\left(\Pi_\manifold{N}(u) \right)\left[\partial_iu,\Deriv\Pi_\manifold{N}\left( \Pi_\manifold{N}(u)\right)[\partial_iu]\right]\\
&\qquad+\Deriv^2\Pi_\manifold{N}(\Pi_\manifold{N}(u))[\partial_iu,v] \cdot \Deriv \Pi_\manifold{N}(\Pi_\manifold{N}(u))[\partial_iu] \\
&\qquad\qquad+O(\abs{v}\abs{u-\Pi_\manifold{N}(u)}\abs{\partial_iu}^2).
\end{split}
\end{equation}
Hence from \eqref{eq_ohthieY1iibibiega4tiv7Ce}, \eqref{eq:arupvhaae}, \eqref{eq:ouaegva} and \eqref{eq:azgiahzg} we arrive at
\begin{equation}
\label{eq_QuohzuvaijieneCh2dieZiet}
 \abs{\alpha_{\manifold{N}} (u)^{-1}\Deriv\Pi_{\manifold{N}} (u)[\partial^2_i u]
 - \partial_i  \bigl(\alpha_{\manifold{N}} (u)^{-1} \partial_i(\Pi_{\manifold{N}} \compose u)\bigr)}
 \le
 \C \abs{u - \Pi_{\manifold{N}}(u)} \abs{\partial_i u}^2.
\end{equation}
The conclusion then follows by the triangle inequality and summing \eqref{eq_QuohzuvaijieneCh2dieZiet} over \(i \in \{1, 2\}\).
\end{proof}

\begin{proof}%
[Proof of \cref{proposition_convergence_solutions}]%
\resetconstant
By classical regularity estimates, we have \(u_n \in W^{2, p} (\Omega,\Rset^\nu )\) for every \(p\in (1,+\infty)\).
It follows from our assumption \(\lim_{n \to \infty} \norm{\dist (u_n, \manifold{N})}_{L^\infty (\Omega \cap B_{\rho} (a))} = 0\), that for \(n \in \Nset\) large enough  we have \(\norm{\dist (u_n, \manifold{N})}_{L^\infty (\Omega \cap B_{\rho} (a))}<\delta_\manifold{N}\) so that we can define \(v_n \defeq \Pi_{\manifold{N}} \compose u_n\vert_{\Omega \cap B_\rho (a)}\).  By smoothness of \(\Pi_\manifold{N}\) and the assumption \eqref{eq_nnif40289ue2e}, we know that the sequence \((v_n)_{n\in\Nset}\) converges to \(u\) weakly in \(W^{1,2}(\Omega\cap B_\rho(a),\Rset^\nu)\). Moreover, we have
\begin{equation}\label{weakPalaisSmale}
\Deriv\Pi_{\manifold{N}}(v_n)[\Delta v_n]= f_n+g_n,
\end{equation}
where, using the same notation \(\alpha_{\manifold{N}}(y)\defeq \Deriv\Pi_{\manifold{N}}(y) \Deriv\Pi_{\manifold{N}}(y)^*\) as in \cref{lemma_project_equations},
\[
f_n\defeq\Deriv \Pi_{\manifold{N}}(v_n) \left[\operatorname{div} \left(\left(\operatorname{id} -
 \alpha_{\manifold{N}}(u_n)^{-1}\right) \Deriv v_n\right)\right]
 \]
 and
\[
g_n=\Deriv\Pi_\manifold{N}(v_n)\operatorname{div}\left(\alpha_{\manifold{N}}(u_n)^{-1}\Deriv v_n\right).
 \]

By weak convergence, \((\Deriv v_n)_{n\in\Nset}\) is bounded in \(L^2\). Using the fact that \(\alpha_{\manifold{N}}(y)\) depends smoothly on \(y\in\manifold{N}_{\delta_\manifold{N}}\) and that \(\alpha_{\manifold{N}}(y)=\operatorname{id}_{T_y\manifold{N}}\) when \(y\in\manifold{N}\), we find by our assumption \eqref{eq:aavdaeoihvqihf},
\begin{equation*}
 \lim_{n \to \infty} \norm{(\operatorname{id} -
  \alpha_{\manifold{N}}(u_n)^{-1}) \Deriv v_n}_{L^2 (\Omega \cap B_\rho (a))} = 0,
\end{equation*}
and we deduce that
\begin{equation}
\label{eq_convergenLone}
\norm{f_n}_{H^{-1}(\Omega\cap B_\rho(a),\Rset^\nu)}\underset{n\to\infty}{\longrightarrow} 0.
\end{equation}

Now, we have from \cref{lemma_project_equations} and by smoothness of \(\Deriv\Pi_\manifold{N}\),
\begin{equation*}
\norm{g_n}_{L^1}\leq \Cl{cst_fi0jewodck0ok}\left(\norm{\alpha_{\manifold{N}}^{-1}(u_n)\Deriv\Pi_{\manifold{N}}(u_n)\Delta u_n}_{L^1}+\norm{\abs{u_n - \Pi_{\manifold{N}} (u_n)} \abs{\Deriv u_n}^2}_{L^1}\right)
\end{equation*}
and since \(u_n\) satisfies the Ginzburg--Landau equation, we have by our assumption \eqref{eq_Vohyeajee6eig6ohvoothamu}, 
\[
\norm{\Deriv \Pi_{\manifold{N}}(u_n) [\Delta u_n]}_{L^1 (\Omega \cap B_\rho (a))} \le \frac{1}{\varepsilon_n^2}\norm{\Deriv \Pi_{\manifold{N}}(u_n)[\nabla F (u_n)]}_{L^1(\Omega \cap B_\rho(a))} \underset{n\to\infty}{\longrightarrow}0.
\]
We have also \(\norm{\abs{u_n - \Pi_{\manifold{N}} (u_n)} \abs{\Deriv u_n}^2}_{L^1}\to 0\) by the asumption \eqref{eq:aavdaeoihvqihf} and by boundedess of \((\abs{\Deriv u_n})_{n\in\Nset}\) in \(L^2 (\Omega \cap B_\rho (a))\). Hence
\begin{equation}
\label{eq_convergenHone}
\norm{g_n}_{L^1(\Omega\cap B_\rho(a),\Rset^\nu)}\underset{n\to\infty}{\longrightarrow} 0.
\end{equation}

Since \(\Deriv\Pi_\manifold{N}(v_n)\) is the orthogonal projection on \(T_{v_n}\manifold{N}\), the conclusion follows from \eqref{weakPalaisSmale}, \eqref{eq_convergenLone}, \eqref{eq_convergenHone} and the result about weak limits of Palais-Smale sequences for the harmonic maps equation in \cite{Bethuel_1993} (see also \cite{Freire_Muller_Struwe_1998} and \cite{Riviere_2007}).
\end{proof}

\subsection{Higher-order convergence of solutions\label{higherOrderSection}}
Under regularity assumptions on the boundary, we improve the convergence away from singularities. 

\begin{theorem}
\label{prop_strongconvC0}
Let \(\Omega\) be a bounded open set with \(\mathcal{C}^2\) boundary and assume that \(F \in \mathcal{C}^1 (\Rset^\nu, [0,+\infty))\) satisfies \(F^{-1}(\{0\}) = \manifold{N}\) and \eqref{hyp21}. Let also \(g\in \mathcal{C}^2(\partial \Omega,\manifold{N})\), \((\varepsilon_n)_{n \in \Nset}\) be a sequence in \((0,+\infty)\) converging to \(0\) and \((u_n)_{n \in \Nset}\) be a sequence of solutions to \eqref{eq_GinzburgLandau} with \(u_n\in\mathcal{C}^2(\Bar{\Omega},\Rset^\nu)\) and \(u_n{}_{\vert\partial\Omega}=g\). If \(F \in \mathcal{C}^4 (\manifold{N}_\delta)\) for some \(\delta\in (0,\delta_\manifold{N})\), and if for some \(a\in\Bar{\Omega}\) and \(\rho\in(0,+\infty)\), we have
\begin{enumerate}[i)]
\item \((u_n)_{n\in\Nset}\) converges to some \(\manifold{N}\)-valued harmonic map \(u_*\) in \(W^{1,2}(\Omega\cap B_\rho (a),\Rset^\nu)\),
\item \(\lim_{n\to\infty}\norm{\dist (u_n, \manifold{N})}_{L^\infty(\Omega\cap B_\rho (a))}=0\),
\item \(\lim_{n \to \infty} \int_{\Omega\cap B_\rho (a)} \frac{F (u_n)}{\varepsilon_n^2} = 0\),
\end{enumerate}
then \((u_n)_{n \in \Nset}\) is bounded in \(W^{2, p}_{\mathrm{loc}} (\Bar{\Omega} \cap B_{r} (a))\)  for all \(p\in[1,+\infty)\) and \(r\in(0,\rho)\).
\end{theorem}

In particular, it follows by the Morrey--Sobolev embedding that \((u_n)_{n \in \Nset}\) converges to \(u_*\) in \(C^{1, \alpha} (\Bar{\Omega} \cap B_{\rho/2} (a))\) for all \(0<\alpha <1\).

The first tool to prove \cref{prop_strongconvC0} is the following proposition that was proved in \cite{Contreras_Lamy_Rodiac_2018} in dimension \(n\geq 3\) and whose proof is the same for \(n=2\).
It relies on the fact that when \(\dist(u_n,\manifold{N})\) is small, a B\"ochner-type formula holds: \(-\Delta e_\varepsilon(u_n) \leq Ce_\varepsilon(u_n)^2\) if \(u_n\) is a solution of \eqref{eq_GinzburgLandau} and where \(e_\varepsilon (u)= \frac{\abs{\Deriv u}^2}{2}+\frac{F(u)}{\varepsilon^2}\) and on boundary elliptic estimates on the gradient.

\begin{proposition}[\cite{Contreras_Lamy_Rodiac_2018}*{Proposition 3.1}]
\label{smallenergyestimate}
Let \(\Omega\) be a bounded open set with \(\mathcal{C}^2\) boundary, let \(g \in \mathcal{C}^2(\partial \Omega,\manifold{N})\), and let \(F \in \mathcal{C}^1 (\Rset^\nu, [0,+\infty))\) such that \(F^{-1}(\{0\}) = \manifold{N}\), \(F\in\mathcal{C}^3(\manifold{N}_\delta)\) for some \(\delta\in (0,\delta_\manifold{N})\) and \(F\) satisfies \eqref{hyp21}. There exist \(\varepsilon_0,\eta_0\in(0,+\infty)\) and \(C=C(F,\Omega,g)\in(0,+\infty)\) such that for every \(\e\in(0,\varepsilon_0)\), \(\rho\in(0,1)\) and \(a \in \Bar{\Omega}\), if \(u\in \mathcal{C}^2(\Bar{\Omega},\Rset^\nu)\) is a solution of \eqref{eq_GinzburgLandau} with \(\tr_{\partial \Omega} u=g \), \(\norm{\dist (u, \manifold{N})}_{L^\infty(\Omega \cap B_{\rho}(a))}<\delta_{\manifold{N}} \) and
\begin{equation}
\label{eq_small_energy}
E \defeq \int_{\Omega \cap B_{\rho}(a)}\frac{\abs{\Deriv u}^2}{2}+\frac{F(u)}{\varepsilon^2}  \leq \eta_0,
\end{equation}
then
\begin{equation}
\rho^2\sup_{B_{\rho/2} (a)}  \left( \frac{\abs{\Deriv u}^2}{2}+\frac{F(u)}{\varepsilon^2} \right) \leq C(E+\rho^2).
\end{equation}
\end{proposition}

\begin{proof}[Proof of \cref{prop_strongconvC0}]
By a covering argument, we can restrict our attention to the case \(r=\frac{\rho}{4}\) with \(\rho>0\) sufficiently small so that
\[
\int_{\Omega \cap B_{\rho}(a) } \frac{\abs{\Deriv u_*}^2}{2} \le \eta_0/2,
\]
with \(\eta_0\) given by \cref{smallenergyestimate},
and thus when \(n \in \Nset\) is large enough
\[
\int_{\Omega \cap B_{\rho}(a) } \frac{\abs{\Deriv u_n}^2}{2}+\frac{F(u_n)}{\varepsilon_n^2} 
\le \eta_0.
\]
It follows then, from  \cref{smallenergyestimate}, that 
\begin{equation}\label{dunBnd}
\sup_{n \in \Nset} \norm{\Deriv u_n}_{L^\infty (B_{\rho/2} (a))} < + \infty.
\end{equation}

Let \(Q(y)\defeq \dist_{\manifold{N}}(y,\manifold{N})^2\). A direct computation shows that
\begin{equation}
\label{eq_boh6zae3unai5DooPa2aa5ah}
\Delta (Q(u_n))=\Deriv Q(u_n)[\Delta u_n]+
\sum_{i = 1}^2 \Deriv^2 Q(u_n)[\partial_i u_n, \partial_iu_n].
\end{equation}
Since \(u_n\) satisfies the Ginzburg--Landau equation \eqref{eq_GinzburgLandau} and by \eqref{hyp21}, we have
\begin{equation*}
\Deriv Q(u_n)[\Delta u_n]=\Deriv Q(u_n)\Big[\frac{\nabla F(u_n)}{\varepsilon_n^2}\Big]=2\frac{\nabla F(u_n)}{\varepsilon_n^2} \cdot (u_n-\Pi_{\manifold{N}}(u_n))\geq \frac{2m_F}{\varepsilon_n^2}\dist(u_n,\manifold{N})^2.
\end{equation*}
Moreover, by the computation of the second derivatives of the squared distance given in \cref{remarkDistanceSquare}, using \eqref{secondProjectionEstimate} and the inequality \(\frac{1}{\sqrt{1-x}}\leq 1+x\) on \((0,\frac 12)\), we have for every \(z\in\manifold{N}_{\delta_{\manifold{N}}/2}\) and \(v\in\Rset^\nu\),
\[
\Deriv^2 Q(z)[v,v]=2\abs{v}^2-2\Deriv\Pi_{\manifold{N}}(z)[v]\cdot v\geq 2\abs{v}^2-\frac{2\abs{v}^2}{\sqrt{1-\frac{\dist(z,\manifold{N})}{\delta_\manifold{N}}}}\geq -\frac{2\dist(z,\manifold{N})\abs{v}^2}{\delta_\manifold{N}}.
\]
Hence, by \eqref{dunBnd}, \eqref{eq_boh6zae3unai5DooPa2aa5ah} and the two preceding estimates, we have for \(n\) large enough, in \(B_{\rho/2}(a)\),
\[
\Delta (Q(u_n))\geq \frac{2m_F}{\varepsilon_n^2}\dist(u_n,\manifold{N})^2-\frac{2\dist(z,\manifold{N})\abs{\Deriv u_n}^2}{\delta_\manifold{N}}
\geq
\frac{\Cl{cst_43k0ofm}}{\varepsilon_n^2}Q(u_n)-\Cl{cst9032jieon}\sqrt{Q(u_n)}.
\]

We have thus proved that the function
\(Q \compose u_n\) satisfies for \(n\) large enough
\begin{equation}
\label{maxPrincip}
\left\{
\begin{aligned}
-\varepsilon_n^2\Delta (Q \compose u_n)+\Cr{cst_43k0ofm}\, Q \compose u_n &\leq \Cr{cst9032jieon}\varepsilon_n^2\sqrt{Q\compose u_n}  &&\text{ in } B_{\rho/2}(a)\cap \Omega, \\
Q \compose u_n &=0 &&\text{ on } B_{\rho/2}(a) \cap \partial \Omega,
\end{aligned}
\right.
\end{equation}
where the boundary condition holds because \(u_n \in \manifold{N}\) on \(\partial \Omega\). As in \cite{Nguyen_Zarnescu_2013}*{Lemma 6 and Lemma 7}, we  deduce from the maximum principle that 
\[
Q\compose u_n=\dist(u_n,\manifold{N})^2\leq \Cl{cst_4903fjio90efwi}\varepsilon_n^4\quad\text{in \(B_{\rho/4}(a)\cap \Omega\)}.
\]
Since \(\abs{\nabla F}^2=0\) on \(\manifold{N}\), by minimality, we have also \(\Deriv(\abs{\nabla F}^2)=0\) on \(\manifold{N}\); as \(F\in\mathcal{C}^3(\manifold{N}_\delta)\), this means that there is a constant \(\Cl{cst_0493jewir}\in(0,+\infty)\) with \(\abs{\nabla F(z)}^2\leq\Cr{cst_0493jewir}\dist(z,\manifold{N})^2\) for every \(z\in\manifold{N}_{\delta_\manifold{N}}\). Hence,
\begin{equation}
\abs{\Delta u_n}=\frac{\abs{\nabla F(u_n)}}{\varepsilon_n^2}\leq\sqrt{\Cr{cst_4903fjio90efwi}\Cr{cst_0493jewir}}\quad \text{in } B_{\rho/4}(a)\cap \Omega.
\end{equation}
By elliptic estimates we obtain that \((u_n)_{n \in \Nset}\) is bounded in \(W^{2,p}_{\textrm{loc}}(B_{\rho/4}(a)\cap \Bar{\Omega})\) for every \(p\in [1,+\infty)\).
\end{proof}

The \(\mathcal{C}^{1,\alpha}\) convergence is the best we can hope for if we consider convergence up to the boundary, since if we had \(\mathcal{C}^2\) convergence up to the boundary we would have \(\Delta u_*=0\) on the boundary which is incompatible with \(-\Delta u_*=B_{u_*}(\nabla u_*,\nabla u_*)\), where \(B_u*\) is the second fundamental form of \(\mathcal{N}\) at \(u_*\), see \cite{Bethuel_Brezis_Helein_1993}*{Remark 1} when \(\mathcal{N}=\mathbb{S}^1\).
However it is natural to address the question of higher convergence in the interior of \(\Omega\) away from the singularities.
Since this relies on a bootstrap argument such a result is not easy to obtain for a general potential \(F\) and should be rather addressed for specific \(F\).
We refer to \cite{Bethuel_Brezis_Helein_1994} and \cite{Nguyen_Zarnescu_2013} for results in this direction in the Ginzburg--Landau and Landau--de Gennes models.

\end{document}